\documentclass{amsart}
\usepackage{amsmath,amscd,xypic,amssymb,combelow,color,enumitem,graphicx,float,comment,tikz,listings,array,caption,enumitem}
\usepackage[T1]{fontenc}
\makeatletter
\def\squiggly{\bgroup \markoverwith{\textcolor{black}{\lower3.5\p@\hbox{\sixly \char58}}}\ULon}

\usepackage{xr-hyper}
\usepackage[
pdftex,
bookmarks=false,
colorlinks=true,
debug=true,
pdfnewwindow=true]{hyperref}

\definecolor{codegreen}{rgb}{0,0.6,0}
\definecolor{codegray}{rgb}{0.5,0.5,0.5}
\definecolor{codepurple}{rgb}{0.58,0,0.82}
\definecolor{backcolour}{rgb}{0.95,0.95,0.92}

\lstdefinestyle{mystyle}{
  backgroundcolor=\color{backcolour}, commentstyle=\color{codegreen},
  keywordstyle=\color{magenta},
  numberstyle=\tiny\color{codegray},
  stringstyle=\color{codepurple},
  basicstyle=\ttfamily\footnotesize,
  breakatwhitespace=false,
  breaklines=true,
  captionpos=b,
  keepspaces=true,
  numbers=left,
  numbersep=5pt,
  showspaces=false,
  showstringspaces=false,
  showtabs=false,
  tabsize=2
}

\makeatletter
  \@addtoreset{equation}{section}
\makeatother

\newtheorem{theorem}[subsection]{Theorem}

\newtheorem{proposition}[subsection]{Proposition}
\newtheorem{lemma}[subsection]{Lemma}
\newtheorem{corollary}[subsection]{Corollary}
\newtheorem{conjecture}[subsection]{Conjecture}
\newtheorem{definition}[subsection]{Definition}

\theoremstyle{remark}
\newtheorem{claim}[subsection]{Claim}
\newtheorem{example}[subsection]{Example}
\newtheorem{remark}[subsection]{Remark}

\def\fa{{\mathfrak{a}}}
\def\fg{{\mathfrak{g}}}
\def\fh{{\mathfrak{h}}}
\def\fn{{\mathfrak{n}}}

\def\BC{{\mathbb{C}}}

\def\BZ{{\mathbb{Z}}}

\def\slaws{\text{standard Lyndon words}}
\def\aslaw{\text{affine standard Lyndon word}}
\def\aslaws{\text{affine standard Lyndon words}}
\def\irrchain{\text{irreducible chain}}

\def\irrchains{\text{irreducible chains}}
\def\SL{\mathrm{SL}}

\def\rk{\mathrm{rank}}
\def\wI{\widehat{I}}
\def\wQ{\widehat{Q}}
\def\wDelta{\widehat{\Delta}}
\def\imx{\wDelta^{+,\mathrm{imx}}}
\def\sb{\mathsf{b}}
\def\spanset{\mathcal{S}}

\newcommand{\ub}[2]{\underbrace{\SL_{#1} (\delta)}_{#2 \text{ times}}}

\def\hgt{\text{ht}}

\newcommand\iso{\,\vphantom{j^{X^2}}\smash{\overset{\sim}{\vphantom{\rule{0pt}{0.20em}}\smash{\longrightarrow}}}\,}

\DeclareMathOperator{\spn}{span}
\DeclareMathOperator{\chain}{ch}
\DeclareMathOperator{\rank}{rank}
\DeclareMathOperator{\pr}{Pr}

\def\LL{{\mathrm{L}}}

\def\sb{{\mathsf{b}}}

\def\re{\mathrm{re}}
\def\im{\mathrm{im}}
\def\ext{\mathrm{ext}}

\begin{document}

\title[Chains of Affine standard Lyndon words]
      {\Large{\textbf{Chains of Affine standard Lyndon words}}}
	
\author[Corbet Elkins and Alexander Tsymbaliuk]{Corbet Elkins and Alexander Tsymbaliuk}

\address{C.E.: Purdue University, Department of Mathematics, West Lafayette, IN, USA}
\email{cdelkins@purdue.edu}

\address{A.T.: Purdue University, Department of Mathematics, West Lafayette, IN, USA}
\email{sashikts@gmail.com}

\begin{abstract}
In this note, we establish the periodicity of chains of affine standard Lyndon words in all types and determine
tight bounds on that periodicity, greatly generalizing the $A$-type results of~\cite{AT}. Our approach crucially
utilizes the convexity and monotonicity of~\cite{ET} together with the new idea to consider the polarization of
the root system given by increasing and decreasing chains.
\end{abstract}

\maketitle


\section{Introduction}


\subsection{Summary}\label{ssec:summary}
\

The free Lie algebras generated by a finite set $\{e_i\}_{i\in I}$ are known to have bases parametrized by
Lyndon words in $I$. This was generalized to finitely generated Lie algebras $\fa$ in~\cite{LR}, whereas the basis
is parametrized by standard Lyndon words. The key application of~\cite{LR} was to the positive subalgebras $\fn^+$
of simple finite dimensional $\fg$. In that context, one easily derives a natural bijection (\cite{LR})
\begin{equation}\label{eqn:1-to-1 intro}
  \ell \colon \Delta^+ \iso \{\text{standard Lyndon words}\}.
\end{equation}
A decade later, \cite{L} established an iterative Leclerc algorithm for~\eqref{eqn:1-to-1 intro}.
Moreover, the induced total order on $\Delta^+$ is convex due to~\cite{R}. This played the key role in~\cite{L},
where it was shown how to construct a basis of the corresponding positive half $U_q(\fn^+)$ of Drinfeld-Jimbo
quantum group through the word's combinatorics.

In the recent work of Avdieiev and the second author~\cite{AT}, the generalization to (untwisted) affine Lie algebras was initiated.
Let $\widehat{\fg}$ be the affinization of $\fg$, whose Dynkin diagram is obtained by extending the Dynkin diagram of $\fg$ with
a vertex~$0$. Thus, on the combinatorial side, we consider the alphabet $\wI=I\sqcup \{0\}$. The corresponding positive
subalgebra $\widehat{\fn}^+\subset \widehat{\fg}$ still admits the root space decomposition
$\widehat{\fn}^+=\bigoplus_{\alpha\in \wDelta^+} \widehat{\fn}^+_{\alpha}$ with $\wDelta^+=\{\mathrm{positive\ affine\ roots}\}$.
The key difference is that:
\begin{equation*}
  \dim \widehat{\fn}^+_{\alpha}=1 \quad \forall\, \alpha\in \wDelta^{+,\re} \,, \qquad
  \dim \widehat{\fn}^+_{\alpha}=|I| \quad \forall\, \alpha\in \wDelta^{+,\im}.
\end{equation*}
Here, $\wDelta=\wDelta^{+,\re}\sqcup \wDelta^{+,\im}$ is the decomposition into real and imaginary affine roots, with
$\wDelta^{+,\im}=\{k\delta | k\geq 1\}$. It is therefore natural to consider an extended set $\widehat{\Delta}^{+,\ext}$
of~\eqref{eq:extended-affine-roots}. Then, the degree reasoning as in~\cite{LR} provides a natural analogue of~\eqref{eqn:1-to-1 intro}:
\begin{equation}\label{eqn:affine 1-to-1 intro}
  \SL \colon \wDelta^{+,\ext} \iso \{ \text{affine standard Lyndon words} \}.
\end{equation}
In~\cite{AT}, we established a generalized Leclerc algorithm describing this bijection and used it to compute affine SL-words
in type $A$ with any order on $\wI$. The resulting explicit formulas illustrated a stunning \emph{periodicity} for $\aslaws$
(expressing all via $\{\SL(\alpha)\colon |\alpha|<|\delta|\}$) as well as \emph{pre-convexity} and \emph{monotonicity}:
\begin{align}
  & \alpha<\alpha+\beta<\beta \quad \mathrm{or} \quad \beta<\alpha+\beta<\alpha
  \qquad \forall\ \alpha,\beta,\alpha+\beta\in \wDelta^{+,\re},
  \label{eq:pre-convex}\\
  & \alpha<\alpha+\delta<\alpha+2\delta<\cdots \quad \mathrm{or} \quad \alpha>\alpha+\delta>\alpha+2\delta>\cdots
  \qquad \forall\ \alpha\in \wDelta^{+,\re}.
  \label{eq:monot}
\end{align}

Our work is devoted to the further study of affine SL-words in all types, where explicit formulas are deemed
inefficient (especially for exceptional types, like $E_8$). In our earlier work~\cite{ET}, we established~\eqref{eq:monot}
and generalization of~\eqref{eq:pre-convex} in a conceptual way without having explicit formulas for affine SL-words.
One of the key insights underlying~\cite{ET} was the consideration of the complete flags~\eqref{eq:flag}.

In the present work, we reverse the logic of~\cite{AT} by using the results of~\cite{ET} to establish the
periodicity of chains of real affine SL-words. The new idea that is key to our present analysis is to consider
the polarization of the finite root system induced by the sets of increasing and decreasing chains.
For decreasing chains, the second option of~\eqref{eq:monot}, the periodicity is established in
Theorem~\ref{thm:dec.modulo.s}, with the tight bounds on the resulting periodicity function obtained
in Corollary~\ref{cor:max.periodicities} and Proposition~\ref{lem:precise-bound}.
For increasing chains, the first option of~\eqref{eq:monot}, the periodicity is established in
Theorem~\ref{thm:incr.modulo.c}, with the tight bounds on the resulting periodicity function obtained
in Corollary~\ref{cor:max.period-increasing} and Proposition~\ref{prop:c.bounds}.
In contrast to~\cite{AT}, these chain periodicities do not always start from the very first elements in the chains.
We note that these periodicity patterns allow us to derive all affine SL-words from a
finite subset which can be easily computed through our code in the Appendix.

We also note that while the usage of standard and costandard factorizations of SL-words for finite setup was quite
interchangeable, we crucially use the standard factorization for decreasing chains and costandard factorization for increasing chains,
see Examples~\ref{ex:costand-decr-fails} and~\ref{ex:stand-incr-fails} highlighting the importance of this choice. At the same time,
one can also ``break down'' any sufficiently long real affine SL-words into ``simpler blocks'' through the costandard factorization,
see Remarks~\ref{rem:cost.fact.structure} and~\ref{rem:cost.fact.structure-incr}.


\subsection{Outline}\label{ssec:outline}
\

\noindent
The structure of the present paper is the following:

\smallskip
\noindent
$\bullet$
In Section~\ref{sec:setup}, we recall the basic results about affine standard Lyndon words.
This exposition, apart from Lemmas~\ref{lemma:lyndon.subword}--\ref{lemma:stronger.no.splitting},
is mostly based on works~\cite{AT,ET}.

\smallskip
\noindent
$\bullet$
In Section~\ref{sec:earlier-results}, we summarize the key Convexity and Monotonicity properties of~\cite{ET},
along with some immediate corollaries of those that will be needed later.

\smallskip
\noindent
$\bullet$
In Section~\ref{sec:chains}, we introduce the polarization of the corresponding finite root system induced by
increasing and decreasing chains, the key new idea of the present note. We establish some basic properties of
increasing and decreasing chains. We also introduce the Connectivity property, crucial for the study of increasing chains.

\smallskip
\noindent
$\bullet$
In Section~\ref{sec:dec}, we thoroughly study decreasing chains, establishing the periodicity pattern for those and obtaining
tight bounds on the periodicity function.

\smallskip
\noindent
$\bullet$
In Section~\ref{sec:incr}, we present a similar (but more tedious) analysis of increasing chains,
establishing the periodicity patterns and obtaining tight bounds on their periodicity.

\smallskip
\noindent
$\bullet$
In Appendix~\ref{sec:app_code}, we present the computer code that we heavily used to find the correct patterns
for affine standard Lyndon words, verify Conjecture~\ref{conj:connectivity} for all exceptional types, and derive most
of the examples presented through the note.


\subsection{Acknowledgment}\label{ssec:acknowl}
\

A.T.\ is deeply indebted to Andrei Negu\c{t} for numerous inspiring  discussions over the years;
to IHES for the hospitality and wonderful working conditions in the summer 2025, where key ideas
for this project were developed. The work of both authors was partially supported by NSF Grant DMS-$2302661$.
The work of C.E.\ was supported by the Purdue Joel Spira Undergraduate Summer Research Award.


\section{Setup, Notations, and Basic Properties}\label{sec:setup}

In this section, we recall the basics of standard Lyndon words~\cite{LR,L}, with the main emphasize placed on
affine Lie algebras. While this section mostly follows~\cite{AT} and~\cite[\S2]{ET}, it also contains new
Lemmas~\ref{lemma:lyndon.subword}--\ref{lemma:stronger.no.splitting} that will be used later.


\subsection{Lyndon words}\label{ssec:L-words}
\

Let $I$ be a finite ordered alphabet, and let $I^*$ be the set of all finite length words in the alphabet $I$.
For $u = [i_1\ldots i_k] \in I^*$, we define its length by $|u| = k$. Moreover, we consider the lexicographical
order on $I^*$ defined as follows:
$$
    [i_1\ldots i_k] < [j_1\ldots j_l]\quad \text{if } \begin{cases}
        i_1 = j_1,\ldots,i_a=j_a,i_{a+1}<j_{a+1} \text{ for some }a \geq 0\\
          \ \  \text{or}\\
        i_1=j_1,\ldots,i_k=j_k \text{ and } k < l
    \end{cases}.
$$

\begin{definition}\label{def:lyndon}
A word $\ell=[i_1\dots i_k]$ is called \textbf{Lyndon} if it is smaller than all of its cyclic permutations:
$[i_1 \dots i_{a-1} i_a \dots i_k] < [i_a \dots i_k i_1 \dots i_{a-1}]$ for all $a \in \{2,\dots,k\}$.
\end{definition}

For a word $w = [i_1 \dots i_k]\in I^*$, the subwords:
\begin{equation*}
  w_{a|} =  [i_1 \dots i_a] \qquad \text{and} \qquad w_{|a} = [i_{k-a+1} \dots i_k]
\end{equation*}
with $0\leq a\leq k$ are usually called a \textbf{prefix} and a \textbf{suffix} of $w$, respectively.
We call such a prefix or a suffix proper if $0<a<k$. It is well-known that Definition~\ref{def:lyndon}
is equivalent to the following one:

\begin{definition}\label{def:lyndon-2}
A word $w$ is Lyndon if it is smaller than all of its proper suffixes:
$w < w_{| a}$ for all $0 < a < |w|$.
\end{definition}

As a corollary, we record the following basic property:

\begin{lemma}\label{lemma:lyndon}
If $\ell_1 < \ell_2$ are Lyndon, then $\ell_1\ell_2$ is also Lyndon, and so $\ell_1\ell_2 < \ell_2\ell_1$.
\end{lemma}

Let us now recall several basic facts from the theory of Lyndon words (cf.~\cite{Lo}).

\begin{proposition}\label{prop:cost.factor}
Any Lyndon word $\ell$ with $|\ell|>1$ has a \textbf{costandard factorization} $\ell = \ell^l \ell^r$ defined by the property
that $\ell^r$ is the longest proper suffix of $\ell$ which is also a Lyndon word. Then, $\ell^l$ is also a Lyndon word.
\end{proposition}

We have an analogous factorization with the longest proper Lyndon prefix:

\begin{proposition}\label{prop:stand.factor}
Any Lyndon word $\ell$ with $|\ell|>1$ has a \textbf{standard factorization} $\ell = \ell^{ls} \ell^{rs}$ defined
by the property that $\ell^{ls}$ is the longest proper prefix of $\ell$ which is also a Lyndon word.
Then, $\ell^{rs}$ is also a Lyndon word.
\end{proposition}

Let $\LL$ be the set of all Lyndon words. Any word can be canonically built from~$\LL$:

\begin{proposition}\label{prop:canon.factor}
Any word $w\in I^*$ has a unique  \textbf{canonical factorization}:
\begin{equation*}
  w = \ell_1 \dots \ell_k  \quad \text{with} \quad \ell_1 \geq \dots \geq \ell_k \in \LL.
\end{equation*}
\end{proposition}

The following result of~\cite{M} is often used when comparing words:

\begin{lemma}\label{lem:Melancon}
For two canonical factorizations $w = \ell_1 \dots \ell_k$ and $w' = \ell'_1 \dots \ell'_{k'}$ such that
$w$ is not a prefix of $w'$:
  $$w<w' \iff \exists 1\leq a\leq \min\{k,k'\} \colon \ell_{1}=\ell'_1,\ldots, \ell_{a-1}=\ell'_{a-1}, \ \ \mathrm{and} \ \ \ell_a<\ell'_a.$$
\end{lemma}


\subsection{Bracketing and standard Lyndon words}\label{ssec:SL-words}
\

Let $\fa$ be a Lie algebra generated by a finite set $\{e_i\}_{i \in I}$ labeled by the alphabet~$I$.

\begin{definition}\label{bracketing}
The \textbf{bracketing} of $\ell\in \LL$ is given inductively by:
\begin{itemize}[leftmargin=0.7cm]

\item[$\bullet$]
$\sb[i] = e_i \in \fa$ for $i \in I$,

\item[$\bullet$]
$\sb[\ell] = [\sb[\ell^l],\sb[\ell^r]] \in \fa$ if $|\ell|>1$.

\end{itemize}
\end{definition}

The major importance of this definition is due to the following result of Lyndon:

\begin{theorem}\label{thm:Lyndon.theorem}(\cite[Theorem 5.3.1]{Lo})
If $\fa$ is a free Lie algebra in the generators $\{e_i\}_{i\in I}$, then the set
$\big\{\sb[\ell] \,|\, \ell\in \LL \big\}$ provides a basis of $\fa$.
\end{theorem}

A generalization of Theorem~\ref{thm:Lyndon.theorem} to Lie algebras $\fa$ generated by $\{e_i\}_{i\in I}$
was provided in~\cite{LR}. To state the result, define ${_we}\in U(\fa)$ for any word $w\in I^*$:
\begin{equation}\label{eqn:word}
  _{[i_1 \dots i_k]}e = e_{i_1} \dots e_{i_k} \in U(\fa).
\end{equation}
Consider the following new order on $I^*$:
\begin{equation*}
    v\succeq w \iff |v|<|w|  \ \  \mathrm{or} \ \ |v|=|w| \ \mathrm{and} \ v \geq w.
\end{equation*}
The following definition is due to \cite{LR}:

\begin{definition}\label{def:standard}
(a) A word $w$ is called \textbf{standard} if $_we\in U(\fa)$ cannot be expressed as a linear combination of
$_ve$ for various $v \succ w$, with $_we$ as in~\eqref{eqn:word}.

\medskip
\noindent
(b) A Lyndon word $\ell$ is called \textbf{standard Lyndon} if $\sb[\ell]\in \fa$ cannot be expressed as
a linear combination of $\sb[v]$ for various Lyndon words $v \succ \ell$.
\end{definition}

The following result is nontrivial and justifies the above terminology:

\begin{proposition}\label{prop:standard}(\cite{LR})
A Lyndon word is standard iff it is standard Lyndon.
\end{proposition}

We shall use $\SL$ to denote the set of all standard Lyndon words. The major importance of this definition
is due to the following result of Lalonde-Ram:

\begin{theorem}\label{thm:standard Lyndon theorem}(\cite[Theorem 2.1]{LR})
For any Lie algebra $\fa$ generated by a finite collection $\{e_i\}_{i\in I}$, the set
$\big\{\sb[\ell] \,|\, \ell\in \SL \big\}$ provides a basis of $\fa$.
\end{theorem}


\subsection{Affine Lie algebras}\label{ssec:affineLie}
\

This note is devoted to the study of the above concepts in the context of (untwisted) affine Lie algebras,
whose definition is reminded below. Let $\fg$ be a simple finite dimensional Lie algebra,
$\{\alpha_i\}_{i\in I}$ be the simple roots, and $\theta\in \Delta^+$ be the highest root. We define
$\wI = I \sqcup \{0\}$. Consider the affine root lattice $\wQ=Q \times \BZ$ with the generators
$\{(\alpha_i,0)\}_{i\in I}$ and $\alpha_0:=(-\theta,1)$. We endow $\wQ$ with the symmetric pairing
\begin{equation*}
  \big((\alpha,n),(\beta,m)\big)=(\alpha,\beta) \qquad \forall\ \alpha,\beta\in Q \,,\, n,m \in \BZ.
\end{equation*}
This leads to the affine Cartan matrix $(a_{ij})_{i,j\in \wI}$ and the \textbf{affine Lie algebra} $\widehat{\fg}$,
generated by $\{e_i,f_i,h_i\}_{i\in \wI}$ subject to Chevalley-Serre defining relations.
The associated affine root system $\wDelta=\wDelta^+ \sqcup \wDelta^-$ has the following explicit description:
\begin{equation}\label{eq:affine-roots}
  \wDelta^+ = \big\{ \Delta^+ \times \BZ_{\geq 0} \big\}
  \sqcup \big\{ 0 \times \BZ_{>0} \big\}
  \sqcup \big\{ \Delta^- \times \BZ_{>0} \big\}.
\end{equation}
Here, $\delta=\alpha_0+\theta=(0,1) \in Q \times \BZ$ is the \textbf{minimal imaginary root} of the
affine root system $\wDelta$. With this notation, we have the following root space decomposition:
\begin{equation}\label{eq:aff.root.decomp}
  \widehat{\fg}=\widehat{\fh} \oplus \bigoplus_{\alpha \in \wDelta} \widehat{\fg}_{\alpha} ,
  \qquad \widehat{\fh}\subset \widehat{\fg} -\mathrm{Cartan\ subalgebra}.
\end{equation}

Let us now recall another realization of $\widehat{\fg}$. To this end, consider the Lie algebra
\begin{equation*}
\begin{split}
  & \widetilde{\fg}=\fg\otimes \BC[t,t^{-1}]\oplus \BC\cdot \mathsf{c}
    \quad \mathrm{with\ a\ Lie\ bracket\ given\ by}\\
  & [x\otimes t^n,y\otimes t^m]=[x,y]\otimes t^{n+m}+n\delta_{n,-m}(x,y)\cdot \mathsf{c}
    \quad \mathrm{and} \quad [\mathsf{c},x\otimes t^n]=0,
\end{split}
\end{equation*}
where $x,y\in \fg$, $m,n\in \BZ$, and $(\cdot,\cdot)\colon \fg\times\fg\to \BC$ is
a non-degenerate invariant pairing.

The rich theory of affine Lie algebras is mainly based on the following key result:

\begin{claim}
There exists a Lie algebra isomorphism:
\begin{equation*}
  \widehat{\fg}\ \iso\ \widetilde{\fg}
\end{equation*}
determined on the generators by the following formulas:
\begin{align*}
  & e_i \mapsto e_i \otimes t^0 & & f_i \mapsto f_i \otimes t^0 & & h_i \mapsto h_i \otimes t^0 \qquad \forall\, i\in I \\
  & e_0 \mapsto e_{-\theta} \otimes t^1 & & f_0 \mapsto e_\theta \otimes t^{-1} &
  & h_0 \mapsto [e_{-\theta},e_\theta]\otimes t^0 + (e_{-\theta},e_{\theta}) \mathsf{c}.
\end{align*}
\end{claim}

In view of this result, we can explicitly describe the root subspaces from~\eqref{eq:aff.root.decomp}:
\begin{align*}
  & \widehat{\fg}_{(\alpha,k)}=\fg_\alpha\otimes t^k \quad \forall\
    (\alpha,k)\in \wDelta^{+,\re}:=\big\{\Delta^+ \times \BZ_{\geq 0} \big\} \sqcup \big\{ \Delta^- \times \BZ_{>0} \big\},\\
  & \widehat{\fg}_{k\delta}=\fh \otimes t^k \quad \mathrm{for} \
    k\delta\in \wDelta^{+,\im}:=\big\{ 0 \times \BZ_{>0} \big\},
\end{align*}
where $\fh$ is a Cartan subalgebra of $\fg$ and $\fg=\fh\oplus \bigoplus_{\alpha\in \Delta} \fg_\alpha$ is the standard
root space decomposition. As $\dim(\fg_\alpha)=1$ for $\alpha\in \Delta$, $\dim(\fh)=\rk(\fg)=|I|$, we get:
\begin{equation}\label{eq:aff-dim}
  \dim(\widehat{\fg}_\alpha)=1 \quad \forall\ \alpha\in \wDelta^{+,\re} \,, \qquad
  \dim(\widehat{\fg}_\alpha)=|I| \quad \forall\ \alpha\in \wDelta^{+,\im}.
\end{equation}

In what follows, we shall always write $xt^n$ instead of $x\otimes t^n$.


\subsection{Affine standard Lyndon words}\label{ssec:aslaws}
\

While the theory of standard Lyndon words has been well developed in the context of simple $\fg$,
we shall now recall some similarities and differences in the affine setup. To this end, we consider the positive subalgebra
$\widehat{\fn}^+=\bigoplus_{\alpha\in \wDelta^+} \widehat{\fg}_{\alpha}$, generated by $\{e_i\}_{i\in \wI}$ subject to
the Serre relations for $i\ne j\in \wI$. Endowing $\wI$ with any order allows us to introduce Lyndon and
standard Lyndon words. Henceforth, we shall often use the term \textbf{$\aslaws$} in the present setup.

One of the key results of~\cite[\S3]{LR} for finite $\fg$ was the \textbf{Lalonde-Ram bijection}:
\begin{equation}\label{eqn:associated word}
  \ell \colon \Delta^+ \,\iso\, \big\{\slaws \big\}
  \quad \mathrm{with} \quad \deg \ell(\alpha) = \alpha,
\end{equation}
described explicitly in~\cite[Proposition 25]{L} by the so-called \textbf{Leclerc algorithm}. The key difference of
the affine setup is that imaginary root subspaces are higher dimensional by~\eqref{eq:aff-dim}, which also affects
all real roots if applying Leclerc algorithm.  Thus, we no longer have a bijection~\eqref{eqn:associated word} on the nose.
However, a degree reasoning similar to the one of~\cite{LR} still implies that there is a unique $\aslaw$ in each real degree
$\alpha\in \wDelta^{+,\re}$, denoted by $\SL(\alpha)$, and $\aslaws$ $\SL_1(\alpha)>\dots>\SL_{|I|}(\alpha)$ in each imaginary
degree $\alpha\in \wDelta^{+,\im}$. Following~\cite[(5.1)]{AT}, let us consider the following upgrade of~\eqref{eq:affine-roots}:
\begin{equation}\label{eq:extended-affine-roots}
  \wDelta^{+,\ext} = \wDelta^{+,\re} \cup \imx  \quad \mathrm{with} \quad
  \imx=\big\{(k\delta,r) \,|\, k \geq 1, 1 \leq r \leq |I| \big\},
\end{equation}
counting imaginary roots with multiplicities. We can thus naturally generalize~\eqref{eqn:associated word}:
\begin{equation}\label{eqn:associated word affine}
  \SL \colon \wDelta^{+,\ext} \,\iso\, \big\{\aslaws \big\},
  \quad \SL((k\delta,r)) = \SL_r(k\delta).
\end{equation}
We note that this bijection gives rise to the induced lexicographical order on $\wDelta^{+,\ext}$:
\begin{equation}\label{eq:order.ext}
    \alpha < \beta \iff \SL(\alpha) < \SL(\beta) \qquad \forall\, \alpha,\beta \in \wDelta^{+,\ext}.
\end{equation}

Furthermore, the bijection~\eqref{eqn:associated word affine} can be computed through~\cite[Proposition 3.4]{AT}:

\begin{proposition}\label{prop:generalized.Leclerc.algo} (\textbf{generalized Leclerc algorithm})
The affine standard Lyndon words (with respect to $\widehat{\fn}^+$) are determined inductively by the following rules:

\medskip
\noindent
(a) For simple roots, we have $\SL(\alpha_i)=[i]$. For other real $\alpha\in \wDelta^{+,\re}$, we have:
\begin{equation}\label{eq:generalized Leclerc}
  \SL(\alpha) =
  \max\left\{\SL_*(\gamma_1)\SL_*(\gamma_2) \,\Big|\,
   \substack{\alpha=\gamma_1+\gamma_2,\, \gamma_1,\gamma_2\in \wDelta^+\\ \SL_*(\gamma_1)<\SL_*(\gamma_2)\\
             [\sb[\SL_*(\gamma_1)],\sb[\SL_*(\gamma_2)]]\neq 0} \right\},
\end{equation}
where $\SL_*(\gamma)$ denotes $\SL(\gamma)$ for $\gamma\in \wDelta^{+,\re}$
and any of $\{\SL_k(\gamma)\}_{k=1}^{|I|}$ for $\gamma\in \wDelta^{+,\im}$.

\medskip
\noindent
(b) For imaginary $\alpha\in \wDelta^{+,\im}$, the corresponding $|I|$ $\aslaws$ $\{\SL_k(\alpha)\}_{k=1}^{|I|}$ are
the $|I|$ lexicographically largest words from the list as in the right-hand side of~\eqref{eq:generalized Leclerc}
whose standard bracketings are linearly independent.
\end{proposition}

We shall call $u=\SL_{*}(\gamma)$ real (resp.\ imaginary) if $\gamma\in \wDelta^{+,\re}$ (resp.\ $\gamma\in \wDelta^{+,\im}$).


\subsection{Properties of factorizations}\label{ssec:factorizations}
\

In this subsection, we recall some results about various factorizations.
We start with the following very useful result that will be often used in the rest of the paper:

\begin{lemma}\label{lemma:equiv.to.standard.fac}(\cite[Lemma 4.2]{ET})
Let $\ell$ be a Lyndon word and $uv$ be the standard factorization of a Lyndon word, with $|\ell| < |uv|$.
Then $\ell > u \Leftrightarrow \ell > uv$.
\end{lemma}

We note the following description of the costandard suffix (\cite[Lemma 3.1]{ET}):

\begin{lemma}\label{lemma:right.costfac.minimal}
The smallest proper suffix of a Lyndon word $\ell$ is $\ell^r$.
\end{lemma}

As a simple corollary of~\cite[Lemma 3.5]{ET}, we recall the following result:

\begin{corollary}\label{cor:left.cost.right}(\cite[Corollary~3.8]{ET})
For $\ell = \ell_1\ell_2$ with $\ell,\ell_1,\ell_2\in \LL$, we have
\begin{equation*}
   \ell_2 = \ell^r \iff \ell_1^{r} \geq \ell_2.
\end{equation*}
\end{corollary}

We shall also need the following two results (\cite[Lemmas 3.12--3.13]{ET}):

\begin{lemma}\label{lemma:seq.right.word.Lyndon}
For $u \in \LL$, consider any splitting $u=vw$ with $v\in \LL$, and let $w = w_1w_2 \ldots w_N$ be the canonical factorization.
Then $vw_1$,$vw_1w_2, \ldots ,vw_1\ldots w_N\in \LL$.
\end{lemma}

\begin{lemma}\label{lemma:seq.left.word.Lyndon}
For $u\in \LL$, consider any splitting $u=vw$ with $w\in \LL$ and the canonical factorization $v = v_1v_2\ldots v_N$.
Then $v_Nw,v_{N-1}v_Nw, \ldots ,v_1\ldots v_Nw\in \LL$.
\end{lemma}

For the standard factorizations, one has the following result (\cite[Lemma 14]{L}):

\begin{lemma}\label{lemma:leclerc.14}
Any standard Lyndon word $\ell$ can be written as
\begin{equation*}
  \ell = (\ell^{ls})^{k+1}fx,
\end{equation*}
where $k$ is a non-negative integer number, $f$ is a possibly empty proper prefix of $\ell^{ls}$, and $x$ is a single letter.
\end{lemma}

We conclude with two results that did not appear in~\cite{ET}, but will be used later.

\begin{lemma}\label{lemma:lyndon.subword}
If $w = w_1\ldots w_n$ is the canonical factorization and a Lyndon word $\ell$ is a subword of $w$,
then $\ell$ must be fully contained in one of the $w_i$'s.
\end{lemma}

\begin{proof}
Assuming the contrary, some proper prefix $u$ of $\ell$ is a suffix of $w_i$ and some proper suffix $v$ of $\ell$ is
a prefix of $w_j$ for some $i < j$. But then $w_i \leq u < \ell < v \leq w_j$ by properties of Lyndon words, which
contradicts to $w_i \geq w_{i+1} \geq \cdots \geq w_j$.
\end{proof}

\begin{lemma}\label{lemma:stronger.no.splitting}
Let a Lyndon word $u$ be contained as a subword in a Lyndon word~$\ell$.

\smallskip
\noindent
(1) If $u$ is not a suffix of $\ell$, then the standard factorization of $\ell$ does not split $u$.

\smallskip
\noindent
(2) If $u$ is not a prefix of $\ell$, then the costandard factorization of $\ell$ does not split $u$.
\end{lemma}

\begin{proof}
For (1), suppose the contrary, i.e.\ the standard factorization of $\ell$ splits $u$: $\ell^{ls} = vw,\ell^{rs} = xy$,
where $w,x,y \neq \emptyset$ and $wx=u$. Let $w = w_1\ldots w_n$ and $x = x_1\ldots x_m$ be the canonical factorizations.
Since $u=wx$ is Lyndon, we have $w < x$ which implies that $w_n \leq x_1$ by Lemma~\ref{lem:Melancon}.
Then we have $\ell^{ls}=vw < w_n\leq x_1$, hence $\ell^{ls}x_1 = vwx_1$ is Lyndon by Lemma~\ref{lemma:lyndon},
contradicting $\ell^{ls}$ being the longest Lyndon proper prefix of $\ell$.

For (2), suppose the contrary, i.e.\ the costandard factorization of $\ell$ splits $u$ that is $\ell^{l} = vw,\ell^{r} = xy$
where $v,w,x \neq \emptyset$ and $wx=u$. Let $w = w_1\ldots w_n$ and $x = x_1\ldots x_m$ be the canonical factorizations.
Since $u=wx$ is Lyndon, we again have $w<x$ which implies $w_n < x < xy = \ell^r$ but then $w_nxy=w_n\ell^r$ is Lyndon by
Lemma~\ref{lemma:lyndon}, contradicting $\ell^r$ being the longest Lyndon proper suffix of $\ell$.
\end{proof}


\section{Earlier Results}\label{sec:earlier-results}

In this section, we recall some key results of~\cite[\S4--6]{ET} that will play the crucial role in the present note.
We omit the proofs, referring the interested reader to~\cite{ET}.


\subsection{Key constructions}\label{ssec:constructions}
\

As mentioned in Subsection~\ref{ssec:aslaws}, the major difficulty of the affine setup lies in the treatment of imaginary SL-words.
To this end, one of the key invariants introduced in~\cite[(4.1)]{ET} is the complete flag of the corresponding
standard bracketings:
\begin{equation}\label{eq:flag}
\begin{split}
  & 0=\spanset_0^k \subset \spanset_1^k \subset \cdots \subset\spanset_{|I|}^k = \fh t^k \qquad \forall\, k\in \BZ_{>0},\\
  & \mathrm{with}\quad  \spanset_i^k := \spn\big\{\sb[\SL_1(k\delta)],\ldots,\sb[\SL_i(k\delta)]\big\}.
\end{split}
\end{equation}
We shall define two imaginary roots $m_k(\alpha), M_k(\alpha)$ for any $\alpha \in \wDelta^{+,\re}$ and $k\in \BZ_{>0}$
(these are equivalent to~\cite[Definition 4.14]{ET} according to~\cite[Lemmas 4.15, 4.17]{ET}).

\begin{definition}\label{def:mk}
For any $\alpha \in \wDelta^{+,\re}$ and $k\in \BZ_{>0}$, we define:
\begin{equation*}
  m_k(\alpha) = (k\delta,i)\quad \mathrm{with} \quad  h_\alpha t^k \in \spanset_{i}^k \backslash\spanset_{i-1}^k.
\end{equation*}
\end{definition}

\begin{definition}\label{def:Mk}
For any $\alpha \in \wDelta^{+,\re}$ and $k\in \BZ_{>0}$ we define:
\begin{equation*}
  M_k(\alpha) = \max \big\{\beta \in \imx \,\big|\, |\beta| = |k\delta|,\, [\sb[\SL(\beta)],\sb[\SL(\alpha)]] \neq 0 \big\}.
\end{equation*}
\end{definition}

The later definition clearly satisfies the following property (\cite[Corollary~4.16]{ET}):

\begin{corollary}\label{cor:im.span.vanish}
If $M_k(\alpha) = (k\delta,i)$ for $\alpha \in \wDelta^{+,\re}$, then $[h,\SL(\alpha)] = 0\ \forall\, h \in \spanset_{i-1}^k$.
\end{corollary}

The dependence of above constructions on the parameter $k$ is actually inessential as follows
from~\cite[Proposition 6.2, Corollary 6.4]{ET}, which are summarized below.

\begin{proposition}\label{prop:spanset.equiv}
For any $i \in \{0,1,\ldots,|I|\}$ and $k \in \BZ_{>0}$, we have $\spanset_{i}^{k+1} = \spanset_i^kt$.
\end{proposition}

\begin{corollary}\label{cor:equiv.mk.Mk}
For any $k,p \in \BZ_{>0}$ and $\alpha \in \wDelta^{+,\re}$, we have:
\begin{gather*}
  M_k(\alpha) = (k\delta,i) \iff M_p(\alpha) = (p\delta,i), \qquad
  m_k(\alpha) = (k\delta,i) \iff m_p(\alpha) = (p\delta,i).
\end{gather*}
\end{corollary}

Let us also record the behavior of above functions $m_k$ and $M_k$ on the sums of roots that are also roots
(\cite[Corollary 5.11, Lemma 5.12]{ET}):

\begin{corollary}\label{cor:max.im.rule}
Let $\alpha,\beta \in \wDelta^{+,\re}$ satisfy $\alpha + \beta \in \wDelta^{+,\re}$. If $M_k(\alpha) > M_k(\beta)$
then $M_k(\alpha) = M_k(\alpha + \beta)$, and if $M_k(\alpha) = M_k(\beta)$ then $M_k(\alpha + \beta) \leq M_k(\alpha)$.
\end{corollary}

\begin{lemma}\label{lemma:min.im.rule}
Let $\alpha,\beta \in \wDelta^{+,\re}$ satisfy $\alpha + \beta \in \wDelta^{+,\re}$. If $m_k(\alpha) < m_k(\beta)$ then
$m_k(\alpha + \beta) = m_k(\alpha)$, and if $m_k(\alpha) = m_k(\beta)$ then $m_k(\alpha + \beta) \geq m_k(\alpha)$.
\end{lemma}


\subsection{Invariance of bracketing}
\

While the bracketing of Definition~\ref{bracketing} was defined using the costandard factorization,
it is actually independent of a factorization (\cite[Proposition 4.18]{ET}).

\begin{proposition}\label{prop:general.bracketing}
For any factorization $u = u_1u_2$ with $u,u_1,u_2 \in \SL$, we have:

\medskip
\noindent
(a) $[\sb[u_1],\sb[u_2]] \neq 0$;

\medskip
\noindent
(b) if $u$ is imaginary and $u = \SL_i(k\delta)$, then $[\sb[u_1],\sb[u_2]] \in \spanset^{k}_i \backslash\spanset_{i-1}^k$.
\end{proposition}

In particular, if we were to use the standard factorization instead of the costandard one in Definition~\ref{bracketing},
we would still get exactly the same affine standard Lyndon words and the same flags $\spanset_\bullet^k$
(\cite[Remark 4.19]{ET}).
We now list several corollaries of above proposition (\cite[Corollary 4.22, Corollary 4.23, Corollary 6.9]{ET}).

\begin{corollary}\label{cor:mk.im.factor}
For any $1\leq i\leq |I|$, $k\in \BZ_{>0}$, and any splitting $\SL_i(k\delta) = uv$ with $u,v \in \SL$, we have
$m_k(\deg(u)) = (k\delta,i)=m_k(\deg(v))$.
\end{corollary}

\begin{corollary}\label{cor:imaginary.suffix.prefix}
Let $\alpha \in \wDelta^{+,\re}$ and $\SL(\alpha) = uv$ with $u,v \in \SL$.

\smallskip
\noindent
(1) If $u$ is imaginary, then $u = \SL(M_1(\alpha))$.

\smallskip
\noindent
(2) If $v$ is imaginary, then $v = \SL(M_1(\alpha))$.
\end{corollary}


\subsection{Convexity and Monotonicity}
\

In this subsection, we recall the key two results of~\cite{ET}, referred to as \textbf{Convexity}
(in analogy with the result of~\cite{R} for finite type) and \textbf{Monotonicity}. Henceforth,
we shall use the lexicographical order~\eqref{eq:order.ext} on $\wDelta^{+,\ext}$.

\begin{theorem}[Convexity]\cite[Theorem 5.4, Remark 5.8]{ET}\label{thm:convexity}
\
\begin{enumerate}

\item
If $\alpha < \beta$ with $\alpha,\beta,\alpha + \beta \in \wDelta^{+,\re}$, then $\alpha < \alpha + \beta < \beta$;

\item
For any $k$ if $\alpha < m_k(\alpha)$ with $\alpha \in \wDelta^{+,\re}$, then $\alpha < \alpha + k\delta < m_k(\alpha)$;

\item
For any $k$ if $M_k(\beta) < \beta$ with $\beta \in \wDelta^{+,\re}$, then $M_k(\beta) < \beta + k\delta < \beta$;

\item
If $\alpha < \beta$ with $\alpha,\beta \in \wDelta^{+,\re}$ and $\alpha + \beta$ is imaginary, then for all $k$:
  $$\alpha < m_k(\alpha) = m_k(\beta) \leq M_k(\alpha) = M_k(\beta) < \beta.$$

\end{enumerate}
\end{theorem}

We are now ready to introduce the key actor of the present work.

\begin{definition}\label{def:chains}
For $\alpha\in \wDelta^{+,\re}$, consider the decomposition $\alpha=\alpha'+k\delta$ with
$\alpha'\in \wDelta^{+,\re}, |\alpha'|<|\delta|, k\in \BZ_{\geq 0}$ (i.e.\ $\alpha'\in \Delta^+\cup (\delta-\Delta^+)$).
Define the \textbf{chain} of $\alpha$ as:
\begin{equation*}
  \chain(\alpha) = (\alpha',\alpha'+\delta,\alpha'+2\delta,\ldots).
\end{equation*}
\end{definition}

The following chain monotonicity property of~\cite[Lemma 5.7, Proposition 5.21]{ET} was actually established
at the same time as the above convexity:

\begin{proposition}[Monotonicity]\label{prop:monotonicity}
For $\alpha \in \wDelta^{+,\re}$, the following are equivalent:
\begin{enumerate}

\item
$\chain(\alpha)$ is monotone increasing (resp.\ decreasing);

\item
$\alpha < M_1(\alpha)$ (resp.\ $\alpha > M_1(\alpha))$;

\item
for all $k$: $\alpha < m_k(\alpha)$ (resp.\ $\alpha > M_k(\alpha)$).

\end{enumerate}
In particular, each chain $\chain(\alpha)$ is monotonous.
\end{proposition}

As an immediate corollary of this result, we obtain (\cite[Corollary 5.23]{ET}):

\begin{corollary}\label{cor:monotonicity}
For $\alpha \in \wDelta^{+,\re}$ with $|\alpha| < |\delta|$, the chain $\chain(\alpha)$ is monotone increasing
(resp.\ decreasing) iff $\chain(\delta - \alpha)$ is monotone decreasing (resp.\ increasing).
\end{corollary}

The next two lemmas follow easily from results above (\cite[Lemmas~5.25--5.26]{ET}):

\begin{lemma}\label{lemma:chain.sum.monotone}
If $\alpha,\beta, \alpha + \beta \in \wDelta^{+,\re}$, and both chains $\chain(\alpha),\chain(\beta)$ are increasing
(resp.\ decreasing), then the chain $\chain(\alpha + \beta)$ is also increasing (resp.\ decreasing).
\end{lemma}

\begin{lemma}\label{lemma:parity.same.sum}
(a) If $\alpha,\beta,\alpha + \beta \in \wDelta^{+,\re}$, with both $\chain(\alpha), \chain(\beta)$ increasing, then:
\begin{equation*}
  \chain(\alpha + \beta) < \min\{m_k(\alpha),m_k(\beta)\} \qquad \forall\, k \in \BZ_{>0}.
\end{equation*}

\noindent
(b) If $\alpha,\beta,\alpha + \beta \in \wDelta^{+,\re}$, with both $\chain(\alpha), \chain(\beta)$ decreasing, then:
\begin{equation*}
  \chain(\alpha + \beta) > \max\{M_k(\alpha),M_k(\beta)\} \qquad \forall\, k \in \BZ_{>0}.
\end{equation*}
\end{lemma}

As another application of Proposition~\ref{prop:monotonicity}, we recall the following three results pertaining to
imaginary SL-words (\cite[Lemma~5.27, Lemma~6.7, Lemma~6.8]{ET}):

\begin{lemma}\label{lemma:left.standard.comp.imaginary}
For any $i \in \{1,2,\ldots,|I|-1\}$ and $k \in \BZ_{>0}$: $\SL_{i+1}(k\delta) < \SL^{ls}_i(k\delta)$.
\end{lemma}

\begin{lemma}\label{lemma:imaginary.words.decreasing}
For any $i \in \{1,2,\ldots,|I|\}$, we have $\SL_i(\delta) > \SL_i(2\delta) > \dots$.
\end{lemma}

\begin{lemma}\label{lemma:length.standfac.imaginary}
For any $i \in \{1,2,\ldots,|I|\}$ and $k \in \BZ_{>0}$: $|\deg(\SL^{ls}_i(k\delta))| > |(k-1)\delta|$.
\end{lemma}

We conclude this section with the following two results pertaining to special orders on $\wI$ where the smallest
simple root $\alpha_\varepsilon$ occurs only once in $\delta$ (in particular, this applies to all orders with
$0\in \wI$ being the smallest letter), cf.~\cite[\S6.14]{ET}.

\begin{lemma}\label{lemma:monotonicity.smallest.once}
If $\alpha_\varepsilon$ occurs only once in $\delta$, then for any $\alpha \in \wDelta^{+,\re}$ with $|\alpha| < |\delta|$:
  $$ \chain(\alpha) \ \textit{increases} \ \Longleftrightarrow  \ \alpha \ \textit{contains\ } \alpha_\varepsilon. $$
\end{lemma}

\begin{lemma}\label{lemma:delta.stand.one.letter}
If $\alpha_\varepsilon$ occurs only once in $\delta$, then for any $i \in \{1,2,\ldots,|I|\}$, we have
$|\SL^{rs}_i(\delta)| = 1$ and all $\{\SL_i(\delta)\}_{i=1}^{|I|}$ end with pairwise distinct letters.
\end{lemma}


\section{Chains}\label{sec:chains}


\subsection{Combinatorics of chains}
\

We start with some basic operations on the chains of roots, see Definition~\ref{def:chains}.

\begin{definition}
a) For $\alpha,\beta,\alpha + \beta \in \wDelta^{+,\re}$, we define $\chain(\alpha) + \chain(\beta)$ as
$\chain(\alpha + \beta)$.

\smallskip
\noindent
b) For $\alpha,\beta,\alpha - \beta + k\delta \in \wDelta^{+,\re}$ with $k \in \BZ_{>0}$, we define
$\chain(\alpha) - \chain(\beta)$ as $\chain(\alpha - \beta)$.

\smallskip
\noindent
c) For a chain $\chain(\alpha)$ with $\alpha \in \wDelta^{+,\re}$, we define
$\delta - \chain(\alpha) = \chain(\delta - \alpha) = \chain(k\delta - \alpha)$,
whereas $k$ is large enough so that $k\delta - \alpha \in \wDelta^{+,\re}$.
\end{definition}

Each chain is monotonously decreasing or increasing, due to Proposition~\ref{prop:monotonicity}.
The overall analysis of the present note will crucially distinguish those two types.

\begin{definition}
We define the following two sets:
\begin{align*}
  \mathcal{I} &= \big\{ \text{increasing  }\chain(\alpha) \in \chain(\wDelta^{+,\re}) \big\} ,\\
  \mathcal{D} &= \big\{ \text{decreasing  } \chain(\alpha) \in \chain(\wDelta^{+,\re}) \big\}.
\end{align*}
\end{definition}

\noindent
Following the conventions of Definition~\ref{def:chains}, let $\alpha'$ be the shortest element of $\chain(\alpha)$.

\begin{definition}
Define $\pr\colon \chain(\wDelta^{+,\re}) \to \Delta$ via
  $\pr(\chain(\alpha))=
   \begin{cases}
     \alpha' & \text{if } \alpha' \in \Delta^+ \\
     \alpha'-\delta & \text{if } \alpha' \notin \Delta^+
   \end{cases}$.
We also set $\Delta^+_p=\pr(\mathcal{I})$ and $\Delta^-_p=\pr(\mathcal{D})$.
\end{definition}

\begin{lemma}\label{lemma:chain.positive.roots}
$\Delta=\Delta^+_p\cup \Delta^-_p$ is a polarization of the root system $\Delta$.
\end{lemma}

\begin{proof}
First, we note that $\beta \in \Pr(\mathcal{I})$ iff $-\beta \in \Pr(\mathcal{D})$, due to Corollary~\ref{cor:monotonicity}.
Additionally, if $\beta,\gamma \in \Pr(\mathcal{I})$ and $\beta + \gamma \in \Delta$, then
$\pr(\beta + \gamma) \in \Pr(\mathcal{I})$ by Lemma~\ref{lemma:chain.sum.monotone}. This completes the proof.
\end{proof}

By the very definition, we note that $\chain(\alpha)$ increases iff $\Pr(\chain(\alpha)) \in \Delta_p^+$.

\begin{definition}
a) We define the \textbf{relative height} of a chain $\chain(\alpha)$, denoted by $\hgt(\chain(\alpha))$,
to be $\hgt(\Pr(\chain(\alpha)))$ under the polarization $\Delta_p^+ \cup \Delta_p^-$.

\smallskip
\noindent
b) We call $\chain(\alpha)$ increasing (resp.\ decreasing) \textbf{\irrchain} if it is increasing
(resp.\ decreasing) and cannot be decomposed as $\chain(\alpha) = \chain(\beta) + \chain(\gamma)$ with
both $\chain(\beta),\chain(\gamma) \in \chain(\wDelta^{+,\re})$ increasing (resp.\ decreasing).
\end{definition}

Let $\Pi_{p}\subset \Delta^+_p$ be the set of simple roots of the polarization from Lemma~\ref{lemma:chain.positive.roots}.
They offer a natural description of irreducible increasing chains as follows.

\begin{lemma}\label{lemma:irr.chain.simple.root}
$\chain(\alpha)$ is an increasing \irrchain\ iff $\pr(\chain(\alpha))\in \Pi_{p}$.
\end{lemma}

Combining this with standard properties of root systems, we obtain:

\begin{corollary}\label{cor:lin.comb.chains}
a) Any $\chain(\alpha) \in \mathcal{I}$ can be expressed uniquely as a linear combination with
nonnegative integer coefficients of increasing \irrchains.

\smallskip
\noindent
b) Any $\chain(\alpha) \in \mathcal{D}$ can be expressed uniquely as a linear combination with
non-positive integer coefficients of increasing \irrchains.

\smallskip
\noindent
c) The chain graph is acyclic.
\end{corollary}

We shall often need the following useful result on finite root systems.

\begin{lemma}\label{lemma:seq.simple.root}
Consider any polarization $\Delta=\tilde{\Delta}^+\cup \tilde{\Delta}^-$ of a finite root system, and let
$\tilde{\Pi}=\{\epsilon_1,\ldots,\epsilon_n\}\subset \tilde{\Delta}^+$ be the simple positive roots. For any
$1\leq i_1,\ldots,i_N\leq n$ such that $\alpha=\epsilon_{i_1}+\ldots+\epsilon_{i_N}\in \tilde{\Delta}^+$ and
any $1\leq k\leq N$, there is a permutation $\sigma\in S_N$ such that $\sigma(1)=k$ and
$\eta_r:=\epsilon_{i_{\sigma(1)}}+\ldots+\epsilon_{i_{\sigma(r)}}$ are in $\tilde{\Delta}^+$ for all $1\leq r\leq N$.
\end{lemma}

\begin{proof}
We shall prove the generalization with all $\epsilon_{i_\bullet}$ being positive roots, instead of simple ones.
The proof proceeds by induction on $N$, the base case $N=1$ being vacuous. To simplify our exposition we shall
also assume $k=1$. If $(\epsilon_{i_1},\epsilon_{i_r}) < 0$ for some $r$, then $\epsilon_{i_1}+\epsilon_{i_r}$
is a positive root, hence we can apply the induction hypothesis for the collection of positive roots
$(\epsilon_{i_1}+\epsilon_{i_r},\epsilon_{i_2},\dots,\epsilon_{i_{r-1}},\epsilon_{i_{r+1}},\dots,\epsilon_{i_n})$.
The analogous argument applies if $(\alpha,\epsilon_{i_r}) > 0$ for some $r$, in which case we can apply the
induction hypothesis for the collection of positive roots
$(\epsilon_{i_1},\dots,\epsilon_{i_{r-1}},\epsilon_{i_{r+1}},\dots,\epsilon_{i_n})$.
Hence the only situation when we could not prove the claim via the argument above would be if
$(\epsilon_{i_1}, \epsilon_{i_r}) \geq 0 \geq (\alpha,\epsilon_{i_r})$ for all $r$. But then
$(\epsilon_{i_1}, \alpha-\epsilon_{i_1}) \geq 0 \geq (\alpha,\alpha-\epsilon_{i_1})$ and so
$(\alpha-\epsilon_{i_1},\alpha-\epsilon_{i_1}) \leq 0$, which is impossible as $\alpha-\epsilon_{i_1} \neq 0$.
\end{proof}

We can now relate irreducible chains to the flag $\spanset^1_{\bullet}$ of~\eqref{eq:flag}.

\begin{proposition}\label{prop:irr.chains}
There are exactly $|I|$ increasing \irrchains\ and any element of $\mathcal{I}$ can be uniquely expressed as
a sum of those. For any $i \in \{1,2,\ldots,|I|\}$:
\begin{equation}\label{eq:chains-vs-flags}
  \exists!  \ \ \text{increasing \irrchain\ } \ \ \chain(\alpha) \in \mathcal{I} \ \
  \text{ such that } \ \  m_1(\alpha) = (\delta,i).
\end{equation}
\end{proposition}

\begin{proof}
By Lemma~\ref{lemma:irr.chain.simple.root} we know that increasing \irrchains\ are in bijection with $\Pi_p$ and
clearly $|\Pi_p|=|I|$. Any element of $\mathcal{I}$ can be uniquely expressed as an unordered sum of increasing
\irrchains\ by Corollary~\ref{cor:lin.comb.chains}a).

We shall now prove~\eqref{eq:chains-vs-flags} by induction on $i$. For the base case $i=1$, consider
$\alpha=\deg(\SL^{ls}_1(\delta))$. We note that $m_1(\alpha) = (\delta,1)$ which determines $\chain(\alpha)$ uniquely
by definition. Furthermore, $\chain(\alpha)$ is increasing by Proposition~\ref{prop:monotonicity}.
If $\chain(\alpha)=\chain(\beta)+\chain(\gamma)$ with $\chain(\beta),\chain(\gamma)\in \mathcal{I}$, then
$m_1(\beta),m_1(\gamma)\leq(\delta,2)$, since $\chain(\alpha)$ is the only increasing chain satisfying
$m_1(\alpha) = (\delta,1)$ as $\dim(\spanset_{1}^1)=1$. Thus $\alpha<(\delta,2)$ by Lemma~\ref{lemma:parity.same.sum},
which contradicts to Lemma~\ref{lemma:left.standard.comp.imaginary}. This completes the base case.

For the inductive step, assume that~\eqref{eq:chains-vs-flags} holds for $i-1$. First, we note that if there is
$\chain(\alpha) \in \mathcal{I}$ such that $m_1(\alpha) = (\delta,i)$ and $\chain(\alpha)$ is an \irrchain, then it is unique.
If not, then we have $i+1$ increasing \irrchain\ with $m_1(\cdot)\geq (\delta,i)$, and hence their $\Pr$-images (simple roots
from $\Pi_p$ by Lemma~\ref{lemma:irr.chain.simple.root}) are contained in $i$-dimensional $\mathcal{S}^1_i$. As simple roots
are linearly independent, we get a contradiction.

To prove existence of such $\alpha$, let us consider $\chain(\tilde{\alpha})$ with $\tilde{\alpha}=\deg(\SL_i^{ls}(\delta))$.
If that chain is irreducible, then $\alpha=\tilde{\alpha}$ satisfies~\eqref{eq:chains-vs-flags}. Otherwise, apply
Corollary~\ref{cor:lin.comb.chains} to write
$\chain(\tilde{\alpha})=c_1 \chain(\beta_1) + c_2\chain(\beta_2) + \ldots + c_{|I|}\chain(\beta_{|I|})$ with
$c_k\in \BZ_{\geq 0}$, $\chain(\beta_k)\in \mathcal{I}$--irreducible, $|\beta_k|<|\delta|$ (i.e.\ $\beta_k=\beta'_k$),
and $m_1(\beta_k)=(\delta,k)$ for $k<i$ by the induction assumption. Let $\gamma$ be one of $\beta_j$'s with $c_j\ne 0$
and the smallest value of $m_1(\beta_j)$ (it may be not unique). If $m_1(\gamma)=(\delta,i)$, then the result holds with
$\alpha=\gamma$. Otherwise, $m_1(\gamma)\leq (\delta,i+1)$, which implies
$\chain(\gamma) < m_1(\gamma) < \deg(\SL_i^{ls}(\delta)) \leq \chain(\deg(\SL_i^{ls}(\delta)))$ by
Proposition~\ref{prop:monotonicity} and Lemma~\ref{lemma:left.standard.comp.imaginary}.

Let us list above $\beta_\bullet$ in any order, so that $\hat{\alpha}=\beta_{i_1}+\ldots+\beta_{i_N}$ with
$N=c_1+\ldots+c_{|I|}$ and $\hat{\alpha}\in \chain(\tilde{\alpha})$. Applying Lemma~\ref{lemma:seq.simple.root} with
$\tilde{\Delta}^+=\Delta^+_p$, there is a permutation $\sigma\in S_N$ such that $\beta_{i_{\sigma(1)}}=\gamma$ and
$\eta_r:=\beta_{i_{\sigma(1)}}+\ldots+\beta_{i_{\sigma(r)}}\in \Delta^+_p$ for all $1\leq r\leq N$. Note that
$m_1(\beta_s)\geq m_1(\gamma)$ by Lemma~\ref{lemma:min.im.rule} and so $\delta-\beta_s>m_1(\beta_s)\geq m_1(\gamma)$
for any $s$, due to Proposition~\ref{prop:monotonicity} and Corollary~\ref{cor:monotonicity}. Note that
$\eta_N=\hat{\alpha}\geq \tilde{\alpha}>(\delta,i+1)\geq m_1(\gamma)$. Applying Theorem~\ref{thm:convexity} to
$\eta_r+(N-r)\delta=\eta_{r+1}+(N-r-1)\delta +(\delta-\beta_{i_{\sigma(r+1)}})$ and arguing by a descending induction
on $r$, we get $\eta_r+(N-r)\delta > m_1(\gamma)$ for all $r$. For $r=1$, this gives $\gamma+(N-1)\delta>m_1(\gamma)$,
a contradiction with $\chain(\gamma)<m_1(\gamma)$.
\end{proof}

For $1\leq i\leq |I|$, let $\beta_i\in \Pi_p$ be such that $m_1(\beta_i) = (\delta,i)$.

\begin{corollary}\label{cor:mk.roots}
For any chain $\chain(\alpha) \in \chain(\wDelta^{+,\re})$, we have
\[
  m_k(\alpha) = (k\delta,i) \Longleftrightarrow
  \Pr(\chain(\alpha)) = c_1 \beta_1 + c_2\beta_2 + \ldots + c_i\beta_i  \ \text{ with }\ c_i \neq 0.
\]
\end{corollary}

\begin{proof}
We have $\spn\{h_{\beta_1}t^k,\ldots, h_{\beta_j}t^k\} = \spanset_{j}^k\ \ \forall\, j$,
due to Proposition~\ref{prop:spanset.equiv} and Corollary~\ref{cor:equiv.mk.Mk}.
Then
  $c_i \neq 0 \Leftrightarrow h_{\alpha}t^k \in \spanset_{i}^k \setminus \spanset_{i-1}^k
   \Leftrightarrow m_k(\alpha) = (k\delta,i)$
by Definition~\ref{def:mk}.
\end{proof}

As an immediate corollary, we obtain the following result.

\begin{corollary}\label{cor:both.increasing.m.k}
For any $\chain(\alpha),\chain(\beta) \in \chain(\wDelta^{+,\re})$ which are both increasing or decreasing and
$\chain(\alpha) + \chain(\beta) \in \chain(\wDelta^{+,\re})$, we have $m_k(\alpha + \beta) = \min\{m_k(\alpha),m_k(\beta)\}$.
\end{corollary}

We shall also often use the following simple observation.

\begin{corollary}\label{cor:b.j.M.k}
For any $\alpha \in \wDelta^{+,\re}$ with $M_k(\alpha) = (k\delta,i)$, we have:
\begin{equation*}
  (\alpha,\beta_j) = 0 \quad \forall\, j < i , \qquad (\alpha,\beta_i) \neq 0.
\end{equation*}
\end{corollary}

\begin{proof}
Note that $h_{\beta_j}t \in \spanset_{i-1}^1$ for all $j < i$ by Proposition~\ref{prop:irr.chains}, which
implies that $(\alpha,\beta_j) = 0$ due to Corollary~\ref{cor:im.span.vanish}. On the other hand, since
$h_{\beta_i}t \in \spanset_{i}^{1} \setminus \spanset_{i-1}^1$ by Proposition~\ref{prop:irr.chains},
we see that $(\alpha,\beta_i) \neq 0$ by Definition~\ref{def:mk}.
\end{proof}


\subsection{Properties of $\mathcal{I}$ and $\mathcal{D}$}
\

In this subsection, we explore some basic properties of increasing and decreasing chains.
The next definition will play a pivotal role through our analysis in Section~\ref{sec:dec}.

\begin{definition}\label{def:M.prime}
For any $\chain(\alpha) \in \mathcal{D}$ with $\chain(\alpha) = c_1\chain(\beta_1) + \ldots + c_{|I|}\chain(\beta_{|I|})$,
we define $M'_{k}(\chain(\alpha))\in \imx$ via
\begin{equation*}
  M'_{k}(\chain(\alpha)) = \max \big\{M_k(\beta_i) \,|\, c_i \ne 0\big\}.
\end{equation*}
We also define $M'_k(\alpha) = M'_k(\chain(\alpha)) = M'_k(\chain(\delta - \alpha))$ for all $\alpha \in \wDelta^{+,\re}$.
\end{definition}

The following property is straightforward from this definition:
\begin{equation}\label{eq:Mprime-additive}
  M'_1(\alpha + \beta) = \max\{M'_1(\alpha), M'_1(\beta)\} \quad \mathrm{for\ any} \quad
  \chain(\alpha),\chain(\beta),\chain(\alpha+\beta) \in \mathcal{D}.
\end{equation}

\begin{lemma}\label{lemma:M.prime.bound}
For any $\chain(\alpha) \in \chain(\wDelta^{+,\re})$ and $k\geq 1$, we have $M'_k(\alpha) \geq M_k(\alpha)$.
If $\chain(\alpha) = c_1\chain(\beta_1) + \ldots + c_{|I|}\chain(\beta_{|I|})$ with $c_{<i}=0, c_i \neq 0$,
then $M'_k(\alpha) \geq (k\delta,i)$.
\end{lemma}

\begin{proof}
For the first statement, let $M_k(\alpha) = (k\delta,j)$, so that $(\Pr(\chain(\alpha)),\beta_j) \neq 0$ by
Corollary~\ref{cor:b.j.M.k}. Then there exists $p$ with $c_p \neq 0$ such that $(\beta_p,\beta_j) \neq 0$, which
implies $M_k(\beta_p) \geq M_k(\alpha)$ by Corollary~\ref{cor:b.j.M.k}. The latter yields $M'_k(\alpha) \geq M_k(\alpha)$.

For the second statement, note that $M_k(\beta_i) \geq m_k(\beta_i) = (k\delta,i)$, which implies the claimed
inequality $M'_k(\alpha) \geq (k\delta,i)$.
\end{proof}

\begin{definition}\label{def:s-function}
Consider a decreasing chain $\chain(\alpha) = c_1\chain(\beta_1) + \ldots + c_{|I|}\chain(\beta_{|I|})$ and let
$M'_1(\chain(\alpha)) = (\delta,i)$. We then define $s(\chain(\alpha)) = s_\alpha\in \BZ_{>0}$ via
\begin{equation*}
  s(\chain(\alpha)) = \sum_{j\colon M_1(\beta_j) = (\delta,i)} (-c_j).
\end{equation*}
\end{definition}

The following result is crucial for many inductive arguments in this subsection.

\begin{lemma}\label{lemma:s.splitting}
For any $\chain(\alpha) \in \mathcal{D}$, if $s_\alpha > 1$, there are $\chain(\beta),\chain(\gamma) \in \mathcal{D}$
such that $M'_1(\beta) = M'_1(\gamma) = M'_1(\alpha)$ and $\chain(\beta) + \chain(\gamma) = \chain(\alpha)$.
\end{lemma}

\begin{proof}
This will be shown by induction on $\hgt(\chain(\alpha))$. Let $\chain(\alpha)$ be the chain with the minimal relative height
that satisfies the requirement (note that $\chain(\alpha)$ is not irreducible due to $s_\alpha \ne 1$). As $\chain(\alpha)$ is
not an \irrchain, it can be expressed as a sum $\chain(\alpha)=\chain(\beta)+\chain(\gamma)$ of two decreasing chains. Assuming
the contradiction with the lemma, we then have $M'_1(\beta)=M'_1(\alpha)>M'_1(\gamma)$ or $M'_1(\gamma)=M'_1(\alpha)>M'_1(\beta)$,
which implies $s_\beta = s_\alpha > 1$ or $s_\gamma = s_\alpha > 1$, thus contradicting to the choice of $\alpha$.
This completes the base case.

As for the induction step, assume that the claim holds for all chains of the relative height $<\hgt(\chain(\alpha))$.
Since $\chain(\alpha)$ is not an \irrchain, it can be expressed as a sum $\chain(\alpha)=\chain(\beta)+\chain(\gamma)$ of two
decreasing chains. Assuming that $\chain(\beta),\chain(\gamma)$ do not meet the requirements of the claim, we then again have
$M'_1(\beta) = M'_1(\alpha) > M'_1(\gamma)$ or $M'_1(\gamma) = M'_1(\alpha) > M'_1(\beta)$ by~\eqref{eq:Mprime-additive}.
Without loss of generality, we shall assume the former case. By the induction hypothesis, we can decompose
$\chain(\beta) = \chain(\eta) + \chain(\mu)$ into the sum of two decreasing chains satisfying
$M'_1(\eta)=M'_1(\mu) = M'_1(\beta)=M'_1(\alpha)$. Then for any $\hat{\eta},\hat{\mu},\hat{\gamma}$
in $\chain(\eta),\chain(\mu),\chain(\gamma)$, respectively, we have
$[[\sb[\SL(\hat{\eta})],\sb[\SL(\hat{\mu})]],\sb[\SL(\hat{\gamma})]] \neq 0$. By the Jacobi identity this implies that
$[\sb[\SL(\hat{\eta})],\sb[\SL(\hat{\gamma})]] \neq 0$ or $[\sb[\SL(\hat{\mu})],\sb[\SL(\hat{\gamma})]] \neq 0$, and therefore
either $\hat{\eta} + \hat{\gamma}$ or $\hat{\mu} + \hat{\gamma}$ are in $\wDelta^{+}$. Whichever of these sums is a root it
must be real since it is the sum of two elements determining decreasing chains (cf.\ Corollary~\ref{cor:monotonicity}), and
furthermore determines a decreasing chain by Lemma~\ref{lemma:chain.sum.monotone}. Hence, we can express $\chain(\alpha)$ as
either $(\chain(\gamma) + \chain(\eta)) + \chain(\mu)$ or $(\chain(\gamma) + \chain(\mu)) + \chain(\eta)$, respectively, with
both terms satisfying the conditions. This completes the inductive hypothesis.
\end{proof}

We can now deduce a useful lower bound for decreasing chains.

\begin{lemma}\label{lemma:dec.chains.lower.bound}
For any $\chain(\alpha) \in \mathcal{D}$, we have $M_1'(\alpha)<\chain(\alpha)$.
\end{lemma}

\begin{proof}
This will be shown by an induction on $s_\alpha$.

For the base case $s_\alpha = 1$, we have $M_1(\alpha) = M_1'(\alpha)$ as follows from
Corollary~\ref{cor:max.im.rule} and Lemma~\ref{lemma:seq.simple.root}. Hence, the result is implied by
Proposition~\ref{prop:monotonicity}.

If $s_\alpha > 1$, then there are $\chain(\beta),\chain(\gamma) \in \mathcal{D}$ such that
$\chain(\alpha) = \chain(\beta) + \chain(\gamma)$, $M'_1(\beta) = M'_1(\gamma) = M'_1(\alpha)$ by Lemma~\ref{lemma:s.splitting},
so that $s_\beta,s_\gamma<s_\alpha$. For any $\hat{\alpha} \in \chain(\alpha)$ with $|\hat{\alpha}| > |\delta|$ there
are $\hat{\beta} \in \chain(\beta), \hat{\gamma}\in\chain(\gamma)$ such that $\hat{\alpha} = \hat{\beta} + \hat{\gamma}$.
By the induction hypothesis, we have $M'_1(\alpha)=M'_1(\beta)<\hat{\beta}$, $M'_1(\alpha)=M'_1(\gamma)<\hat{\gamma}$, and
so $M'_1(\alpha) < \hat{\alpha}$ by Theorem~\ref{thm:convexity}. The result follows as $\chain(\alpha) \in \mathcal{D}$.
\end{proof}

The following useful observation is due to the first case of the above proof.

\begin{corollary}\label{cor:M.M.prime.diff}
For any $\chain(\alpha) \in \mathcal{D}$, if $M'_1(\chain(\alpha)) \neq M_1(\chain(\alpha))$, then $s_\alpha > 1$.
\end{corollary}

We also make the following general observation.

\begin{lemma}\label{lemma:interlock.m.M}
If $\chain(\alpha),\chain(\beta), \chain(\alpha+\beta) \in \chain(\wDelta^{+,\re})$, then:
\begin{gather*}
  m_1(\chain(\alpha)), m_1(\chain(\beta)), m_1(\chain(\alpha+\beta)) \leq
  M'_1(\chain(\alpha)), M'_1(\chain(\beta)), M'_1(\chain(\alpha+\beta)).
\end{gather*}
\end{lemma}

\begin{proof}
We will first show that $m_1(\chain(\alpha)), m_1(\chain(\beta)) \leq  M'_1(\chain(\alpha)), M'_1(\chain(\beta))$.
Note that $m_1(\chain(\gamma))\leq M_1(\chain(\gamma))\leq M'_1(\chain(\gamma))$ for $\gamma\in \{\alpha,\beta\}$.
Let $\alpha',\beta'$ be the shortest elements in the chains $\chain(\alpha),\chain(\beta)$. We note that
$m_1(\alpha')=m_1(\chain(\alpha))$, $M'_1(\alpha')=M'_1(\chain(\alpha))$, and alike for $\beta'$.
We consider two~cases:
\begin{enumerate}[leftmargin=0.7cm]

\item[1)]
If $(\alpha',\beta') \neq 0$, then $[h_{\beta'},\fg_{\alpha'}] \neq 0$, hence $[h_{\beta'}t,\fg_{\alpha'}] \neq 0$.
If $m_1(\beta')=(\delta,i)$, then
$h_{\beta'}t \in \spanset_{i}^1$, and so $M'_1(\alpha')\geq M_1(\alpha') \geq m_1(\beta')$ by
Lemma~\ref{lemma:M.prime.bound} and Corollary~\ref{cor:im.span.vanish}. Switching the roles of $\alpha',\beta'$, we
also get $M'_1(\beta') \geq m_1(\alpha')$.

\item[2)]
If $(\alpha',\beta') = 0$, then let $r_{\alpha'}$ denote the reflection with respect to the root $\alpha'$. Then,
$\alpha'-\beta'=-r_{\alpha'}(\alpha'+\beta')\in \wDelta^{\re}$, so that $\chain(\alpha) - \chain(\beta) \in \chain(\wDelta^{+,\re})$.
If $M'_1(\beta') = (\delta,i)$, then in the decomposition $\chain(\beta) = c_1\chain(\beta_1) + \ldots + c_{|I|}\chain(\beta_{|I|})$
we have $c_j =0$ for $j < i$ by Lemma~\ref{lemma:M.prime.bound}. Similarly, if $m_1(\alpha') = (\delta,k)$ and
$\chain(\alpha) = \bar{c}_1\chain(\beta_1) + \ldots + \bar{c}_{|I|}\chain(\beta_{|I|})$, then $\bar{c}_\jmath = 0$ for $\jmath > k$
by Lemma~\ref{cor:mk.roots}. Assuming the contradiction, let $M'_1(\beta') < m_1(\alpha')$ so that $i>k$. If $\chain(\alpha)$
and $\chain(\beta)$ have the same monotonicity (resp.\ different monotonicity), then $\chain(\alpha) - \chain(\beta)$
(resp.\ $\chain(\alpha) + \chain(\beta)$) would clearly have both positive and negative coefficients in its decomposition
into increasing \irrchains. The latter contradicts to Corollary~\ref{cor:lin.comb.chains}. This implies that
$M'_1(\beta') \geq m_1(\alpha')$, and switching the roles of $\alpha',\beta'$, we also get $M'_1(\alpha') \geq m_1(\beta')$.

\end{enumerate}

For the proof of the statement including $m_1(\chain(\alpha + \beta)),M'_1(\chain(\alpha + \beta))$, we can
apply the previous logic to $\chain(\alpha + \beta),\chain(-\beta)$ to get
$m_1(\chain(\beta)),m_1(\chain(\alpha + \beta))\leq M'_1(\chain(\beta)),M'_1(\chain(\alpha+\beta))$,
and similar argument applied to $\chain(\alpha + \beta),\chain(-\alpha)$ yields
$m_1(\chain(\alpha)),m_1(\chain(\alpha + \beta)) \leq M'_1(\chain(\alpha)),M'_1(\chain(\alpha+\beta))$.
\end{proof}

\begin{corollary}\label{cor:inc.dec.comparison}
For any $\alpha,\beta,\alpha + \beta \in \wDelta^{+,\re}$, if $\chain(\alpha) \in \mathcal{I}$ and
$\chain(\beta) \in \mathcal{D}$, then $\alpha< \alpha +\beta < \beta$.
\end{corollary}

\begin{proof}
If $\alpha,\beta,\alpha+\beta \in \wDelta^{+,\re}$, then $\alpha < m_1(\alpha)$ by Proposition~\ref{prop:monotonicity},
and $M'_1(\beta) < \beta$ by Lemma~\ref{lemma:dec.chains.lower.bound}. Thus by Lemma~\ref{lemma:interlock.m.M}, we have
$\alpha < m_1(\alpha) \leq M'_1(\beta) < \beta$. Applying Theorem~\ref{thm:convexity} we conclude that
$\alpha < \alpha + \beta < \beta$.
\end{proof}

This corollary implies the following useful property of chains.

\begin{corollary}\label{cor:factor.sets.monotonicity}
For any $\alpha,\beta,\alpha + \beta \in \wDelta^{+,\re}$ with $\alpha < \beta$:
\begin{enumerate}[leftmargin=1cm]

\item
if $\chain(\alpha + \beta) \in \mathcal{I}$, then $\chain(\alpha) \in \mathcal{I}$;

\item
if $\chain(\alpha + \beta) \in \mathcal{D}$, then $\chain(\beta) \in \mathcal{D}$.

\end{enumerate}
\end{corollary}

\begin{proof}
If $\chain(\alpha + \beta) \in \mathcal{I}$ but $\chain(\alpha) \in \mathcal{D}$, then
$\chain(\beta) \in \mathcal{I}$ by Lemma~\ref{lemma:chain.sum.monotone}. As $\beta < \alpha$
by Corollary~\ref{cor:inc.dec.comparison}, we obtain a contradiction. Thus $\chain(\alpha)$ must be increasing.

If $\chain(\alpha + \beta) \in \mathcal{D}$ but $\chain(\beta) \in \mathcal{I}$, then $\chain(\alpha) \in \mathcal{D}$
by Lemma~\ref{lemma:chain.sum.monotone}. As $\beta < \alpha$ by Corollary~\ref{cor:inc.dec.comparison}, we obtain
a contradiction. Hence $\chain(\beta)$ must be decreasing.
\end{proof}

The above result implies the following important property of affine SL-words.

\begin{corollary}\label{cor:lyndon.prefix.suffix.chains}
For any $\alpha \in \wDelta^{+,\re}$ and $\hat\alpha \in \chain(\alpha)$:
\begin{itemize}[leftmargin=1cm]

\item[(1)]
if $\chain(\alpha) \in \mathcal{I}$ and $\ell \in \LL$ is a nonempty prefix of $\SL(\hat\alpha)$,
then $\chain(\deg(\ell)) \in \mathcal{I}$;

\item[(2)]
if $\chain(\alpha) \in \mathcal{D}$ and $\ell \in \LL$ is a nonempty suffix of $\SL(\hat\alpha)$,
then $\chain(\deg(\ell)) \in \mathcal{D}$.

\end{itemize}
For any $\alpha \in \imx$:
\begin{itemize}[leftmargin=1cm]

\item[(3)]
if $\ell \in \LL$ is a nonempty proper prefix of $\SL(\alpha)$, then $\chain(\deg(\ell)) \in \mathcal{I}$;

\item[(4)]
if $\ell \in \LL$ is a nonempty proper suffix of $\SL(\alpha)$, then $\chain(\deg(\ell)) \in \mathcal{D}$.

\end{itemize}
\end{corollary}

\begin{proof}
We prove part (1) by a decreasing induction on the length of $\ell$. The base case $\ell = \SL(\hat\alpha)$ is obvious.
For the inductive step, suppose the result holds for all Lyndon prefixes $\ell'$ such that $|\ell'| > |\ell|$. We have
$\SL(\hat\alpha) = \ell w$ for some nonempty~$w$, and let $w = w_1\ldots w_n$ be its canonical factorization. Then
$\ell w_1\in \LL$ by Lemma~\ref{lemma:seq.right.word.Lyndon}, and so $\chain(\deg(\ell w_1)) \in \mathcal{I}$ by
the inductive hypothesis. If $\deg(w_1)$ is imaginary, then $\chain(\deg(\ell)) = \chain(\deg(\ell w_1)) \in \mathcal{I}$.
Assume now that $\deg(w_1)$ is real.
If $\deg(\ell)$ was imaginary, then $\ell = \SL(M_1(\deg(\ell w_1)))$ by Corollary~\ref{cor:imaginary.suffix.prefix}, implying
$\ell > \ell w_1$ by Proposition~\ref{prop:monotonicity}. The latter contradicts an obvious $\ell<\ell w_1$,
so that $\deg(\ell)$ must be real. Applying Corollary~\ref{cor:factor.sets.monotonicity} to $\deg(\ell)$ and $\deg(w_1)$,
we then obtain $\chain(\deg(\ell)) \in \mathcal{I}$.

Part (2) proceeds alike by a decreasing induction on the length of $\ell$. The base case $\ell = \SL(\hat\alpha)$
is clear. For the inductive step, suppose the result holds for all Lyndon suffixes $\ell'$ such that $|\ell'|> |\ell|$.
We have $\SL(\hat\alpha) = w\ell$ for some nonempty $w$, and let $w = w_1\ldots w_n$ be its canonical factorization.
Then $w_n\ell\in \LL$ by Lemma~\ref{lemma:seq.left.word.Lyndon}, and hence $\chain(\deg(w_n\ell)) \in \mathcal{D}$ by
the inductive hypothesis. If $\deg(w_n)$ is imaginary, then $\chain(\deg(\ell)) = \chain(\deg(w_n\ell)) \in \mathcal{D}$.
Assume now that $\deg(w_n)$ is real.
If $\deg(\ell)$ was imaginary, then $\ell = \SL(M_1(\deg(w_n\ell)))$ by Corollary~\ref{cor:imaginary.suffix.prefix},
implying $\ell < w_n\ell$ by Proposition~\ref{prop:monotonicity}. The latter contradicts to $w_n\ell\in \LL$,
so that $\deg(\ell)$ must be real. Applying Corollary~\ref{cor:factor.sets.monotonicity} to $\deg(w_n)$ and $\deg(\ell)$,
we then obtain $\chain(\deg(\ell)) \in \mathcal{D}$.

For part (3) note that any proper Lyndon prefix of $\SL(\alpha)$ must be a prefix of $\SL^{ls}(\alpha)$.
As $\chain(\deg(\SL^{ls}(\alpha))) \in \mathcal{I}$, applying part (1) yields the result.
Similarly for part (4) note that any proper Lyndon suffix of $\SL(\alpha)$ must be a suffix of $\SL^r(\alpha)$.
As $\chain(\deg(\SL^r(\alpha))) \in \mathcal{D}$, applying part (2) yields the result.
\end{proof}


\subsection{Intermediate root systems: connectivity}
\

In this subsection, we explore \textbf{intermediate root systems}, which form a nested system of root subsystems
of $\Delta$ induced by the complete flag $\spanset_\bullet^1$ of $\mathfrak{h}t$.

\begin{definition}\label{def:root_subsystems}
For $1 \leq i \leq |I|$, let $\Delta_i =\{\alpha \in \Delta \colon h_\alpha t \in \spanset_i^1\}$.
\end{definition}

We note that $\Delta_i$ is naturally equipped with a polarization through Lemma~\ref{lemma:chain.positive.roots}. Moreover,
$\{\Pr(\chain(\beta_j))\}_{j=1}^{i}$ is a set of simple roots for $\Delta_i$, due to Proposition~\ref{prop:irr.chains} and
Corollary~\ref{cor:mk.roots}. Our analysis of increasing chains is crucially based on the following conjectured property of
these root subsystems.

\begin{conjecture}[Connectivity]\label{conj:connectivity}
For any $1\leq i\leq |I|$, the root subsystem $\Delta_i$ is irreducible, or equivalently, the Dynkin subdiagram formed by
$\{\beta_j\}_{j=1}^{i}$ is connected.
\end{conjecture}

We note that if $\Delta_{i-1}$ is irreducible but $\Delta_i$ was not, then we would get
  $$(\beta_i,\alpha) = 0 \qquad \forall\, \alpha \in \Delta_{i-1}.$$
Thus the above conjecture is equivalent to
\begin{equation}\label{eq:connectivity.equiv}
  \forall\, i>1 \ \ \exists\, j<i \colon (\beta_i,\beta_j)\ne 0.
\end{equation}

We note the following simple criteria that suffices to establish Conjecture~\ref{conj:connectivity}:

\begin{proposition}\label{prop:connectivity.helper}
Conjecture~\ref{conj:connectivity} holds if $\SL_i^l(\delta) \neq \SL_i^{ls}(\delta)$ for all $i > 1$.
\end{proposition}

The proof will use the classical connectedness property of supports of roots, see e.g.\ \cite[Proposition 16.21]{C},
which we formulate in the form as will be used below:

\begin{claim}\label{claim:support-connectedness}
For any $\gamma=\sum_{i=1}^{|I|} c_i\beta_i\in \Delta^+_{p}$, its \emph{support}
$\mathrm{supp}(\gamma):=\{\beta_i\colon c_i\ne 0\}$ determines a connected subdiagram of the Dynkin diagram of $\fg$.
\end{claim}

\begin{proof}[Proof of Proposition~\ref{prop:connectivity.helper}]
Assuming $\SL_i^l(\delta) \neq \SL_i^{ls}(\delta)$ for all $i > 1$, we shall show that $\Delta_i$ are irreducible
by induction on $i$. The base case $i=1$ is vacuous. For the inductive hypothesis, suppose that $\Delta_{i-1}$
is irreducible. Assuming the contrary with $\Delta_i$ being irreducible, we get $(\beta_i,\beta_j) = 0$ for all $j<i$,
due to~\eqref{eq:connectivity.equiv}.

Then there is only one increasing chain $\chain(\gamma)$ (namely $\chain(\beta_i)$) with $m_1(\chain(\gamma))=(\delta,i)$,
due to Claim~\ref{claim:support-connectedness}. But the roots $\alpha = \deg(\SL^l_i(\delta))$ and
$\beta = \deg(\SL^{ls}_i(\delta))$ satisfy $m_1(\alpha)=m_1(\beta)=(\delta,i)$ by Corollary~\ref{cor:mk.im.factor},
and thus $\chain(\alpha),\chain(\beta)\in \mathcal{I}$ by Proposition~\ref{prop:monotonicity}.
Moreover, since $0<|\alpha|, |\beta| < |\delta|$ and $\SL_i^l(\delta) \neq \SL_i^{ls}(\delta)$, we get
$\chain(\alpha) \ne \chain(\beta)$, a contradiction with above. This completes the inductive step.
\end{proof}

\begin{remark}
Using the code, we verified the above property $\SL_i^l(\delta) \neq \SL_i^{ls}(\delta)\, \forall\, i > 1$
for all orders on $\wI$ for all exceptional types and all classical types in rank $\leq 6$.
\end{remark}

\begin{lemma}
If the smallest simple root occurs only once in $\delta$, then Conjecture~\ref{conj:connectivity} holds.
\end{lemma}

\begin{proof}
According to Proposition~\ref{prop:connectivity.helper}, it suffices to show that $\SL_i^r(\delta) \neq \SL_i^{rs}(\delta)$
for all $i > 1$. First, we note that $\SL_1^r(\delta) = \SL_1^{rs}(\delta) = k$, where $k$ is the second smallest letter of $\wI$
(the first equality follows from Corollary~\ref{cor:mk.im.factor}, while the second one is then due to
Lemma~\ref{lemma:delta.stand.one.letter}). For any $i>1$, we also know that $\SL^{rs}_i(\delta) = k'$ for a letter
$k' > k$, due to Lemma~\ref{lemma:delta.stand.one.letter}. Hence, the letter $k$ must appear in $\SL^{ls}_i(\delta)$ and
cannot be its first letter, implying $\SL^r_i(\delta) \neq \SL^{rs}_i(\delta)$ by Lemma~\ref{lemma:right.costfac.minimal}.
\end{proof}

\begin{lemma}\label{lemma:unique.simple.add}
If Conjecture~\ref{conj:connectivity} holds, then for any $1 < i \leq |I|$, there exists a unique $j < i$ such that
$\beta_i + \beta_j \in \Delta_{i}$, and for such $j$ we also have $M_1(\beta_i) = (\delta,j)$.
\end{lemma}

\begin{proof}
For the existence, we note that such $j$ is provided by $M_1(\beta_i) = (\delta,j)$: indeed $j < i$ is due
to~\eqref{eq:connectivity.equiv}. Assume now that $j\ne j' < i$ both satisfy
$\beta_i + \beta_j,\beta_i + \beta_{j'} \in \Delta_i$. Consider the Dynkin diagram of $\Delta_{i-1}$,
which is connected by Conjecture~\ref{conj:connectivity}.
Thus the nodes in $\Delta_{i-1}$ corresponding to $\beta_j,\beta_{j'}$ are connected by a path, and since
both of them are connected by an edge to the node corresponding to $\beta_i$ in $\Delta_i$, we obtain a
cycle in $\Delta_i$, a contradiction. This establishes the uniqueness of such $j$.
\end{proof}

\begin{corollary}\label{cor:irr.chain.M.prime}
If Conjecture~\ref{conj:connectivity} holds, then for any $1 \leq i \leq |I|$ and $\chain(\alpha) \in \mathcal{I}$
with $m_1(\alpha) \geq (\delta,i)$ such that $\beta_i + \alpha$ is a real root, we have $M_1(\beta_i) \leq M'_1(\alpha)$.
\end{corollary}

\begin{proof}
The claim is clear if $m_1(\alpha) = (\delta,i)$, due to the definition of $M'_1$ and Corollary~\ref{cor:mk.roots}.
Suppose now $m_1(\alpha) > (\delta,i)$, so that in the decomposition
$\chain(\alpha) = c_1\chain(\beta_1) + \cdots + c_{|I|}\chain(\beta_{|I|})$, we have $c_{i'} = 0$ for all $i' \geq i$
by Corollary~\ref{cor:mk.roots}. As $\beta_i + \alpha$ is a real root, evoking Lemma~\ref{lemma:seq.simple.root},
we obtain $\beta_i + \beta_{j}$ is a root for some $j < i$ with $c_j > 0$, which must be the same $j$ as in
Lemma~\ref{lemma:unique.simple.add}. Thus, by Lemma~\ref{lemma:unique.simple.add} and~\eqref{eq:Mprime-additive},
we get $M_1(\beta_i) = (\delta,j) = m_1(\beta_j) \leq M_1(\beta_j) \leq M'_1(\alpha)$.
\end{proof}

The following simple property of root systems will be used below.

\begin{lemma}\label{lemma:triple.sum.pairs}
For any $\alpha,\beta,\gamma \in \Delta$ such that $\alpha+\beta+\gamma \in \Delta$, if none of the pairwise sums
$\alpha + \beta,\alpha + \gamma,\beta + \gamma$ vanishes then at least two of them are in $\Delta$.
\end{lemma}

\begin{proof}
Let us first show that at least one of above sums is a root. If not, then we would have
$(\alpha,\beta), (\alpha,\gamma), (\beta,\gamma) \geq 0$. But then
$(\alpha + \beta + \gamma,\alpha) = (\alpha,\alpha) + (\beta,\alpha) + (\gamma,\alpha) > 0$, which implies
that $(\alpha + \beta + \gamma) - \alpha = \beta + \gamma$ is a root (as $\beta + \gamma \neq 0$).

Suppose now, without loss of generality, that $\alpha + \beta$ is a root. Then we must have
$[[\fg_\alpha,\fg_\beta],\fg_{\gamma}] \neq 0$. By the Jacobian identity we then obtain
$[[\fg_\beta,\fg_\gamma],\fg_\alpha]\ne 0$ or $[[\fg_{\gamma},\fg_\alpha],\fg_\beta]\ne 0$,
so that $\beta + \gamma\in \Delta$ or $\alpha + \gamma\in \Delta$, respectively, as claimed.
\end{proof}

For the upcoming analysis of increasing chains, we shall need the following result.

\begin{lemma}\label{lemma:connectivity.splitting}
If connectivity of Conjecture~\ref{conj:connectivity} holds, then for any not irreducible $\chain(\alpha) \in \mathcal{I}$
there exists a splitting $\chain(\alpha) = \chain(\beta) + \chain(\gamma)$ with $\chain(\beta),\chain(\gamma) \in \mathcal{I}$
such that $m_1(\beta) = m_1(\alpha)$ and $M'_1(\gamma) = M'_1(\alpha)$.
\end{lemma}

\begin{proof}
Let $m_1(\alpha) = (\delta,i)$. We will prove this by induction on the relative height of $\chain(\alpha)$.
The base case $\hgt(\chain(\alpha)) = 2$ corresponds to $\chain(\alpha) = \chain(\beta_i) + \chain(\beta_j)$
with $j$ as in Lemma~\ref{lemma:unique.simple.add}. Then
$M'_1(\alpha) = M'_1(-\alpha) = \max\{M'_1(-\beta_i),M'_1(-\beta_j)\} = \max\{(\delta,j),M'_1(-\beta_j)\} = M_1(\beta_j)$
by Lemma~\ref{lemma:unique.simple.add}, establishing the base case.

Suppose now that the claim holds for all increasing chains of the relative height $<\hgt(\chain(\alpha))$. Consider
any splitting $\chain(\alpha) = \chain(\beta) + \chain(\gamma)$ with $\chain(\beta),\chain(\gamma) \in \mathcal{I}$.
By Corollary~\ref{cor:both.increasing.m.k}, we can assume without loss of generality that $m_1(\beta) = m_1(\alpha)$.
According to~\eqref{eq:Mprime-additive}, we shall now consider two cases:
\begin{itemize}[leftmargin=0.7cm]

\item
If $M'_1(\gamma) = M'_1(\alpha)$, then $\chain(\beta),\chain(\gamma)$ satisfy the requirements of the lemma;

\item
If $M'_1(\gamma) \ne M'_1(\alpha)$, then $M'_1(\alpha) = M'_1(\beta)$ by~\eqref{eq:Mprime-additive}. We also note that
$\chain(\beta)$ cannot irreducible, as that would contradict Corollary~\ref{cor:irr.chain.M.prime}. Applying the
inductive hypothesis to $\chain(\beta)$, there exists a decomposition $\chain(\beta) = \chain(\eta) + \chain(\mu)$
with $\chain(\eta),\chain(\mu)\in \mathcal{I}$, $m_1(\eta) = m_1(\beta)$ and $M'_1(\mu) = M'_1(\beta)$.
Considering the projections of $\chain(\eta),\chain(\mu),\chain(\gamma)$, by Lemma~\ref{lemma:triple.sum.pairs}
either $\Pr(\chain(\eta)) + \Pr(\chain( \gamma))$ or $\Pr(\chain(\mu)) + \Pr(\chain(\gamma))$ are in $\Delta$,
hence either $\chain(\eta + \gamma) \in \mathcal{I}$ or $\chain(\mu + \gamma) \in \mathcal{I}$ by
Lemma~\ref{lemma:chain.sum.monotone}. In the first case, $m_1(\eta + \gamma) = m_1(\alpha)$ and
$M'_1(\mu) = M'_1(\beta) = M'_1(\alpha)$ by Corollary~\ref{cor:both.increasing.m.k}. In the second case, we have
$m_1(\eta) = m_1(\beta) = m_1(\alpha)$ and $M'_1(\mu + \gamma) =  M'_1(\beta) = M'_1(\alpha)$ by~\eqref{eq:Mprime-additive}.

\end{itemize}
This completes the inductive step.
\end{proof}


\section{Periodicity of Decreasing Chains}\label{sec:dec}


\subsection{Preliminaries}
\

In this section, we will show that decreasing chains exhibit a compelling \textbf{periodicity}, that is we
can easily relate $\SL(\alpha)$ to $\SL(\alpha + k\delta)$ for some $k$ with $|\alpha| \geq q$, where $q,k$ depend only on
$\chain(\alpha)\in \mathcal{D}$. In fact, we explicitly describe what $k,q$ are.

Our first result investigates the prefix and suffix of words in decreasing chains.

\begin{lemma}\label{lemma:left.standard.decreasing}
For any $\alpha \in \wDelta^{+,\re}$ such that $|\alpha| > |\delta|$ and $\chain(\alpha)\in \mathcal{D}$,
$\chain(\SL^{rs}(\alpha))$ decreases and if $\SL^{ls}(\alpha)$ is real then $\chain(\deg(\SL^{ls}(\alpha)))$ decreases.
\end{lemma}

\begin{proof}
The first claim, asserting that $\chain(\SL^{rs}(\alpha))$ is decreasing, follows immediately from
Corollary~\ref{cor:lyndon.prefix.suffix.chains}.

We prove the second part by a contradiction: assume that $\beta = \deg(\SL^{ls}(\alpha))$ is real and
$\chain(\beta)\in \mathcal{I}$. We then have $M'_1(\alpha) < m_1(\beta)$: otherwise we would have $M'_1(\alpha) > \beta$
by Proposition~\ref{prop:monotonicity}, which contradicts to Lemma~\ref{lemma:equiv.to.standard.fac} as
$\beta < M'_1(\alpha) < \alpha$ by Lemma~\ref{lemma:dec.chains.lower.bound} and $|M'_1(\alpha)| = |\delta| < |\alpha|$.
But the inequality $M'_1(\alpha) < m_1(\beta)$ contradicts to Lemma~\ref{lemma:interlock.m.M}, hence,
$\chain(\beta)\in \mathcal{D}$ if $\beta$ is real.
\end{proof}

This allows to explicitly describe SL-words in irreducible decreasing chains.

\begin{corollary}\label{cor:dec.irr.chains}
For any $1\leq i\leq |I|$, pick $\beta \in \chain(\delta - \beta_i)$ with $|\beta| < |\delta|$. Then:
\begin{equation*}
  \SL(\beta + k\delta) = \underbrace{\SL(M_1(\beta))}_{k \text{ times}}\SL(\beta), \qquad
  \SL^{ls}(\beta + k\delta) = \SL(M_1(\beta)) \qquad \forall\, k\geq 1.
\end{equation*}
\end{corollary}

\begin{proof}
First, we note that $\SL^{ls}(\beta + k\delta)$ is imaginary due to Lemma~\ref{lemma:left.standard.decreasing},
as $\chain(\beta)$ is a decreasing \irrchain. Therefore $\SL^{ls}(\beta + k\delta) = \SL(M_1(\beta))$ by
Corollary~\ref{cor:imaginary.suffix.prefix}. Note that $M_1(\beta)=M'_1(\beta)$. We further obtain the first formula
by an iterative application of $\SL(\beta + k\delta)=\SL(M_1(\beta))\SL(\beta + (k-1)\delta)$.
\end{proof}

The goal of the next few results is to ultimately show that $\SL(M'_1(\alpha))$ is the only imaginary word which can appear
as a subword of $\SL(\alpha)$ when $\chain(\alpha)\in \mathcal{D}$.

\begin{lemma}\label{lemma:M.prime.prefix}
For any $\alpha \in \wDelta^{+,\re}$, if $\SL(\alpha)$ has $\SL_i(\delta)$ as a prefix, then $\chain(\alpha) \in \mathcal{D}$
and $M'_1(\alpha) = (\delta,i)$.
\end{lemma}

\begin{proof}
First, we note that $\chain(\alpha) \in \mathcal{D}$ by Corollary~\ref{cor:lyndon.prefix.suffix.chains}. By the assumption
we have $\SL(\alpha) = \SL_i(\delta)w$ for some nonempty $w$, and let $w= w_1w_2\ldots w_n$ be its canonical factorization.
We shall prove that $M'_1(\alpha) = (\delta,i)$ by induction on $n$. The base case $n=1$ corresponds to
$\SL(\alpha) = \SL_i(\delta)\ell$ with Lyndon~$\ell$. From Corollary~\ref{cor:imaginary.suffix.prefix} we have
$M_1(\alpha) = (\delta,i)$, hence $M'_1(\alpha) \geq (\delta,i)$ by Lemma~\ref{lemma:M.prime.bound}. However, since
$M'_1(\alpha) < \alpha$ by Lemma~\ref{lemma:dec.chains.lower.bound}, we must have $M'_1(\alpha) = (\delta,i)$.

For the step of induction, we can assume that the result holds for $v=\SL_i(\delta)w_1w_2\ldots w_{n-1}$, which is Lyndon
by Lemma~\ref{lemma:seq.right.word.Lyndon}. Thus $M'_1(\beta) = (\delta,i)$ with $\beta = \deg(v)$.
As $\chain(\gamma) \in \mathcal{D}$ by Corollary~\ref{cor:lyndon.prefix.suffix.chains}, if $M'_1(\alpha) \neq (\delta,i)$
then $M'_1(\alpha) > (\delta,i)$ by~\eqref{eq:Mprime-additive}, a contradiction with $\SL(M'_1(\alpha)) < \SL(\alpha)$ of
Lemma~\ref{lemma:dec.chains.lower.bound} as now $\SL(M'_1(\alpha)) > \SL_i(\delta)w =\SL(\alpha)$.
This completes the inductive step.
\end{proof}

As an immediate application of Lemma~\ref{lemma:stronger.no.splitting} to the present setup, we get:

\begin{lemma}\label{lemma:no.splitting.decreasing}
For any $\alpha \in \wDelta^{+,\re}$ if $\chain(\alpha) \in \mathcal{D}$, then the standard factorization of $\SL(\alpha)$ does
not split imaginary words. Moreover, if $\SL^{ls}(\alpha)$ is real, then $\SL_i(\delta)$ is not a suffix of $\SL^{ls}(\alpha)$.
\end{lemma}

\begin{proof}
For the first statement, note that $\SL(\alpha)$ cannot have an imaginary suffix by Corollary~\ref{cor:lyndon.prefix.suffix.chains},
hence the claim follows immediately from Lemma~\ref{lemma:stronger.no.splitting}. For the second statement, assume that
$\SL^{ls}(\alpha)$ is real and contains $\SL_i(\delta)$ as a suffix, then $|\alpha| > |\delta|$. However, by
Lemma~\ref{lemma:left.standard.decreasing}, we must have $\chain(\deg(\SL^{ls}(\alpha))) \in \mathcal{D}$.
Applying Corollary~\ref{cor:lyndon.prefix.suffix.chains} to $\SL^{ls}(\alpha)$ then yields a contradiction.
\end{proof}

Here, it is crucial to use the standard factorization for decreasing chains, since the costandard
factorization can actually split imaginary words, as illustrated in the following example.

\begin{example}\label{ex:costand-decr-fails}
Consider the affine type $A_n^{(1)}$ with the order $1 < 2 < \cdots < n < 0$. According to~\cite[Theorem 4.2]{AT},
applied to $\chain(\alpha_0 + \delta)\in \mathcal{D}$ by Lemma~\ref{lemma:monotonicity.smallest.once}:
\begin{equation*}
  \SL(\alpha_0 + \delta) = 10\vert2\ldots n0 = \SL_{n-1}(\delta)0,
\end{equation*}
where the vertical line denotes the costandard factorization by Lemma~\ref{lemma:right.costfac.minimal}.
Moreover, this example also shows that there is no analogue of Lemma~\ref{lemma:left.standard.decreasing}
for costandard factorizations: $\chain(\deg(\SL^l(\alpha_0+\delta)))=\chain(\alpha_0+\alpha_1)$ is increasing
by Lemma~\ref{lemma:monotonicity.smallest.once}.
\end{example}

With the help of above auxiliary steps, we are finally ready to establish the aforementioned result on imaginary subwords
in decreasing chains.

\begin{proposition}\label{prop:imaginary.substring.decreasing}
For any $\alpha \in \wDelta^{+,\re}$ with $\chain(\alpha) \in \mathcal{D}$, if $\SL_i(\delta)$ is a subword of $\SL(\alpha)$,
then $M'_1(\chain(\alpha)) = (\delta,i)$. Moreover, $\SL_i(\delta)$ is also a prefix of $\SL(\alpha)$.
\end{proposition}

\begin{proof}
We shall prove this by induction on $|\alpha|$. For the base case, $\alpha$ being of minimal height satisfying the conditions,
we claim that $\SL^{ls}(\alpha)$ is imaginary, in which case $|\SL^{ls}(\alpha)| = |\delta|$ by
Corollary~\ref{cor:imaginary.suffix.prefix}, and the result holds by Lemma~\ref{lemma:M.prime.prefix}.
If not, then $\SL^{ls}(\alpha)$ is real and so $\chain(\deg(\SL^{ls}(\alpha))),\chain(\deg(\SL^{rs}(\alpha)))\in \mathcal{D}$
by Lemma~\ref{lemma:left.standard.decreasing}. But then $\SL_i(\delta)$ is a subword of either $\SL^{ls}(\alpha)$ or
$\SL^{rs}(\alpha)$ by Lemma~\ref{lemma:no.splitting.decreasing}, a contradiction with the choice of $\alpha$.

For the inductive step, assume the result holds for all roots in decreasing chains of height $< |\alpha|$.
By Lemma~\ref{lemma:no.splitting.decreasing}, if $\SL_j(\delta)$ is a subword of $\SL(\alpha)$ then it is a subword of
either $\SL^{ls}(\alpha)$ or $\SL^{rs}(\alpha)$. Since $|\alpha| > |\delta|$, Lemma~\ref{lemma:left.standard.decreasing}
provides two cases:
\begin{enumerate}[leftmargin=0.7cm]

\item[1)]
If $\SL^{ls}(\alpha)$ is imaginary, then it must be of length $|\delta|$ by Corollary~\ref{cor:imaginary.suffix.prefix},
and so $M'_1(\alpha) = (\delta,j)$ by Lemma~\ref{lemma:M.prime.prefix}. By the inductive assumption for $\SL^{rs}(\alpha)$,
the only possible imaginary subword of $\SL(\alpha)$ is $\SL_i(\delta)$, and is its~prefix.

\item[2)]
If $\SL^{ls}(\alpha)$ is real, then $\chain(\deg(\SL^{ls}(\alpha))),\chain(\deg(\SL^{rs}(\alpha)))$ are both decreasing
by Lemma~\ref{lemma:left.standard.decreasing}. Let $M'_1(\chain(\alpha)) = (\delta,i)$. Assume first that $\SL_j(\delta)$
is a subword of $\SL^{ls}(\alpha)$ for $j \neq i$. We note that $j > i$, as otherwise applying the inductive hypothesis to
$\SL^{ls}(\alpha)$ we would get $M'_1(\chain(\deg(\SL^{ls}(\alpha)))) > M'_1(\chain(\alpha))$, a contradiction
with~\eqref{eq:Mprime-additive}. By the inductive hypothesis, $\SL_j(\delta)$ is then a prefix of $\SL^{ls}(\alpha)$,
and so $M'_1(\alpha) > \alpha$, in contradiction with Lemma~\ref{lemma:dec.chains.lower.bound}. Likewise
by~\eqref{eq:Mprime-additive} if $\SL_j(\delta)$ was a subword of $\SL^{rs}(\alpha)$, then $j \geq i$. However, if $j > i$,
then by the inductive hypothesis $\SL_j(\delta)$ is a prefix of $\SL^{rs}(\alpha)$ and we get
$\SL(\alpha) < \SL^{rs}(\alpha) < \SL(M'_1(\alpha))$ again contradicting Lemma~\ref{lemma:dec.chains.lower.bound}.

\end{enumerate}

It thus only remains to show that if $\SL_i(\delta)$ is a subword of $\SL(\alpha)$, then $\SL_i(\delta)$ is a prefix of
$\SL(\alpha)$. If $\SL_i(\delta)$ is a subword of $\SL^{ls}(\alpha)$, then applying the inductive hypothesis yields the result.
If $\SL_i(\delta)$ is a subword of $\SL^{rs}(\alpha)$, then by the inductive hypothesis we have $\SL_i(\delta)$ is a prefix
of $\SL^{rs}(\alpha)$, and we get $\SL(M'_1(\alpha)) = \SL_i(\delta) < \SL(\alpha) < \SL^{rs}(\alpha)$ by
Lemma~\ref{lemma:dec.chains.lower.bound}, which implies that $\SL_i(\delta)$ is also a prefix of $\SL(\alpha)$.
This completes the inductive step.
\end{proof}

We are now ready to prove the first result regarding repetitive occurrence of imaginary subwords in decreasing chains.

\begin{proposition}\label{prop:weak.form.desc.irr.chains}
For any decreasing chain $\chain(\alpha)$ and any $\hat{\alpha}\in \chain(\alpha)$, we have:
\begin{gather}\label{eqn:imaginary.prefix.chunks}
  \SL(\hat{\alpha}) = \underbrace{\SL(M'_1(\alpha))}_{p \text{ times}} w
\end{gather}
for some $p \geq 0$, where $\SL(M'_1(\alpha))$ is not a prefix of $w$ and $w > \SL(M'_1(\alpha))$.
Moreover, for any $p' > 0$ there exists a $q$ such that if $|\hat{\alpha}| > q|\delta|$, then $p \geq p'$.
Finally, the above function $p\colon \chain(\alpha)\to \BZ_{\geq 0}$ is monotone non-decreasing.
\end{proposition}

\begin{proof}
The first statement follows from $\SL(\hat{\alpha})>\SL(M'_1(\alpha))$ due to Lemma~\ref{lemma:dec.chains.lower.bound} and
$\SL(\hat{\alpha})$ being Lyndon (implying that all proper suffixes of $\SL(\hat{\alpha})$ are $>\SL(M'_1(\alpha))$). The last
claim that $p\colon \chain(\alpha)\to \BZ_{\geq 0}$ is monotone non-decreasing is clear, as otherwise we would have
$\hat{\alpha}<\hat{\alpha} + \delta$ for some $\hat{\alpha}\in \chain(\alpha)$, a contradiction with $\chain(\alpha)\in \mathcal{D}$.

We shall now prove the second statement by induction on $\hgt(\chain(\alpha))$. The base case $\hgt(\chain(\alpha))=1$ follows
from Corollary~\ref{cor:dec.irr.chains}. As for the induction step, let us assume that the claim holds for all decreasing chains
of relative height $<\hgt(\chain(\alpha))$. We shall consider two cases:
\begin{enumerate}[leftmargin=0.7cm]

\item[1)]
First, if $\SL^{ls}(\alpha' + k\delta)$ is imaginary for $k \gg 0$, then $\SL^{ls}(\alpha' + k\delta) = \SL(M_1(\alpha'))$
for $k \gg 0$ by Corollaries~\ref{cor:imaginary.suffix.prefix}. We note that then
$M_1(\alpha') = M'_1(\alpha')$ as otherwise $M_1(\alpha') < M'_1(\alpha')$ by Lemma~\ref{lemma:M.prime.bound}, and so
$\alpha' + k\delta < M'_1(\alpha')$ by Lemma~\ref{lemma:equiv.to.standard.fac}, a contradiction with
Lemma~\ref{lemma:dec.chains.lower.bound}. Thus, $\SL^{ls}(\alpha' + k\delta) = \SL(M_1(\alpha')) = \SL(M'_1(\alpha'))$ and
$\SL^{rs}(\alpha' + k\delta)=\SL(\alpha' + (k-1)\delta)$. Hence, for any $p'>0$ there is a sufficiently large $q$ such that
$p(\alpha' + k\delta) \geq p'$ for $k\geq q$.

\item[2)]
Let us now assume that $\SL^{ls}(\alpha' + k\delta)$ is real for an unbounded set of $k$. By the induction assumption, we know
that for any $\chain(\beta)\in \mathcal{D}$ with $\hgt(\chain(\beta))<\hgt(\chain(\alpha))$, given any $p' > 0$ there is some
$q_{\chain(\beta)}$ such that $p(\hat{\beta}) \geq p'$ for any $\hat{\beta} \in \chain(\beta)$ with
$|\hat{\beta}| > q_{\chain(\beta)}|\delta|$. We claim that $p(\alpha' + k\delta) \geq p'$ whenever
$\SL^{ls}(\alpha' + k\delta)$ is real and
\begin{equation*}
  k \geq 1 + 2\max\big\{q_{\chain(\beta)} \,|\, \hgt(\chain(\beta)) < \hgt(\chain(\alpha))\big\}.
\end{equation*}
Assuming that $\SL^{ls}(\alpha' + k\delta)$ is real, both $\SL^{ls}(\alpha' + k\delta)$, $\SL^{rs}(\alpha' + k\delta)$
determine decreasing chains by Lemma~\ref{lemma:left.standard.decreasing}, and one of these words is sufficiently long
so that it starts with at least $p'$ copies of $\SL(M'_1(\cdot))$. We consider two cases:
\begin{enumerate}[leftmargin=0.7cm]

\item[i)]
Suppose first that $\SL^{ls}(\alpha' + k\delta)$ starts with $\geq p'$ copies of $\SL(M'_1(\beta))$, where
$\beta = \deg(\SL^{ls}(\alpha' + k\delta))$. We note that $M'_1(\beta) \leq M'_1(\alpha)$ by~\eqref{eq:Mprime-additive}.
We also cannot have $M'_1(\beta) < M'_1(\alpha)$ as this would imply $M'_1(\alpha) > \alpha' + k\delta$, a contradiction
with Lemma~\ref{lemma:dec.chains.lower.bound}. Thus, $M'_1(\beta) = M'_1(\alpha)$,
and so $\SL(\alpha' + k\delta)$ starts with at least $p'$ copies of $\SL(M'_1(\alpha))$, as claimed.

\item[ii)]
Assume next that $\SL^{rs}(\alpha' + k\delta)$ starts with $\geq p'$ copies of $\SL(M'_1(\gamma))$, where
$\gamma = \deg(\SL^{rs}(\alpha' + k\delta))$. According to Lemma~\ref{lemma:leclerc.14} (an already established
inequality $w > \SL(M'_1(\alpha))$ implies that $w$ cannot be a proper prefix of $\SL(M'_1(\alpha))$) we see that
$\SL^{ls}(\alpha' + k\delta)$ also starts with $\geq p'$ copies of $\SL(M'_1(\gamma))$.
The latter reduces to the setup of i) by Lemma~\ref{lemma:M.prime.prefix}.

\end{enumerate}

\end{enumerate}
This completes the inductive step, and hence concludes the proof.
\end{proof}

Let us record two important corollaries of the above proposition.

\begin{corollary}\label{cor:tight.bound.M.prime}
For any $\chain(\alpha)\in \mathcal{D}$ with $M'_1(\alpha) = (\delta,i)$, we have $\alpha + k\delta < (\delta,i-1)$ for $k \gg 0$.
\end{corollary}

\begin{remark}\label{rem:M'1-role}
This result together with Lemma~\ref{lemma:dec.chains.lower.bound} signifies the importance of $M'_1$ instead of $M_1$,
as we get $\SL_i(\delta)<\SL(\hat{\alpha})<\SL_{i-1}(\delta)$ for all $\hat{\alpha}\in \chain(\alpha)$ with $|\hat\alpha|\gg 1$.
\end{remark}

\begin{corollary}\label{cor:M.prime.bounding}
For any $\chain(\alpha)\ne \chain(\beta) \in \mathcal{D}$ with $M'_1(\alpha) = M'_1(\beta)$, the chain $\chain(\alpha)$
is not bounded below by any element of $\chain(\beta)$.
\end{corollary}

\begin{proof}
Suppose the contrary, that is $\chain(\alpha)>\hat{\beta}$ for some $\hat{\beta}\in \chain(\beta)$.
Let $\SL(\hat{\beta}) = \underbrace{\SL(M'_1(\beta))}_{p \text{ times}} w$ for some $p\geq 0$, as
in~\eqref{eqn:imaginary.prefix.chunks}. According to Proposition~\ref{prop:weak.form.desc.irr.chains}, there is
$\hat{\alpha}\in \chain(\alpha)$ such that $\SL(\hat{\alpha})$ starts with $(p + 1)$ repeated occurrences of
$\SL(M'_1(\alpha))$. Then $\hat{\alpha}<\hat{\beta}$ by Proposition~\ref{prop:weak.form.desc.irr.chains}
(as $w>\SL(M'_1(\beta))=\SL(M'_1(\alpha))$), a contradiction.
\end{proof}

Let us now investigate further the first case from the proof of Proposition~\ref{prop:weak.form.desc.irr.chains}.

\begin{lemma}\label{lemma:decreasing.left.standard.imaginary.chain}
If $\chain(\alpha) \in \mathcal{D}$ and $\SL^{ls}(\alpha) = \SL(M_1(\alpha))$, then $M'_1(\alpha) = M_1(\alpha)$ and
\begin{equation}\label{eq:prefix-periodicity}
  \SL(\alpha' + k\delta) = \underbrace{\SL(M'_1(\alpha))}_{k\text{ times}}\SL(\alpha')
  \quad \mathrm{and} \quad \SL^{ls}(\alpha' + k\delta) = \SL(M'_1(\alpha))
\end{equation}
for all $k \in \BZ_{\geq 0}$, where as before $\alpha'$ denotes the shortest element of $\chain(\alpha)$.
\end{lemma}

\begin{proof}
The equality $M'_1(\alpha) = M_1(\alpha)$ follows directly from Lemma~\ref{lemma:M.prime.prefix}.

Next, let us show that $\SL(\alpha - \delta)$ is of the form~\eqref{eq:prefix-periodicity}. If $|2\delta| > |\alpha| > |\delta|$,
then this is clear as $\alpha-\delta=\alpha'$. If $|\alpha| > |2\delta|$, then $\SL(\alpha-\delta) = \SL(M'_1(\alpha))w$ for some
$w$ by Lemma~\ref{lemma:leclerc.14}, and we claim that this splitting is actually the standard factorization. Indeed,
if it was not, then $\SL^{ls}(\alpha - \delta) = \SL(M'_1(\alpha))u$ for some $u > \SL(M'_1(\alpha))$, hence,
$\SL(M'_1(\alpha))\SL^{ls}(\alpha - \delta)\in \LL$, a contradiction with $\SL(M'_1(\alpha))$ being the longest proper Lyndon
prefix of $\SL(\alpha)$. Thus $\SL^{ls}(\alpha-\delta) = \SL(M'_1(\alpha))$, and iterating this argument we
derive~\eqref{eq:prefix-periodicity} whenever $|\alpha' + k\delta| \leq |\alpha|$.

It remains to show that if $\SL^{ls}(\alpha) = \SL(M'_1(\alpha))$ then $\SL^{ls}(\alpha + \delta) = \SL(M'_1(\alpha))$.
Assuming the contradiction and using that $\SL(\alpha + \delta)$ starts with $\SL(M'_1(\alpha))$ by
Proposition~\ref{prop:weak.form.desc.irr.chains}, we get $\SL^{ls}(\alpha + \delta)=\SL(M'_1(\alpha))w$ for some nonempty $w$.
Let $\beta = \deg(\SL^{ls}(\alpha + \delta))$ so that $|\beta| > |\delta|$ and $\beta$ is real due to
Corollary~\ref{cor:imaginary.suffix.prefix}. Moreover, we have
$\SL(\beta-\delta) > \SL(\beta) > \SL(M'_1(\alpha))=\SL^{ls}(\alpha)$, since $\chain(\beta)$ is decreasing by
Lemma~\ref{lemma:left.standard.decreasing}. Thus $\beta - \delta > \alpha$ by Lemma~\ref{lemma:equiv.to.standard.fac}.
Additionally we have $\alpha - \beta + \delta = \deg(\SL^{rs}(\alpha + \delta)) > \beta > M'_1(\alpha)$, hence
$\alpha - \beta + \delta > \alpha$ by Lemma~\ref{lemma:equiv.to.standard.fac}.
As $\alpha = \beta- \delta + (\alpha - \beta + \delta)$, we get a contradiction with Theorem~\ref{thm:convexity}.
Hence $\SL^{ls}(\alpha + \delta) = \SL(M_1(\alpha))$, and iterating this argument we derive~\eqref{eq:prefix-periodicity}
whenever $|\alpha' + k\delta| > |\alpha|$.
\end{proof}

We also record the following property of the above setup.

\begin{corollary}\label{cor:decreasing.left.factor.imaginary}
For any $\alpha \in \wDelta^{+,\re}$ such that $\chain(\alpha)\in \mathcal{D}$, if $\SL^{ls}(\alpha) = \SL(M_1(\alpha))$,
then $s_\alpha = 1$.
\end{corollary}

\begin{proof}
Assuming the contrary, i.e.\ $s_\alpha > 1$, there are $\chain(\beta),\chain(\gamma) \in \mathcal{D}$ such that
$\chain(\beta) + \chain(\gamma) = \chain(\alpha)$ and $M'_1(\beta) = M'_1(\gamma) = M'_1(\alpha)$ according to
Lemma~\ref{lemma:s.splitting}. We also note that $|\alpha| > |\delta|$. Thus, we can express $\alpha = \hat{\beta} + \hat{\gamma}$
for some $\hat{\beta}\in \chain(\beta)$, $\hat{\gamma}\in \chain(\gamma)$. Without loss of generality, we shall assume
that $\hat{\beta} < \hat{\gamma}$. According to Lemma~\ref{lemma:dec.chains.lower.bound} and
$M'_1(\hat{\beta}) = M'_1(\hat{\gamma}) = M'_1(\alpha)$, we have $M'(\alpha) < \hat{\beta} < \hat{\gamma}$.
However, this implies $\SL^{ls}(\alpha) = \SL(M_1(\alpha)) \leq \SL(M'_1(\alpha)) < \SL(\hat{\beta})$ by
Lemma~\ref{lemma:M.prime.bound}, hence $\alpha < \hat\beta,\hat\gamma$ by Lemma~\ref{lemma:equiv.to.standard.fac},
contradicting Theorem~\ref{thm:convexity}.
\end{proof}


\subsection{$u$-function}
\

The next handful of lemmas will allow us to determine when imaginary words start appearing as subwords in decreasing chains.

\begin{definition}\label{def:u}
For any $\chain(\alpha) \in \mathcal{D}$, define $u(\chain(\alpha))=u(\alpha) \in \chain(\alpha)$ to be the longest element
such that $\SL(u(\alpha))$ does not contain $\SL(M'_1(\alpha))$ as a prefix (equivalently, does not contain $\SL(M'_1(\alpha))$
as a subword, by Proposition~\ref{prop:imaginary.substring.decreasing}).
\end{definition}

While $u$ is the primary object of interest for us, the following extension is useful for computing the value of $u$.
Let $C_i = \{\chain(\alpha) \in \mathcal{D} \,|\, M'_1(\alpha) \leq (\delta,i), \alpha' > (\delta,i)\}$.
Define the function $u_i\colon C_i \to \wDelta^{+,\re}$ via:
\begin{gather*}
  u_i(\chain(\alpha)) = u_i(\alpha) :=
  \max_{|\cdot|}\left\{\hat\alpha \in \chain(\alpha) \,\bigg|\,
        \substack{\SL(\hat\alpha) > \SL_{i}(\delta) \ \mathrm{and} \\
                  \SL(\hat\alpha) \text{ does not have } \SL_i(\delta) \text{ as a prefix}}\right\}.
\end{gather*}
Note that this set is non-empty by definition of $C_i$, and is finite by Proposition~\ref{prop:weak.form.desc.irr.chains}.
We also note that if $\chain(\alpha) \in C_i,C_j$ for $i < j$, then $|u_i(\alpha)| \geq |u_j(\alpha)|$.
Finally:
\begin{equation}\label{eq:u-vs-ui}
  u_i(\alpha) = u(\alpha) \qquad \mathrm{if} \quad  M'_1(\alpha) = (\delta,i).
\end{equation}

\noindent
Though we cannot say much about $\SL(\hat\alpha)$ for $\hat{\alpha}\in \chain(\alpha)$ with $|\hat\alpha| \leq |u(\alpha))|$,
we note that the value of $u(\alpha)$ can be computed iteratively via the following simple result
(henceforth, we identify $u(\alpha)$ with one of $u_i(\alpha)$ through~\eqref{eq:u-vs-ui}).

\begin{lemma}\label{lemma:alt.u.def}
For any $\chain(\alpha) \in C_i$, we have:
\begin{gather}\label{eqn:alt.u.def}
  u_i(\alpha) = \max_{|\cdot|}\left(\left\{u_i(\beta) + u_i(\gamma) \,\bigg\vert\,
  \substack{\chain(\beta),\chain(\gamma) \in C_i\\
  \chain(\beta) + \chain(\gamma) = \chain(\alpha)}\right\} \cup \{\alpha'\}\right),
\end{gather}
where as before $\alpha'$ denotes the shortest element of $\chain(\alpha)$.
\end{lemma}

\begin{proof}
We will prove this by induction on $\hgt(\chain(\alpha))$. The base case $\hgt(\chain(\alpha)) = 1$ corresponds
to $\chain(\alpha)$ being an \irrchain. Then the first set in the RHS of~\eqref{eqn:alt.u.def} is empty and
$u_i(\alpha) = \alpha'$ by Corollary~\ref{cor:dec.irr.chains}, completing the base case.

As for the inductive step, let us assume the claim holds for all decreasing chains of relative height $< \hgt(\chain(\alpha))$.
We claim first that $|u_i(\alpha)| \geq |\hat{\alpha}|$, where $\hat{\alpha}$ is the RHS of~\eqref{eqn:alt.u.def}. This is
obvious if $\hat{\alpha} = \alpha'$. Otherwise let $u_i(\beta) = \hat\beta, u_i(\gamma) = \hat\gamma$ be two elements which
attain the max, so that $\hat{\alpha}=\hat{\beta}+\hat{\gamma}$. Assuming without loss of generality that $\hat\beta < \hat\gamma$,
we get $\hat\beta < \hat\beta + \hat\gamma < \hat\gamma$ by Theorem~\ref{thm:convexity}. Since $\SL(\hat\beta)$ does not contain
$\SL_i(\delta)$ as a prefix and $\SL(\hat\beta) > \SL_i(\delta)$, then $\SL(\hat{\alpha}) > \SL_i(\delta)$ and $\SL(\hat{\alpha})$
does not contain $\SL_i(\delta)$ as a prefix, establishing $|u_i(\alpha)| \geq |\hat\alpha|$.

To show $|u_i(\alpha)| \leq |\hat\alpha|$, consider $\mu = \deg(\SL^{ls}(u(\alpha)))$ and
$\eta = \deg(\SL^{rs}(u(\alpha)))$. We note that $\eta \not \in \chain(\alpha)$, since otherwise $\mu = M_1(\chain(\alpha))$
by Corollary~\ref{cor:imaginary.suffix.prefix}, and so $\alpha' = u_i(\alpha)$
by Lemma~\ref{lemma:decreasing.left.standard.imaginary.chain}. Thus $\chain(\mu),\chain(\eta) \in \mathcal{D}$ by
Lemma~\ref{lemma:left.standard.decreasing}, and $\SL(\mu) > \SL_i(\delta)$ by Lemma~\ref{lemma:equiv.to.standard.fac}.
Since $\SL(\mu)$ does not contain $\SL_i(\delta)$ as a prefix, we note that $\SL(\eta) > \SL_i(\delta)$ and $\SL(\eta)$ cannot
contain $\SL_i(\delta)$ as a prefix by Lemma~\ref{lemma:leclerc.14}. Hence, $|\mu| \leq |u_i(\mu)|$, $|\eta| \leq |u_i(\eta)|$,
and so $|u_i(\alpha)|= |\mu + \eta| \leq |\hat\alpha|$.

This completes the proof of the lemma.
\end{proof}

We note the following simple property of $u(\cdot)$.

\begin{lemma}\label{lem:aux_unused}
For any $\alpha$ with $\chain(\alpha) \in \mathcal{D}$ and $|\alpha| \geq |u(\alpha)|$, consider any splitting
$\SL(\alpha) = \SL(\beta)\SL(\gamma)$. If $M'_1(\gamma) < M'_1(\alpha) = (\delta,i)$, then $\gamma = u_i(\gamma)$.
\end{lemma}

\begin{proof}
By Lemma~\ref{lemma:dec.chains.lower.bound}, we have $ (\delta,i) = M'_1(\alpha) < \alpha < \gamma$, which together with
$M'_1(\gamma) < M'_1(\alpha)$ implies that $\chain(\gamma) \in C_i$ and $|\gamma| \leq |u_i(\gamma)|$. We note that $\beta$
must be real as $M'_1(\gamma) < M'_1(\alpha)$. Moreover $\chain(\beta)\in \mathcal{D}$, since if $\chain(\beta)$ was
increasing, then $M'_1(\gamma) = \max\{M'_1(-\beta),M'_1(\alpha)\}$ by~\eqref{eq:Mprime-additive}, thus contradicting to
$M'_1(\gamma) < M'_1(\alpha)$. Henceforth, we shall assume the contradiction to to the claim: $|\gamma| < u_i(\gamma)$.

If $\alpha = u(\alpha)$, then we consider the appropriate concatenation of $\SL(\beta), \SL(\gamma + \delta)$:
\begin{itemize}[leftmargin=0.7cm]

\item
If $\SL(\beta) < \SL(\gamma + \delta)$, then $\SL(\hat\alpha + \delta) \geq \SL(\beta)\SL(\gamma + \delta)$ by Leclerc's algorithm.
However, $\SL(\beta)\SL(\gamma + \delta)$ does not contain $\SL(M'_1(\alpha))$ as a prefix: $\SL(\beta)$ does not have
$\SL(M'_1(\alpha))$ as a prefix nor it can be a prefix of $\SL(M'_1(\alpha))$ since  $M'_1(\alpha) < \beta$ by
Lemma~\ref{lemma:dec.chains.lower.bound}. Invoking that $\SL(\alpha + \delta)$ does have a prefix $\SL(M'_1(\alpha))$,
we thus obtain $\SL(\alpha + \delta) < \SL(\beta)\SL(\gamma + \delta)$, a contradiction;

\item
If $\SL(\gamma + \delta) < \SL(\beta)$, then we have $\SL(\alpha + \delta) \geq \SL(\gamma + \delta)\SL(\beta)$ by Leclerc's
algorithm. But $|\gamma + \delta| > |\delta|$ and $\SL(\gamma + \delta)$ does not have $\SL(M'_1(\alpha))$ as a prefix since
that would contradict $M'_1(\gamma) \ne M'_1(\alpha)$ by Lemma~\ref{lemma:M.prime.prefix}.
Thus, we again obtain a contradiction, exactly as in the previous case.

\end{itemize}
Assume now that $\alpha > u(\alpha)$, hence, $\SL(\alpha)$ has $\SL(M'_1(\alpha))$ as a prefix. We then claim that $\SL(\beta)$
also has $\SL(M'_1(\alpha))$ as a prefix. Indeed if $\SL(\beta)$ was a proper prefix of $\SL(M'_1(\alpha))$, then
$\chain(\beta) \in \mathcal{I}$ by part (3) of Corollary~\ref{cor:lyndon.prefix.suffix.chains}, a contradiction.
We then consider the appropriate concatenation of $\SL(\beta - \delta)$ and $\SL(\gamma + \delta)$:
\begin{itemize}[leftmargin = 0.7cm]

\item
If $\SL(\beta - \delta) < \SL(\gamma + \delta)$, then invoking $\beta - \delta > \beta$, we obtain
$\gamma +\delta > \beta - \delta > \alpha$ contradicting to Theorem~\ref{thm:convexity};

\item
If $\SL(\gamma + \delta) < \SL(\beta - \delta)$, then $\SL(\alpha) \geq \SL(\gamma + \delta)\SL(\beta - \delta)$ by Leclerc's
algorithm. Since we assumed that $|\gamma| < |u_i(\gamma)|$, we have $\gamma + \delta > (\delta,i) = M'_1(\alpha)$, and as
$\SL(M'_1(\alpha))$ is not a prefix of $\SL(\gamma + \delta)$ by Lemma~\ref{lemma:M.prime.prefix}, we obtain
$\beta - \delta > \gamma + \delta > \alpha$, which again contradicts to Theorem~\ref{thm:convexity}.

\end{itemize}
Thus in either case we obtain a contradiction, and so we must have $\gamma = u_i(\gamma)$.
\end{proof}

The following upgrade of Lemma~\ref{lemma:s.splitting} is needed for inductive arguments below.

\begin{lemma}\label{lemma:u.splitting}
For any $\chain(\alpha) \in \mathcal{D}$ with $s_\alpha \geq 2$ there are $\chain(\beta),\chain(\gamma) \in \mathcal{D}$
such that $\chain(\beta) + \chain(\gamma) = \chain(\alpha)$, $M'_1(\beta) = M'_1(\gamma) = M'_1(\alpha)$, and
$u(\beta) + u(\gamma) = u(\alpha)$.
\end{lemma}

\begin{proof}
We shall proceed by induction on $\hgt(\chain(\alpha))$. The base case is when $\chain(\alpha)$ is the sum of two irreducible
chains: $\chain(\alpha)=\chain(\beta_i)+\chain(\beta_j)$. We must then have $M'_1(\beta_i) = M'_1(\beta_j) = M'_1(\alpha)$ as
$s_\alpha = 2$, and the result follows from Lemma~\ref{lemma:alt.u.def}.

Assume now the result holds for all decreasing chains of the relative height $<\hgt(\chain(\alpha))$.
Let $M'_1(\alpha)=(\delta,i)$ and pick $\hat\beta,\hat\gamma$ which achieve the max in~\eqref{eqn:alt.u.def}.
If $M'_1(\beta) = M'_1(\gamma) = M'_1(\alpha)$, then the claim follows. Otherwise, without loss of generality,
we can assume that $M'_1(\beta) < M'_1(\alpha)$, and so $M'_1(\gamma) = M'_1(\alpha) = (\delta,i)$ by~\eqref{eq:Mprime-additive}.
Then $s_\gamma = s_\alpha$, and by the inductive hypothesis there exist $\chain(\mu),\chain(\nu) \in \mathcal{D}$ such that
$\chain(\mu) + \chain(\nu) = \chain(\gamma)$, $M'_1(\mu) = M'_1(\nu) = M'_1(\gamma)=(\delta,i)$, and $u(\mu) + u(\nu) = u(\gamma)$.
We must then have that either $\chain(\mu) + \chain(\beta)$ or $\chain(\nu) + \chain(\beta) \in \mathcal{D}$ by the same logic
as that used in the end of the proof of Lemma~\ref{lemma:s.splitting}. Assume without loss of generality that
$\chain(\mu) + \chain(\beta) \in \mathcal{D}$. Note $M'_1(\mu + \beta) = M'_1(\alpha)$ by~\eqref{eq:Mprime-additive}.
Applying Lemma~\ref{lemma:alt.u.def} (with $u=u_i$), we have $|u(\mu)| + |u(\beta)| \leq |u(\mu + \beta)|$. Thus we get
$|u(\mu + \beta)| + |u(\nu)| \geq |u(\mu)| + |u(\beta)| + |u(\nu)| = |u(\beta)| + |u(\gamma)| = |u(\alpha)|$, but we must we have
the equality by Lemma~\ref{lemma:alt.u.def}. Thus $\chain(\mu + \beta)$ and $\chain(\nu)$ satisfy the requirements of the lemma,
completing the step of induction.
\end{proof}


\subsection{Periodicity}
\

Our main goal is to describe the structure of decreasing chains $\chain(\alpha)$ once imaginary prefixes start appearing,
that is,  for $|\hat\alpha| > |u(\alpha)|$. To this end, we introduce some new terminology: given a standard Lyndon word of
the form $uvw$, where $v$ is of the form $\underbrace{\SL_i(\delta)}_{k \text{ times}}$ for some $i$ and $k \geq 1$, $u$ does
not end with $\SL_i(\delta)$, and $w$ does not start with $\SL_i(\delta)$,  we will refer to $v$ as an \textbf{imaginary chunk}
or just a \emph{chunk}.

The following result gives an upper bound on the number of chunks in a given standard Lyndon word
associated with a decreasing chain.

\begin{lemma}\label{lemma:form.dec.words}
For any $\chain(\alpha) \in \mathcal{D}$ and $k \in\BZ_{> 0}$, we have
\begin{gather}\label{eqn:form.of.dec.word}
  \SL(u(\alpha) + k\delta) =
  \underbrace{\SL_i(\delta)}_{p_1 \text{ times}} w_1 \underbrace{\SL_i(\delta)}_{p_2 \text{ times}}w_2 \dots
  \underbrace{\SL_i(\delta)}_{p_{n} \text{ times}}w_{n}
\end{gather}
where $M'_1(\alpha) = (\delta,i)$, $n \leq s_\alpha$, $p_j > 0$ for all $1\leq j\leq n$,
each $w_j$ does not contain $\SL_i(\delta)$ as a subword, and any non-empty suffix of $w_j$ is $>\SL_i(\delta)$.
\end{lemma}

\begin{proof}
This will be shown by induction on $\hgt(\chain(\alpha))$. The base case corresponds to $\chain(\alpha)$ being irreducible
for which the result follows from Corollary~\ref{cor:dec.irr.chains}. Here, the fact that any suffix of $w_1$ is $> \SL_i(\delta)$
follows from $\SL(u(\alpha)+k\delta)$ being Lyndon, while the fact that $w_1$ does not have $\SL_i(\delta)$ as a subword is due to
Proposition~\ref{prop:imaginary.substring.decreasing}.

For the inductive step, suppose the result holds for all chains of relative height $ < \hgt(\chain(\alpha))$.
Let $\hat\alpha  = u(\alpha) + k\delta$ for some $k \in \BZ_{> 0}$, and let $\SL(\hat\alpha) = \SL(\beta)\SL(\gamma)$ be
the standard factorization, which does not split any of $\SL_i(\delta)$ by Lemma~\ref{lemma:no.splitting.decreasing}.
If $\beta$ is imaginary, then $\beta = M_1(\gamma)$ by Corollary~\ref{cor:imaginary.suffix.prefix}, and the result follows
from Lemma~\ref{lemma:decreasing.left.standard.imaginary.chain}.
If $\beta$ is real, then $\SL(\beta)$ contains $\SL_i(\delta)$ as a proper prefix of $\SL(\beta)$, hence,
$M'_1(\beta) = M'_1(\alpha)$ by Lemma~\ref{lemma:M.prime.prefix}.

We will first establish $n \leq s_\alpha$. To this end, it suffices to show that there are $\leq s_\alpha$ chunks of
$\SL(M'_1(\alpha))$'s in the words $\SL(\beta)$ and $\SL(\gamma)$ combined, due to Lemma~\ref{lemma:no.splitting.decreasing}:
\begin{itemize}[leftmargin=0.7cm]

\item
if $M'_1(\gamma) = M'_1(\alpha)$, then $s_\alpha = s_\beta + s_\gamma$. If $|\gamma| > |u(\gamma)|$, then the result
follows by applying the inductive hypothesis to $\SL(\beta)$ and $\SL(\gamma)$, while if $|\gamma| \leq |u(\gamma)|$,
then the result follows from the inductive hypothesis applied to $\SL(\beta)$ alone;

\item
if $M'_1(\gamma) \neq M'_1(\alpha)$, then $s_\beta=s_\alpha$ and $\SL_i(\delta)$ is not a subword of $\SL(\gamma)$ by
Proposition~\ref{prop:imaginary.substring.decreasing}, hence, the result follows by applying the inductive hypothesis
applied to $\SL(\beta)$.

\end{itemize}

We will now prove the two stated propertied of $w_j$'s. If the standard factorization of $\SL(\hat\alpha)$ does not split any $w_j$,
then the result is clear by applying the inductive hypothesis to $\SL(\beta), \SL(\gamma)$. If some $w_j$ actually got split by
the standard factorization, then we claim that $j=n$: this follows by noting that otherwise $\SL_i(\delta)$ would be a subword
of $\SL(\gamma)$ and hence would also be its prefix by Proposition~\ref{prop:imaginary.substring.decreasing}, a contradiction.
By the inductive hypothesis applied to $\SL(\beta)$, we obtain both claimed properties for $\{w_j\}_{j=1}^{n-1}$.
Finally, given a non-empty suffix $u$ of $w_n$, either $u$ starts in $\SL(\beta)$ and thus $u>\SL_i(\delta)$ by the inductive
hypothesis applied to $\SL(\beta)$, or $u$ is a suffix of $\SL(\gamma)$ and so $u \geq \SL(\gamma)>\SL(\beta)>\SL_i(\delta)$.
This completes the proof.
\end{proof}

We also note the following lower bound on $\delta$-words occurring in decreasing chains.

\begin{lemma}\label{lem:imaginary-lowerbound}
For any $\chain(\alpha) \in \mathcal{D}$ and $k>0$, the word $\SL(u(\alpha)+k\delta)$ contains at least $k$
copies of $\SL(M'_1(\alpha))$ as subwords.
\end{lemma}

\begin{proof}
We will prove this by induction on the relative height $\hgt(\chain(\alpha))$. The base case $\hgt(\chain(\alpha))=1$
corresponds to an irreducible decreasing chain $\chain(\alpha)$, and the result follows from
Corollary~\ref{cor:dec.irr.chains}. For the step of induction, we shall assume the result holds for all decreasing chains of
relative height $<\hgt(\chain(\alpha))$.

Consider the standard factorization $\SL(\alpha) = \SL(\beta)\SL(\gamma)$. If $\beta$ is imaginary, then $s_\alpha = 1$ by
Corollaries~\ref{cor:decreasing.left.factor.imaginary} and~\ref{cor:imaginary.suffix.prefix}, reducing to the
base case. If $\beta$ is real, then $\chain(\beta),\chain(\gamma) \in \mathcal{D}$ by Lemma~\ref{lemma:left.standard.decreasing}.
Let $M'_1(\alpha) = (\delta,i)$. Note that $M'_1(\beta) = M'_1(\alpha)$ by Lemma~\ref{lemma:M.prime.prefix}, thus
$\chain(\beta) \in C_i$ by Lemma~\ref{lemma:dec.chains.lower.bound}. Additionally, we must have $\chain(\gamma) \in C_i$ as
$M'_1(\gamma) \leq M'_1(\alpha)=(\delta,i) < \beta < \gamma$. Thus applying Lemma~\ref{lemma:alt.u.def}, we have
$|u(\alpha)| = |u_i(\alpha)| \geq |u_i(\beta)| + |u_i(\gamma)|$, and so $|\alpha| \geq |u_i(\beta)| + |u_i(\gamma)| + |k\delta|$.
Let $\beta = u_i(\beta) + k_1\delta,\gamma = u_i(\gamma) + k_2\delta$, so that $k_1 + k_2 = k$. We consider two cases:
\begin{itemize}[leftmargin=0.7cm]

\item
if $M'_1(\gamma) < M'_1(\alpha)$, then $k_2 \leq 0$ (as otherwise we would get a contradiction with $\gamma > \beta$) and by the
induction assumption the word $\SL(\beta)$ has at least $k_1 \geq  k$ copies of $\SL(M'_1(\beta))=\SL(M'_1(\alpha))$ as a subword,
hence so does $\SL(\alpha)$;

\item
if $M'_1(\gamma) = M'_1(\alpha)$, then by the induction assumption the words $\SL(\beta)$ and $\SL(\gamma)$ contain at least $k_1$
and $k_2$ copies of $\SL(M'_1(\alpha))$, respectively, so that $\SL(\alpha)=\SL(\beta)\SL(\gamma)$ contains at least $k_1 + k_2 \geq k$
copies of $\SL(M'_1(\alpha))$.

\end{itemize}
This completes the proof.
\end{proof}

The following result provides a lower bound on the length of the first chunk.

\begin{corollary}\label{cor:p.1.value}
For any $\chain(\alpha) \in \mathcal{D}$ and $\hat\alpha = u(\alpha) + k\delta$ with $k > 0$, we have
$p_1 \geq \lceil k /s_\alpha\rceil$ in~\eqref{eqn:form.of.dec.word}.
\end{corollary}

\begin{proof}
Let $M'_1(\alpha) = (\delta,i)$. By Lemma~\ref{lem:imaginary-lowerbound}, we have $\sum_{j=1}^{s_\alpha} p_j \geq k$.
Thus, by pigeonhole principal, it suffices to show $p_1 \geq p_j$ for all $1< j \leq s_\alpha$. Assuming the contradiction,
i.e.\ $p_1 < p_j$ for some $j$, we have
$\underbrace{\SL_i(\delta)}_{p_1 \text{ times}}w_1 > \underbrace{\SL_i(\delta)}_{p_j \text{ times}}$ as $w_1 > \SL_i(\delta)$
and $\SL_i(\delta)$ is not a prefix of $w_1$ by Lemma~\ref{lemma:form.dec.words}. The latter contradicts to $\SL(\hat\alpha)$
being Lyndon, thus completing the proof.
\end{proof}

The following simple result will be important in the proof of Theorem~\ref{thm:dec.modulo.s}.

\begin{lemma}\label{lemma:lifting.chains.dec}
Consider four words $w,w',v,v'$ of the following form:
\begin{align*}
  w &= \underbrace{\ell}_{p_1 \text{ times}}w_1 \underbrace{\ell}_{p_2 \text{ times}}w_2 \ldots
      \underbrace{\ell}_{p_{t} \text{ times}}w_t , \\
  w'&= \underbrace{\ell}_{p_1+1 \text{ times}}w_1 \underbrace{\ell}_{p_2+1 \text{ times}}w_2 \ldots
      \underbrace{\ell}_{p_{t}+1 \text{ times}}w_t , \\
  v &= \underbrace{\ell}_{q_1 \text{ times}}v_1 \underbrace{\ell}_{q_2 \text{ times}}v_2 \ldots
      \underbrace{\ell}_{q_{s} \text{ times}}v_s , \\
  v'&= \underbrace{\ell}_{q_1+1 \text{ times}}v_1 \underbrace{\ell}_{q_2+1 \text{ times}}v_2 \ldots
      \underbrace{\ell}_{q_{s}+1 \text{ times}}v_{s} ,
\end{align*}
with $\ell \in \LL$, which satisfy the following conditions for all $i$:
\begin{enumerate}[leftmargin=0.8cm]

\item
$w_i$,$v_i$ are non-empty;

\item
$\ell$ is not a subword of $v_i,w_i$;

\item
any non-empty suffix of $v_i,w_i$ is $> \ell$;

\item
$p_i,q_i > 0$ for $i>1$.

\end{enumerate}
Then, we have $w > v \Leftrightarrow w' > v'$.
\end{lemma}

\begin{proof}
We will first show the ``$\Rightarrow$'' direction, by induction on $\max\{s,t\}$. The result in the base case $s = t = 1$
clearly holds as $\ell$ is inserted at the same spot for both $w,v$ to produce words $w',v'$. For the inductive hypothesis,
suppose the result holds for all $s',t'$ with $\max\{s',t'\} < \max \{s,t\}$. Note that when comparing two words we can
repeatedly remove common prefixes, hence we shall assume that $\min\{p_1,q_1\} = 0$. We consider three cases:
\begin{itemize}[leftmargin=0.7cm]

\item
If $p_1 > q_1 = 0$, then as $v_1 > \ell$ by (3) and $\ell$ is not a prefix of $v_1$ by (2), we get $v>w$, a contradiction;

\item
If $0 = p_1 < q_1$, then as $w_1 > \ell$ by (3) and $\ell$ is not a prefix of $w_1$ by (2), we get $w'>v'$, as claimed;

\item
Finally, suppose that $p_1 = q_1 = 0$. In this case, we consider three subcases:
\begin{itemize}[leftmargin=0.7cm]

\item
If $w_1 = v_1$, then the result follows from the inductive assumption applied to the pair of words
$\underbrace{\ell}_{p_2 \text{ times}}w_2\ldots$ and $\underbrace{\ell}_{q_2 \text{ times}}v_2\ldots$
(obtained from $w,v$ by erasing the common prefix $w_1=v_1$);

\item
If $w_1 < v_1$, then as $w > v$, we must then have that $w_1$ is a prefix of $v_1$, i.e.\ $v_1=w_1u$ for some non-empty $u$.
Since $p_2 > 0$ by (4), we then get a contradiction with $w>v$ as $u>\ell$ and $\ell$ is not a prefix of $u$ by (3) and~(2);

\item
If $w_1 > v_1$, then we clearly have $w' > v'$ unless $v_1$ is a prefix of $w_1$. In the latter case, $w_1=v_1u$ for some
non-empty $u$. Then we still have $w' > v'$ as $u>\ell$ and $\ell$ is not a prefix of $u$ by (3) and~(2), respectively.

\end{itemize}

\end{itemize}
This exhausts all possible cases, thus completing the inductive step.

We note that the ``$\Leftarrow$'' direction follows by applying the above to the contrapositive as all conditions (1)--(4)
are symmetric.
\end{proof}

In fact, the upper bound on the number of chunks in all decreasing chains from Lemma~\ref{lemma:form.dec.words} is robust,
due to the following key structural result which also features the main periodicity pattern.

\begin{theorem}\label{thm:dec.modulo.s}
For any $\chain(\alpha) \in \mathcal{D}$, let $M'_1(\alpha) = (\delta,i)$. Fix any $s_\alpha\leq k < 2s_\alpha$, and consider
the decomposition~\eqref{eqn:form.of.dec.word}. Then for any $k' = k + qs_\alpha$ with $q \in \BZ_{\geq 0}$:
\begin{itemize}[leftmargin=0.7cm]

\item[(1)]
We have
\begin{gather}\label{eq:chunk-periodicity-s}
  \SL(u(\alpha) + k'\delta) = \underbrace{\SL_i(\delta)}_{p_1 + q \text{ times}}w_1
  \underbrace{\SL_i(\delta)}_{p_2 + q \text{ times}} w_2 \, \cdots
  \underbrace{\SL_i(\delta)}_{p_{s_\alpha} +q\text{ times}}w_{s_\alpha}.
\end{gather}
Furthermore, all $p_j$ all equal to either $1$ or $2$ (for any $1 < j \leq s_\alpha$),
and $p_1$ is equal to $1$ when $k = s_\alpha$ and $2$ otherwise.

\item[(2)]
Let $\beta = \deg(\SL^{ls}(u(\alpha) + k\delta))$. If $\beta$ is imaginary then $\deg(\SL^{ls}(u(\alpha) + k'\delta)) = \beta$,
while if $\beta$ is real then $\deg(\SL^{ls}(u(\alpha) + k'\delta)) = \beta + qs_\beta \delta$.

\end{itemize}
\end{theorem}

\begin{proof}
This will be proven by induction on $\hgt(\chain(\alpha))$. The base case corresponds to irreducible chains $\chain(\alpha)$,
for which the claim follows from Corollary~\ref{cor:dec.irr.chains}. Before proceeding to the step of induction, we shall
first establish the following result:

\begin{claim}\label{claim:aux}
If Theorem~\ref{thm:dec.modulo.s} holds for $\chain(\beta),\chain(\gamma) \in \mathcal{D}$ with $M'_1(\beta) = M'_1(\gamma) = (\delta,i)$,
then inequality $\SL(u(\beta) + (k_1 + s_\beta)\delta) < \SL(u(\gamma) + (k_2 + s_\gamma)\delta)$ for some $k_1,k_2 \in \BZ_{\geq 0}$
holds if and only if $\SL(u(\beta) + (k_1+2s_\beta)\delta) < \SL(u(\gamma) + (k_2+2s_\gamma)\delta)$.
\end{claim}

\begin{proof}
This follows from Lemma~\ref{lemma:lifting.chains.dec}, taking $\ell = \SL_i(\delta)$. We note that conditions (2)--(3)
of the lemma follow from Lemma~\ref{lemma:form.dec.words}, while condition (4) is due to part (1) of the present theorem applied to
$\chain(\beta)$ and $\chain(\gamma)$.
\end{proof}

For the step of induction, we shall assume that the result holds for all chains of relative height $< \hgt(\chain(\alpha))$.
For part~(2), it suffices to show that for any $\hat\alpha \in \chain(\alpha)$ with $|\hat\alpha| \geq |u(\alpha) +  2s_\alpha\delta|$
we have $\deg(\SL^{ls}(\hat\alpha )) = \beta + s_\beta\delta$ if $\beta = \deg(\SL^{ls}(\hat\alpha -  s_\alpha\delta))$ is real,
or $\deg(\SL^{ls}(\hat\alpha)) = \beta$ if $\beta$ is imaginary. The latter case follows from
Lemma~\ref{lemma:decreasing.left.standard.imaginary.chain} and Corollary~\ref{cor:decreasing.left.factor.imaginary}.
We shall henceforth assume that $\beta$ is real.

We also need to show that $\SL(\hat\alpha)$ has $s_\alpha$ chunks in order to apply Lemma~\ref{lemma:lifting.chains.dec}.

\begin{claim}\label{claim:saturated.dec}
For any $\hat\alpha \in \chain(\alpha)$ with $|\hat\alpha| \geq |u(\alpha) + ks_\alpha\delta|$ and $M'_1(\alpha) = (\delta,i)$,
there are $s_\alpha$ chunks each of length at least $|k\delta|$ in $\SL(\hat\alpha)$. Moreover if
$\SL(\hat\alpha) = \SL(\hat\beta)\SL(\hat\gamma)$ is the standard factorization, then
$|\hat\beta| \geq |u(\hat\beta) + ks_{\hat\beta}\delta|$ and
$|\hat\gamma| \geq |u_i(\hat\gamma) + k(s_{\hat\alpha} - s_{\hat\beta})\delta|$.
\end{claim}

\begin{proof}
Recall that $u_i(\alpha) = u(\alpha)$ by~\eqref{eq:u-vs-ui}. We note that there are at most $s_\alpha$ chunks, due to
Lemma~\ref{lemma:form.dec.words}. Let $\SL(\hat\alpha) = \SL(\hat\beta)\SL(\hat\gamma)$ be the standard factorization.
Suppose there are less than $s_\alpha$ chunks in $\SL(\hat\alpha)$ or there is a chunk of length $< |k\delta|$. In accordance
with Lemma~\ref{lemma:no.splitting.decreasing}, then either $\SL(\hat\beta)$ has $< s_{\hat\beta}$ chunks of length at least
$|k\delta|$ or $\SL(\hat\gamma)$ has $< s_\alpha - s_{\hat\beta}$ chunks of length at least $|k(s_\alpha - s_{\hat\beta})\delta|$
in $\SL(\hat\gamma)$):
\begin{itemize}[leftmargin=0.7cm]

\item
Suppose the number of chunks in $\SL(\hat\beta)$ of length at least $|k\delta|$ is less than $s_{\hat\beta}$, so that
$|\hat\beta| < |u(\hat\beta) + ks_{\hat\beta}\delta|$ by the inductive hypothesis. We then consider two cases:
\begin{itemize}[leftmargin=0.7cm]

\item
if $M'_1(\hat\gamma) < M'_1(\hat\beta) = M'_1(\alpha)$, then $\hat\gamma = u_i(\hat\gamma)$ by Lemma~\ref{lem:aux_unused},
which contradicts to Lemma~\ref{lemma:alt.u.def} as $|u(\alpha)| = |u_i(\alpha)| < |u_i(\hat\beta)| + |u_i(\hat\gamma)|$;

\item
if $M'_1(\hat\gamma)  = M'_1(\hat\beta)$, then we must have $|\hat\gamma| \leq |u(\hat\gamma) + ks_{\hat\gamma}\delta|$, as
otherwise $\underbrace{\SL_i(\delta)}_{k+1 \text{ times}}$ would be a prefix of $\SL(\hat\gamma)$ but not of $\SL(\hat\beta)$,
contradicting to $\hat\beta < \hat\gamma$. Then we again get a contradiction with Lemma~\ref{lemma:alt.u.def}
as $|u(\alpha)| = |u_i(\alpha)| < |u_i(\hat\beta)| + |u_i(\hat\gamma)|$.

\end{itemize}

\item
Suppose the number of chunks in $\SL(\hat\gamma)$ of length at least $|k\delta|$ is less than $k(s_{\hat\alpha} - s_{\hat\beta})$.
Then $M'_1(\hat\gamma) = M'_1(\hat\beta)$ (as otherwise $s_{\hat\beta} = s_{\hat\alpha}$) and
$|\hat\gamma| < |u(\hat\gamma) + s_{\hat\gamma}\delta|$ by the inductive hypothesis. We consider two cases:
\begin{itemize}[leftmargin=0.7cm]

\item
if $|\hat\beta| \leq |u(\hat\beta) + ks_{\hat\beta}\delta|$, then we again get a contradiction with Lemma~\ref{lemma:alt.u.def},
since $|u(\alpha)| = |u_i(\alpha)| < |u_i(\hat\beta)| + |u_i(\hat\gamma)|$;

\item
if $|\hat\beta| > |u(\hat\beta) + ks_{\hat\beta} \delta|$, then $\underbrace{\SL_i(k\delta)}_{k+1 \text{ times}}$ is a prefix
of $\SL(\hat\beta)$ by the inductive hypothesis. Consider the appropriate concatenation of
$\SL(\hat\beta- \delta),\SL(\hat\gamma + \delta)$. If $\SL(\hat\beta - \delta) < \SL(\hat\gamma + \delta)$, then as
$\hat\beta - \delta > \hat\beta$ and $\SL(\hat\beta)$ is a prefix of $\SL(\hat\alpha)$, we obtain $\hat\beta - \delta > \hat\alpha$
so that $\SL(\hat\beta - \delta) \SL(\hat\gamma + \delta) > \SL(\hat\alpha)$, contradicting Leclerc's algorithm.
If $\SL(\hat\gamma + \delta) < \SL(\hat\beta - \delta)$, then as $\SL(\hat\gamma + \delta)$ has at most $k$ copies of
$\SL_i(\delta)$ as a prefix while $\SL(\hat\beta)$ (and hence $\SL(\hat\alpha)$) has at least $k+1$ copies of
$\SL_i(\delta)$ as a prefix (by the induction assumption), we obtain
$\SL(\hat\gamma + \delta) \SL(\hat\beta - \delta) > \SL(\hat\alpha)$, again contradicting Leclerc's algorithm.

\end{itemize}

\end{itemize}
This completes the proof.
\end{proof}

Let $\SL(\hat\alpha - s_\alpha\delta) = \SL(\beta)\SL(\gamma)$ be the standard factorization.
By Lemma~\ref{lemma:left.standard.decreasing}, $\gamma$ is real and as we assume $\beta$ is real we have
$\chain(\beta),\chain(\gamma)\in \mathcal{D}$. Additionally, we have $M'_1(\beta) = M'_1(\hat\alpha)$ by
Lemma~\ref{lemma:M.prime.prefix} as $\SL(M'_1(\hat\alpha))$ is a prefix of $\SL(\beta)$.
We shall now show that $\beta + s_\beta\delta < \hat\alpha$ through the following analysis:
\begin{itemize}[leftmargin=0.7cm]

\item
if $M'_1(\gamma) < M'_1(\beta)$, then $s_\beta = s_\alpha$ and $\beta + s_\beta\delta < \beta < \gamma$, so that
$\beta + s_\beta\delta < \hat\alpha < \gamma$ by Theorem~\ref{thm:convexity};

\item
if $M'_1(\gamma) = M'_1(\beta)$, then $s_\alpha=s_\beta+s_\gamma$ and $\beta+s_\beta\delta < \gamma + s_\gamma\delta$
by Claim~\ref{claim:aux}, where we note that $|\beta| \geq |u(\beta) + s_\beta\delta|, |\gamma| \geq |u(\gamma) + s_\gamma\delta|$
by Claim~\ref{claim:saturated.dec}.
Therefore $\beta + s_\beta\delta < (\beta + s_\beta\delta) + (\gamma + s_\gamma\delta) = \hat\alpha < \gamma  + s_\gamma\delta$
by Theorem~\ref{thm:convexity}.

\end{itemize}
We likewise consider the standard factorization $\SL(\hat\alpha)=\SL(\hat\beta)\SL(\hat\gamma)$. As $\beta$ is assumed to be real,
then $\hat\beta$ must also be real by Lemma~\ref{lemma:decreasing.left.standard.imaginary.chain}. Moreover, $\hat{\gamma}$ is real
and $\chain(\hat\beta),\chain(\hat\gamma)\in \mathcal{D}$ by Lemma~\ref{lemma:left.standard.decreasing}, and also
$M'_1(\hat\beta)=M'_1(\hat\alpha)$ by Lemma~\ref{lemma:M.prime.prefix}. Similarly to the above $\beta + s_\beta\delta < \hat\alpha$,
we shall now show that $\hat{\beta} - s_{\hat{\beta}}\delta < \hat\alpha - s_{\alpha}\delta$:
\begin{itemize}[leftmargin=0.7cm]

\item
if $M'_1(\hat\gamma) < M'_1(\hat\beta)$, then $s_{\hat\beta}=s_{\alpha}$ and $\hat\beta - s_{\hat\beta}\delta < \gamma$
since $\SL(\hat\beta-s_{\hat\beta}\delta)$ has a prefix $\SL(M'_1(\hat\beta))$ by definition of $u(\hat\beta)$, thus
$\hat\beta - s_{\hat\beta}\delta < \hat\alpha - s_\alpha\delta < \gamma$ by Theorem~\ref{thm:convexity};

\item
if $M'_1(\hat\gamma) = M'_1(\hat\beta)$, then $\hat\beta - s_{\hat\beta}\delta < \hat\gamma - s_{\hat\gamma}\delta$
by Claim~\ref{claim:aux}, where
  $|\hat\beta - s_{\hat\beta}\delta| \geq |u(\hat\beta) + s_{\hat\beta}\delta|,
   |\hat\gamma - s_{\hat\gamma}\delta| \geq |u(\hat\gamma) + s_{\hat\gamma}\delta|$
by Claim~\ref{claim:saturated.dec}, and so by Theorem~\ref{thm:convexity} we obtain
  $\hat\beta - s_{\hat\beta}\delta < (\hat\beta - s_{\hat\beta}\delta) + (\hat\gamma - s_{\hat\gamma}\delta)
   = \hat\alpha - s_\alpha\delta < \hat\gamma - s_{\hat\gamma}\delta$.

\end{itemize}
The above inequalities $\beta + s_\beta\delta < \hat\alpha$ and $\hat{\beta} - s_{\hat{\beta}}\delta < \hat\alpha - s_{\alpha}\delta$
imply $\hat\beta - s_{\hat\beta}\delta \leq \beta$ and $\beta+s_\beta\delta \leq \hat\beta$, due to emma~\ref{lemma:equiv.to.standard.fac}.
Since $M'_1(\hat\beta)=M'_1(\hat\alpha)=M'_1(\beta)$ and also
$|\hat\beta - s_{\hat\beta}\delta| \geq |u(\hat\beta) + s_{\hat\beta}\delta|$ as well as
$|\beta| \geq |u(\beta) + s_\beta\delta|$ by Claim~\ref{claim:saturated.dec},
we obtain $\hat\beta = \beta + s_\beta\delta$ by Claim~\ref{claim:aux}. This completes the proof of part~(2).

We shall now prove part (1). For $k \in \{s_\alpha,\ldots, 2s_\alpha - 1\}$, we note that $\SL(u(\alpha) + k\delta)$ has the
decomposition~\eqref{eqn:form.of.dec.word}. First, let us show that $p_1 = 1$ if $k = s_\alpha$ and $p_1 = 2$ if $k > s_\alpha$.
By Corollary~\ref{cor:p.1.value}, we note that $p_1 \geq 1$ for $k = s_\alpha$ and $p_1 \geq 2$ for $k > s_\alpha$.
We consider three cases:
\begin{itemize}[leftmargin=0.7cm]

\item
If $\SL_i(\delta) = \SL^{ls}(u(\alpha) + k\delta)$, then $s_\alpha = 1$ by Corollary~\ref{cor:decreasing.left.factor.imaginary},
so that $k=1$ and the result follows from Lemma~\ref{lemma:decreasing.left.standard.imaginary.chain};

\item
If $s_\alpha = 1$ (so that $k=1$), $\SL^{ls}(u(\alpha) + \delta)=\SL(\beta)$ is real, and $p_1 > 1$, then $\SL_i(\delta)\SL_i(\delta)$
must be a prefix of $\SL(\beta)$ by Lemma~\ref{lemma:no.splitting.decreasing}. Note that $|u(\alpha)| > |\delta|$. Let $\gamma,\eta$
be two elements from Lemma~\ref{lemma:alt.u.def} such that $u(\alpha) = u_i(\alpha) = \gamma + \eta$, so that
$\gamma = u_i(\gamma),\eta = u_i(\eta)$ and $\chain(\gamma),\chain(\eta)\in C_i$. As $s_\alpha=1$, without loss of generality
we have $M'_1(\gamma) = M'_1(\alpha)$ and $M'_1(\eta) < M'_1(\alpha)$. As $\SL(\gamma + \delta)$ has a prefix $\SL_i(\delta)$
while $\SL_i(\delta) < \SL(\eta)$ and $\SL_i(\delta)$ is not a prefix of $\SL(\eta)$ by definition of $u_i(\eta)$, we obtain
that $\SL(\gamma+\delta)\SL(\eta)$ is Lyndon. Since $\SL(\gamma + \delta)$ does not have $\SL_i(\delta)\SL_i(\delta)$ as
a prefix and $\SL(\gamma + \delta) = \SL_i(\delta)v$ for some $v > \SL_i(\delta)$, we get $\gamma + \delta > \beta$.
Thus $ \eta > \gamma + \delta > \beta$, so that $\eta,\gamma + \delta > u(\alpha) + k\delta$ by Lemma~\ref{lemma:equiv.to.standard.fac},
yielding a contradiction with Theorem~\ref{thm:convexity};

\item
If $s_\alpha > 1$, then choose $\chain(\beta),\chain(\gamma)$ as in Lemma~\ref{lemma:u.splitting}. Given any
$s_\alpha\leq k < 2s_\alpha$, we define $k_1 = \min(k - s_\alpha,s_\beta)$ and $k_2 = \max(0,(k - s_\alpha) - s_\beta)$.
Note that $s_\beta \leq s_{\beta}+k_1\leq 2s_\beta$, $s_\gamma \leq s_{\gamma}+k_2 < 2s_\gamma$, and
$u(\beta) + (s_{\beta}+ k_1)\delta + u(\gamma) + (s_\gamma + k_2)\delta = u(\alpha) + k\delta$. For the case $k = s_\alpha$,
if we suppose $p_1 > 1$, then applying the inductive hypothesis to $\chain(\beta),\chain(\gamma)$, we see that
$\SL(u(\beta) + s_\beta\delta),\SL(u(\gamma) + s_\gamma\delta)$ do not have prefix $\SL_i(\delta)\SL_i(\delta)$
(note that in this case $k_1 = k_2 = 0$), so that $u(\alpha) + k\delta < u(\beta) + s_\beta\delta, u(\gamma) + s_\gamma\delta$,
a contradiction with Theorem~\ref{thm:convexity}. For $s_\alpha < k < 2s_\alpha$, if $p_1 > 2$, then
$u(\alpha) + k\delta < u(\beta) + (s_\beta + k_1)\delta, u(\gamma) + (s_\gamma + k_2)\delta$, since neither of
$\SL(u(\beta) + (s_\beta + k_1)\delta),\SL(u(\gamma) + (s_\gamma + k_2)\delta)$ have $\underbrace{\SL_i(\delta)}_{3 \text{ times}}$
as a prefix by the inductive hypothesis, yielding again a contradiction with Theorem~\ref{thm:convexity}.

\end{itemize}
This establishes the claimed values of $p_1$. Since $1 \leq p_j \leq p_1$, with the first inequality due to
Claim~\ref{claim:saturated.dec} and the second due to $\SL(u(\alpha)+k\delta)$ being Lyndon (cf.\ the proof of
Corollary~\ref{cor:p.1.value}), we see that $p_j\in\{1,2\}$ for $2\leq j\leq s_\alpha$. This proves (1) for $q=0$.

As with part (2), it suffices to verify (1) for $\hat\alpha, \hat\alpha - s_\alpha\delta$ with
$|\hat\alpha| \geq |u(\alpha) + 2s_\alpha\delta|$. Consider the standard factorization
$\SL(\hat\alpha) = \SL(\hat\beta)\SL(\hat\gamma)$. We can assume that $\hat\beta$ is real, as otherwise the result follows
from Lemma~\ref{lemma:decreasing.left.standard.imaginary.chain}. We consider two cases:
\begin{itemize}[leftmargin=0.7cm]

\item
if $M'_1(\hat\gamma) \neq M'_1(\hat\alpha)$, then $s_{\hat\beta}=s_\alpha$ and the result follows from part (2)
and the inductive hypothesis for $\hat\beta$;

\item
if $M'_1(\hat\gamma) = M'_1(\hat\alpha)$, then $|\hat\gamma| \geq |u(\hat\gamma) + 2s_{\hat\gamma}\delta|$
by~\eqref{eq:u-vs-ui} and Claim~\ref{claim:saturated.dec}, and so the result follows from part (2) and the inductive
hypothesis for both $\hat\beta, \hat\gamma$.

\end{itemize}
This completes the inductive step.
\end{proof}

The structure of decreasing chains from Theorem~\ref{thm:dec.modulo.s} motivates the following:

\begin{definition}\label{def:periodicity-D}
We shall call $s_\alpha$ the \textbf{periodicity} of the chain $\chain(\alpha)\in \mathcal{D}$.
\end{definition}

We note that while it may be natural to expect the periodicity pattern of~\eqref{eq:chunk-periodicity-s} to appear starting
from $u(\alpha) + \delta$, as it was in type $A_n^{(1)}$ (see~\cite[Theorem 4.7]{AT}), this is not true in general as shown by
the following example (found through coding).

\begin{example}\label{ex:strong_period_counterex}
Consider the affine type $F_4^{(1)}$ with the order $0 < 2 < 4 < 1 <3$ (our labeling of simple roots
follows~\cite[\S4, Table Fin]{K}). For $\alpha = \alpha_1 + \alpha_2 + 2\alpha_3 + 2\alpha_4$, using the code
in Listing~\ref{lst:dec.per.counter}, we find for all $k\in \BZ_{\geq 0}$:
\begin{align*}
  \SL(\alpha) &= 233144, \\
  \SL(\alpha + \delta) &= \SL_1(\delta)342314, \\
  \SL(\alpha + 2\delta + 4k\delta) &=
    \underbrace{\SL_1(\delta)}_{k+1 \text{ times}}34 \underbrace{\SL_1(\delta)}_{k+1 \text{ times}}34
    \underbrace{\SL_1(\delta)}_{k \text{ times}}2 \underbrace{\SL_1(\delta)}_{k \text{ times}}1, \\
  \SL(\alpha + 3\delta + 4k\delta) &=
    \underbrace{\SL_1(\delta)}_{k +1 \text{ times}}1 \underbrace{\SL_1(\delta)}_{k \text{ times}}2
    \underbrace{\SL_1(\delta)}_{k +1 \text{ times}}34 \underbrace{\SL_1(\delta)}_{k + 1\text{ times}}34, \\
  \SL(\alpha + 4\delta + 4k\delta) &=
    \underbrace{\SL_1(\delta)}_{k+1 \text{ times}}2\underbrace{\SL_1(\delta)}_{k+1 \text{ times}}34
    \underbrace{\SL_1(\delta)}_{k+1 \text{ times}}34\underbrace{\SL_1(\delta)}_{k+1 \text{ times}}1, \\
  \SL(\alpha + 5\delta + 4k\delta) &=  \underbrace{\SL_1(\delta)}_{k+2 \text{ times}}34
  \underbrace{\SL_1(\delta)}_{k+1 \text{ times}}2\underbrace{\SL_1(\delta)}_{k+1 \text{ times}}34
  \underbrace{\SL_1(\delta)}_{k+1 \text{ times}}1,
\end{align*}
where $\SL_1(\delta) = 012334423312$ and $\chain(\alpha)\in \mathcal{D}$ by Lemma~\ref{lemma:monotonicity.smallest.once}.
In particular, the words $\SL(\alpha+\delta)$ and $\SL(\alpha+5\delta)=\SL(\alpha+\delta+s_{\alpha}\delta)$ are not
related through~\eqref{eq:chunk-periodicity-s}. Here, the equality $s_\alpha = 4$ follows from
$\delta - \beta_1=\alpha_2, \delta - \beta_2=\alpha_1, \delta - \beta_3=\alpha_3, \delta - \beta_4=\alpha_4$ and
the following table:
\begin{center}
\begin{tabular}{|c|cccc|}
  \hline
     $i$ &  1 & 2 & 3 & 4\\
  \hline
     $M_1(\beta_i)$& 1 & 1 & 1 & 3\\
  \hline
\end{tabular}
\end{center}
\end{example}

\begin{remark}
In type $A_n^{(1)}$, considered in~\cite{AT}, we can indeed extend Theorem~\ref{thm:dec.modulo.s} to a statement about
$u(\alpha) + k\delta$ for any $k\geq 1$, not just $k\ge s_\alpha$. The issue with extending this to other types arises from
an edge case in Lemma~\ref{lemma:lifting.chains.dec}, corresponding to the setup $w = yw_1w_2, v = yv_1$ with $w > v$ and
$v_1=w_1u$ with $u > y$, in which case $w' = yyw_1yw_2< yyv_1 = v'$. However, in type $A_n^{(1)}$, the comparisons that we
need to do in Theorem~\ref{thm:dec.modulo.s} do not have this edge case.
\end{remark}


\subsection{Corollaries and periodicity bounds}
\

The next few results provide additional information on the decomposition~\eqref{eqn:form.of.dec.word}. We start with
the following analogue of Claim~\ref{claim:saturated.dec} for the costandard factorization:

\begin{lemma}\label{lemma:saturated.dec.costfac}
For any $\alpha$ with $\chain(\alpha)\in \mathcal{D}$, $|\alpha| \geq |u(\alpha) + ks_\alpha\delta|$, and $s_\alpha > 1$,
its costandard factorization $\SL(\alpha) = \SL(\beta)\SL(\gamma)$ satisfies the following:
\begin{enumerate}[leftmargin=0.7cm]

\item
$\chain(\beta),\chain(\gamma) \in \mathcal{D}$;

\item
$M'_1(\beta) = M'_1(\gamma) = M'_1(\alpha)$;

\item
$|\beta| \geq |u(\beta) + ks_{\beta}\delta|$ and $|\gamma| \geq |u(\gamma) + ks_{\gamma}\delta|$;

\item
$\SL(\beta)$ and $\SL(\gamma)$ do not contain $\SL(M'_1(\alpha))$ as a suffix.

\end{enumerate}
\end{lemma}

\begin{proof}
As (1) $\Rightarrow$ (4) by Corollary~\ref{cor:lyndon.prefix.suffix.chains}, it suffices to show the other three.
Let $M'_1(\alpha) = (\delta,i)$. Since $s_\alpha > 1$ and $|\alpha| \geq |u(\alpha) + ks_\alpha\delta|$, we must have
that there are two separate imaginary chunks in $\SL(\alpha)$ by Theorem~\ref{thm:dec.modulo.s}. Then as $\SL_i(\delta)$
is a prefix of $\SL(\alpha)$, we must have that $\SL_i(\delta)$ is a proper prefix of $\SL(\gamma)$ by
Lemma~\ref{lemma:right.costfac.minimal} (it has to be proper since it is not a suffix of $\SL(\alpha)$).
This implies that $\chain(\gamma) \in \mathcal{D}$ and $M'_1(\gamma) = M'_1(\alpha)$ by
Corollary~\ref{cor:lyndon.prefix.suffix.chains} and Lemma~\ref{lemma:M.prime.prefix}.

We also note that $\SL_i(\delta)$ is a prefix of $\SL(\beta)$. If $\SL(\beta)\ne \SL_i(\delta)$, then we likewise get
$\chain(\beta) \in \mathcal{D}$ and $M'_1(\beta) = M'_1(\alpha)$, establishing (1)--(2).
We shall now prove $\SL(\beta)\ne \SL_i(\delta)$ by induction on the height $|\alpha|$. The base case corresponds to
$|\alpha| = |u(\alpha) + s_\alpha\delta|$, and if $\SL(\beta) = (\delta,i)$, then $\SL_i(\delta)\SL_i(\delta)$ would be
a prefix of $\SL(\alpha)$, contradicting Theorem~\ref{thm:dec.modulo.s}. For the inductive step, suppose that
$\SL(\beta) = \SL_i(\delta)$, but the result holds for $\gamma=\alpha-\delta$. Then (1)--(4) hold for the
costandard factorization $\SL(\gamma) = \SL(\mu)\SL(\eta)$. By Leclerc's algorithm, we must have
$\SL(\alpha) = \SL(\gamma + \delta) \geq \min\{\SL(\eta + \delta)\SL(\mu),\SL(\mu)\SL(\eta + \delta)\}$.
However, we note that
\begin{itemize}[leftmargin=0.7cm]

\item
if $\SL(\mu)\SL(\eta+\delta)$ is the above minimum, then it starts with one less repeated copy of $\SL_i(\delta)$
followed by $w_1>\SL_i(\delta)$, and is thus $>\SL(\alpha)$;

\item
if $\SL(\eta + \delta)\SL(\mu)$ is the above minimum, then it is $\geq \SL_i(\delta)\SL(\eta)\SL(\mu)$ and thus
is $>\SL_i(\delta)\SL(\mu)\SL(\eta)=\SL_i(\delta)\SL(\gamma)=\SL(\alpha)$.

\end{itemize}
This completes the proof of the above claim $\SL(\beta)\ne \SL_i(\delta)$, establishing (1)--(2).
Moreover, since $\SL(\beta), \SL(\gamma)$ contain $\SL_i(\delta)$ as a proper prefix, part (4) follows too.

Finally the proof of part (3) is completely analogous to that of Claim~\ref{claim:saturated.dec}
(as we never used the fact that we considered the standard factorization there).
\end{proof}

The $s_\alpha=1$ counterpart of the above result is treated in the following remark.

\begin{remark}
For any $\chain(\alpha) \in \mathcal{D}$ with $s_\alpha = 1$, the costandard factorization of $\SL(u(\alpha) + k\delta)$
with $k>1$ is $\SL(u(\alpha) + k\delta) = \SL(M'_1(\alpha))\SL(u(\alpha) + (k-1)\delta)$. To see this, we note that
$\SL(u(\alpha) + k\delta) = \SL(M'_1(\alpha)) \SL(u(\alpha) + (k-1)\delta)$, and $\SL(u(\alpha) + (k-1)\delta)$ is the
longest proper Lyndon suffix, since it starts with $\SL(M'_1(\alpha))$.
\end{remark}

We note that Lemma~\ref{lemma:saturated.dec.costfac} also allows to establish a version of Theorem~\ref{thm:dec.modulo.s}(2)
for the costandard factorization.

\begin{corollary}\label{cor:dec.costfac}
For any $\chain(\alpha) \in \mathcal{D}$ with $s_\alpha  > 1$, $q \in \BZ_{\geq 0}$ and $s_\alpha \leq k < 2s_\alpha$,
set $k' = k + qs_\alpha$. For the costandard factorization $\SL(u(\alpha) + k'\delta)=\SL(\beta)\SL(\gamma)$:
\begin{gather*}
    \SL(\beta - qs_\beta\delta)=\SL^l(u(\alpha) + k\delta) ,\qquad \SL(\gamma - qs_\gamma\delta)=\SL^r(u(\alpha)+k\delta).
\end{gather*}
\end{corollary}

\begin{proof}
According to Lemma~\ref{lemma:saturated.dec.costfac}, we note that $\chain(\beta),\chain(\gamma) \in \mathcal{D}$ and
$M'_1(\alpha) = M'_1(\beta) = M'_1(\gamma)$. The proof proceeds by induction on the height of $u(\alpha) + k'\delta$, with
the base case $k' = k = s_\alpha$ being obvious.

For the inductive step, suppose the result holds whenever $|\eta| < |u(\alpha) + k'\delta|$. It suffices to establish
$\SL^r(u(\alpha) + (k'-s_\alpha)\delta) = \SL(\gamma - s_\gamma\delta)$. First, we note that $\SL(\gamma - s_\gamma\delta)$
is a Lyndon suffix of $\SL(u(\alpha) + (k'-s_\alpha)\delta)$ by part~(1) of Theorem~\ref{thm:dec.modulo.s}.
Likewise, for the costandard factorization $\SL(u(\alpha) + (k'-s_\alpha)\delta) = \SL(\mu)\SL(\eta)$, we have
$|\gamma| \geq |u(\gamma) + s_\gamma \delta|$ and $|\eta| \geq |u(\eta) + s_\eta \delta|$ by Lemma~\ref{lemma:saturated.dec.costfac},
and so $\SL(\eta + s_\eta \delta)$ is a Lyndon suffix of $\SL(u(\alpha) + k'\delta)$. Thus
$\eta \leq \gamma - s_\gamma\delta$ and $\gamma \leq \eta + s_\eta\delta$ by Lemma~\ref{lemma:right.costfac.minimal}.
However, invoking Claim~\ref{claim:aux}, we get $\eta \leq \gamma - s_\gamma \delta \Leftrightarrow \eta  + s_\eta \delta \leq \gamma$,
and so we must have the equality $\eta = \gamma - s_\gamma\delta$. This completes the inductive step.
\end{proof}

We shall now derive some basic properties of the subwords $w_i$ in~\eqref{eq:chunk-periodicity-s}.

\begin{corollary}\label{cor:dec.w.form}
For any $\chain(\alpha) \in \mathcal{D}$, consider the subwords $w_i$ as in~\eqref{eq:chunk-periodicity-s}, that is,
subwords connecting chunks in $\SL(u(\alpha)+k\delta)$ for $k \geq s_\alpha$. For all $i \in \{1,\ldots, s_\alpha\}$:
\begin{enumerate}[leftmargin=0.8cm]

\item
$\deg(w_i)$ is real;

\item
$\chain(\deg(w_i)) \in \mathcal{D}$;

\item
$M'_1(\deg(w_i)) = M'_1(\alpha)$;

\item
$s_{\deg(w_i)} = 1$;

\item
$\deg(w_i) = u(\deg(w_i))$.

\end{enumerate}
\end{corollary}

\begin{proof}
We will prove this by induction on $s_\alpha$. The base case $s_\alpha = 1$ follows from the equality $\deg(w_1) = u(\alpha)$,
due to Theorem~\ref{thm:dec.modulo.s}. For the induction step, assume the result holds for all $\chain(\beta)\in \mathcal{D}$
with $s_\beta < s_\alpha$. According to Theorem~\ref{thm:dec.modulo.s}, the subwords $w_i$ all appear in
$\alpha = u(\alpha)+ k\delta$ with $k\geq s_\alpha$. Consider the corresponding costandard factorization
$\SL(u(\alpha) + k\delta) = \SL(\beta)\SL(\gamma)$. According to Lemma~\ref{lemma:saturated.dec.costfac}, we then have
$M'_1(\alpha) = M'_1(\beta) = M'_1(\gamma)$ so that $s_\beta,s_\gamma<s_\alpha$, both chains $\chain(\beta),\chain(\gamma)$
are decreasing, and $|\beta| \geq |u(\beta) + ks_\beta\delta|, |\gamma| \geq |u(\gamma) + ks_\gamma\delta|$.
Thus, all parts (1)--(5) follow by applying the inductive hypothesis for $\beta, \gamma$.
\end{proof}

\begin{remark}\label{rem:cost.fact.structure}
The above two results imply that applying iteratively the costandard factorization to $\SL(u(\alpha) + k\delta)$ with
$k \geq s_\alpha$, one splits this word into $s_\alpha$ pieces of the form
$\underbrace{\SL_i(\delta)}_{ q_j \text{ times}}w_j$ for $q_j \in \{\lfloor k/s_\alpha\rfloor,\lceil k/s_\alpha \rceil\}$,
and moreover affine roots $\gamma_j = \deg(w_j)$ satisfy $\sum_{j=1}^{s_\alpha} \gamma_j \in \chain(\alpha)$,
$M'_1(\gamma_1) = \cdots = M'_1(\gamma_{s_\alpha}) = M'_1(\alpha)$, and $s_{\gamma_j} = 1$ for all $j$.
\end{remark}

Our next result pertains to the structure of canonical factorizations of these $w_i$.

\begin{lemma}\label{lemma:dec.chains.canon.fac}
For any $\chain(\alpha) \in \mathcal{D}$, set $M'_1(\alpha) = (\delta,i)$ and let $w$ be one of $w_\bullet$'s
in~\eqref{eq:chunk-periodicity-s}, i.e.\ a subword connecting imaginary chunks in $\SL(u(\alpha)+k\delta)$
for $k \geq s_\alpha$. For the canonical factorization $w = v_1\ldots v_n$, we have for all $j \in \{1,\ldots, n\}$:
\begin{enumerate}[leftmargin=0.8cm]

\item
$\chain(\deg(v_j)) \in \mathcal{D}$;

\item
$\chain(\deg(v_j)) \in C_i$;

\item
$u_i(\deg(v_j)) = \deg(v_j)$;

\item
$M'_1(\deg(v_j)) = (\delta,i) \Leftrightarrow j = 1$.
\end{enumerate}
\end{lemma}

\begin{proof}
Let us first prove this when $s_\alpha = 1$ by induction on $\hgt(\chain(\alpha))$, in which case there is only one $w$.
The base case $\hgt(\chain(\alpha))=1$ corresponds to $\chain(\alpha)$ being an \irrchain, in which case $w$ is Lyndon
and $w = \SL(u(\alpha)) = \SL(\alpha')$ by Corollary~\ref{cor:dec.irr.chains} and Lemma~\ref{lemma:alt.u.def}.

For the inductive step, suppose the results hold for all chains of relative height $< \hgt(\chain(\alpha))$.
If $w$ is Lyndon, then (1)--(4) follow from Corollary~\ref{cor:dec.w.form}.
If $w$ is not Lyndon, then $n>1$ in its canonical factorization $w = v_1\ldots v_n$.
By Corollary~\ref{cor:lyndon.prefix.suffix.chains}, we have $\chain(\deg(v_n)) \in \mathcal{D}$.
By Lemma~\ref{lemma:seq.right.word.Lyndon}, we have $x =\SL(M'_1(\alpha))v_1\ldots v_{n-1}$ is a Lyndon prefix of
$\SL(u(\alpha) +\delta)$. Note that $x$ is real by Corollary~\ref{cor:imaginary.suffix.prefix},
$\chain(\deg(x))\in \mathcal{D}$ by Corollary~\ref{cor:lyndon.prefix.suffix.chains}, and $M'_1(\deg(x)) = (\delta,i)$
by Lemma~\ref{lemma:M.prime.prefix}. Applying the inductive hypothesis to $x$ yields part (1). We also note that
$M'_1(\deg(v_n)) < (\delta,i)$ as otherwise $s_\alpha > 1$, contradicting our assumption, which together with the
inductive hypothesis for $x$ implies part (4). Additionally we have $\SL_i(\delta) < v_n$ because $\SL(u(\alpha) +\delta)$
is Lyndon, implying $\chain(\deg(v_n)) \in C_i$. Together with the induction hypothesis for $x$ this establishes part (2).
It is clear that $|v_n|  \leq |u_i(\deg(v_n))|$, as otherwise $v_n < \SL_i(\delta)$ which would contradict $xv_n$ being Lyndon.
If $|v_n| < |u_i(\deg(v_n))|$, then words $\SL(\deg(v_n) + \delta)$ and $\SL(\deg(x) - \delta)$ do not contain $\SL_i(\delta)$
as a prefix and are both $>\SL_i(\delta)$, so that their appropriate concatenation is $> xv_n = \SL(u(\alpha) + \delta)$,
contradicting to Leclerc's algorithm. Thus $|v_n| = |u_i(\deg(v_n))|$, which together with the inductive assumption for $x$
proves~(3).

We now consider the general case, proceeding by induction on $s_\alpha$, with the base case $s_\alpha = 1$ treated above.
For the inductive step, suppose the result holds for all $\eta$ with $s_\eta < s_\alpha$. Consider the costandard factorization
$\SL(\alpha) = \SL(\beta)\SL(\gamma)$, so that $\chain(\beta),\chain(\gamma)\in \mathcal{D}$,
$M'_1(\beta) = M'_1(\gamma) = M'_1(\alpha)$, and
$|\beta| \geq |u(\beta) + s_\beta\delta|,|\gamma| \geq |u(\gamma) + s_\gamma\delta|$ by Lemma~\ref{lemma:saturated.dec.costfac}.
In particular, this factorization does not split $w_\bullet$'s in $\SL(\alpha)$, and thus the result follows from the inductive
assumption for $\beta, \gamma$.
\end{proof}

\begin{remark}\label{rem:u_i}
Results from the code (Listing \ref{lst:u_i}) indicate that $|u(\alpha)| < 10|\delta|$ for all decreasing chains $\chain(\alpha)$.
This has been checked for all orders in the five exceptional types as well as in rank $\leq 6$ cases of the classical types.
\end{remark}

In contrast to the pattern of decreasing chains from~\cite{AT} in $A$-type, we note that in general it is \underline{not}
always true that the multiset $\{w_i\}_{i=1}^{s_\alpha}$ for $\SL(u(\alpha) + k\delta)$ from Theorem~\ref{thm:dec.modulo.s}
is independent of $k \geq 1$, as illustrated in the following example.

\begin{example}\label{ex:D5}
Consider the affine type $D_5^{(1)}$ with the order $0<1<2<3<4<5$. Then $\chain(\delta - \alpha_0)$ is decreasing by
Lemma~\ref{lemma:monotonicity.smallest.once} and using the code in Listing~\ref{lst:D5} we~find
\begin{gather*}
  \SL(\delta - \alpha_0 + 3k\delta) = \underbrace{\SL_1(\delta)}_{k \text{ times}} 1
    \underbrace{\SL_1(\delta)}_{k \text{ times}} 23543 \underbrace{\SL_1(\delta)}_{k \text{ times}} 2,\\
  \SL(2\delta - \alpha_0 + 3k\delta) = \underbrace{\SL_1(\delta)}_{k+1 \text{ times}} 23543
    \underbrace{\SL_1(\delta)}_{k \text{ times}} 1 \underbrace{\SL_1(\delta)}_{k \text{ times}} 2,\\
 \SL(3\delta - \alpha_0 + 3k\delta) = \underbrace{\SL_1(\delta)}_{k + 1\text{ times}} 234
   \underbrace{\SL_1(\delta)}_{k \text{ times}} 1 \underbrace{\SL_1(\delta)}_{k + 1 \text{ times}} 235,
\end{gather*}
for all $k \in \BZ_{\geq 0}$. This corresponds to two different splittings $\delta - \alpha_0=\gamma_1 + \gamma_2 + \gamma_3$
with $M'_1(\gamma_\bullet) = M'_1(\delta - \alpha_0)$ and $s_{\gamma_\bullet} = 1$. Explicitly, in this case we have
$\chain(\delta-\beta_i) = \chain(\alpha_i)$ for all $1\leq i\leq 5$, cf.~Corollary~\ref{cor:simple-0}.
The following table lists the values of all $M_1(\beta_i)$:
\begin{center}
\begin{tabular}{|c|ccccc|}
  \hline
    $i$& 1 & 2 & 3 & 4 & 5 \\
  \hline
    $M_1(\beta_i)$ & 1 & 1 & 2 &3  &3 \\
  \hline
\end{tabular}
\end{center}
In particular, all above subwords $w\in \{1, 2, 234, 235, 23543\}$ have $M'_1(\deg(w))=1$ (this matches nicely with
Corollary~\ref{cor:dec.w.form} and Lemma~\ref{lemma:dec.chains.canon.fac}, see also Remark~\ref{rem:cost.fact.structure}).
\end{example}

We shall now obtain a tight upper bound on the periodicity $s_\alpha$ for $\chain(\alpha)\in \mathcal{D}$.

\begin{definition}
For any $\chain(\alpha) \in \mathcal{D}$, we define $f_\alpha\in \BZ_{>0}$ via
\begin{equation}\label{eq:f-function}
  f_{\alpha} = \max_{k \in \{1,2,\ldots, |I|\}} \left\{\sum_{(\beta_k,\beta_j) \neq 0} (-c_j)\right\}
\end{equation}
where $\chain(\alpha) = c_1\chain(\beta_1) + \ldots + c_{|I|}\chain(\beta_{|I|})$.
\end{definition}

Let $\theta'\in \Delta^-_p$ be the highest root with respect to the polarization from Lemma~\ref{lemma:chain.positive.roots}.

\begin{lemma}\label{lemma:f.theta.maximal}
For any $\chain(\alpha) \in \mathcal{D}$, we have $f_{\alpha} \leq f_{\theta'}$.
\end{lemma}

\begin{proof}
This follows from the basic property of finite root systems that the difference of the highest positive root and
any positive root is a sum of simple roots.
\end{proof}

We also note that the number $f_{\theta'}$ is independent of the order on $\wI$.

\begin{lemma}\label{lemma:algo.periodicity.decreasing.chains}
For any $\chain(\alpha)\in \mathcal{D}$, we have $s_\alpha \leq f_{\alpha}$.
\end{lemma}

\begin{proof}
Let $M'_1(\chain(\alpha)) = (\delta,i)$. Then we have:
\begin{equation*}
  s_\alpha = \sum_{j}^{M_1(\beta_j) = (\delta,i)} (-c_j) = \sum_{j}^{(\beta_j,\beta_i) \neq 0} (-c_j)
  \leq \max_{1\leq k\leq |I|} \left\{\sum_{j}^{(\beta_k,\beta_j) \neq 0} (-c_j) \right\} = f_{\alpha},
\end{equation*}
where the second equality is due to Corollary~\ref{cor:b.j.M.k}.
\end{proof}

By combining the above two lemmas, we obtain an upper bound for the periodicity of all decreasing chains, independent of
the underlying order on the root system. A detailed case-by-case treatment of all finite root systems yields the following.

\begin{corollary}\label{cor:max.periodicities}
We have the following upper bounds for the periodicity of all decreasing chains for any untwisted affine root system:
\begin{center}
\begin{tabular}{|c|c|}
    \hline
         type & max periodicity \\
    \hline
         $A_1^{(1)}$ & 1 \\
         $A_2^{(1)}$ & 2 \\
         $A_n^{(1)}, n > 2$ & 3 \\
    \hline
         $B_2^{(1)},C_2^{(1)}$ & 3 \\
         $B_3^{(1)},C_3^{(1)}$ & 5 \\
         $B_n^{(1)},C_n^{(1)}, n > 3$ & 6 \\
    \hline
         $D_4^{(1)}$ & 5 \\
         $D_n^{(1)}, n > 4$ & 6 \\
    \hline
         $E_6^{(1)}$ & 9 \\
         $E_7^{(1)}$ & 12 \\
         $E_8^{(1)}$ & 18 \\
    \hline
         $F_4^{(1)}$ & 9 \\
    \hline
         $G_2^{(1)}$ & 5 \\
    \hline
\end{tabular}
\captionof{table}{Upper bound on periodicity of decreasing chains}\label{table:dec.periodicity}
\end{center}
\end{corollary}

In fact, the above bounds are precise, as we show in Proposition~\ref{lem:precise-bound} below.

\begin{lemma}\label{lem:opposite_polarizations}
If $0\in \wI$ is the smallest letter, then $\Delta^+ = \Delta_p^-$.
\end{lemma}

\begin{proof}
According to Lemma~\ref{lemma:monotonicity.smallest.once}, for $|\alpha|<|\delta|$ the chain $\chain(\alpha)$
is decreasing iff $\alpha$ does not contain $\alpha_0$, that is $\alpha\in \Delta^+$.
\end{proof}

Comparing simple roots of both $\Delta^{+}$ and $\Delta_{p}^{-}$, we thus derive the following result.

\begin{corollary}\label{cor:simple-0}
If $0 \in \wI$ is the smallest letter, then $\{\delta-\beta_i\}_{i=1}^{|I|} = \{\alpha_j\}_{j=1}^{|I|}$.
\end{corollary}

\begin{remark}
The above can be naturally generalized to the setup when the smallest simple root $\alpha_\varepsilon$ occurs only once
in $\delta$: then $\{\delta-\beta_i\}_{i=1}^{|I|} = \{\alpha_j\}_{j=0}^{|I|} \setminus \{\alpha_{\varepsilon}\}$.
\end{remark}

We also note that $\theta'$ coincides with the highest root $\theta$ of $\Delta^{+}$,
due to Lemma~\ref{lem:opposite_polarizations}.

\begin{lemma}\label{lemma:smallest.once.SL.1}
For any order on $\wI$ of the form $0 < k < \cdots $, we have $\SL^{r}_1(\delta) = k$.
Moreover, we also have $\chain(\alpha_k) = \chain(\delta - \beta_1)$.
\end{lemma}

\begin{proof}
According to Lemma~\ref{lemma:right.costfac.minimal}, $\SL^r_1(\delta)$ starts with $k$. Assume that
$\SL^r_1(\delta) \neq k$, and consider its costandard factorization with $u = \SL^{rl}_1(\delta)$,
$v = \SL^{rr}_1(\delta)$. As $\SL^l_1(\delta)<v$, we have $\SL(\delta - \deg(u)) \geq \SL^l_1(\delta)v$ by
Leclerc's algorithm, but still $\SL(\delta - \deg(u)) < u$ as the former starts with $0$ and the latter does not.
But then $\SL^{l}_1(\delta)uv = \SL_1(\delta) \geq \SL(\delta - \deg(u))u \geq  \SL^{l}_1(\delta)vu$, a contradiction.
Therefore, $\SL^r_1(\delta) = k$ as claimed.
We also note that $\chain(\beta_1)=\chain(\deg(\SL^{ls}_1(\delta)))=\chain(\deg(\SL^l_1(\delta)))$
(with the first equality following from the proof of Proposition~\ref{prop:irr.chains} and the second equality
due to Corollary~\ref{cor:mk.im.factor}) and so $\chain(\alpha_k) = \chain(\delta - \beta_1)$.
\end{proof}

\begin{proposition}\label{lem:precise-bound}
Let $1 \leq k \leq |I|$ be the value which maximizes~\eqref{eq:f-function} for $\theta$. For any order $0 < k < \cdots$
on $\wI$ and any integer $p$ between $1$ and the value listed in Table~\ref{table:dec.periodicity}, there is
$\chain(\alpha) \in \mathcal{D}$ with periodicity $p$. Moreover, the periodicity of $\chain(\theta)$ is given precisely
by the value in Table~\ref{table:dec.periodicity}.
\end{proposition}

\begin{proof}
By Lemma~\ref{lem:opposite_polarizations}, simple positive roots of $\Delta^+_p$ are the negative of those in~$\Delta^+$.

We start by proving the last part about the periodicity $s_\theta$ of $\chain(\theta')=\chain(\theta)$. Note that
$M'_1(\theta') = (\delta,1)$ by Lemma~\ref{lemma:M.prime.bound}, as each simple root occurs at least once in the highest
root (cf.~the proof of Lemma~\ref{lemma:f.theta.maximal}). We have $M_1(\beta_i) = (\delta,1)$ iff $(\beta_1,\beta_i) \neq 0$,
that is, if $\Pr(\chain(\delta - \beta_i))$ either coincides with $\Pr(\chain(\delta - \beta_1))$ or is connected to it in the
Dynkin diagram of $\fg$, due to Corollary~\ref{cor:b.j.M.k}. As $\chain(\delta - \beta_1) = \chain(\alpha_k)$ by
Lemma~\ref{lemma:smallest.once.SL.1}, we conclude that $s_\theta$ precisely coincides with the values in
Table~\ref{table:dec.periodicity}.

Consider a sequence of chains $(\chain(\eta_i))_{i=1}^{N}$ with roots $\eta_i$ satisfying Lemma~\ref{lemma:seq.simple.root},
that is $\chain(\eta_{i+1})- \chain(\eta_i)  = - \chain(\beta_j)$ for some $j$, where $\eta_0:=0$, and
$\chain(\eta_N) = \chain(\theta)$. Choose the minimal $k \in [1,N]$ such that $M'_1(\eta_k) = (\delta,1)$, hence
$s_{\eta_k} = 1$. Then $M'_1(\eta_{i}) = (\delta,1)$ for any $i \geq k$ by~\eqref{eq:Mprime-additive}, and
$s_{\eta_{i+1}} - s_{\eta_i} \in \{0,1\}$ as $\chain(\eta_{i+1}) - \chain(\eta_{i})$ is an \irrchain.
As $s_{\eta_k} = 1$ and $\chain(\eta_N)=\chain(\theta')$, the first claim follows.
\end{proof}

We conclude this section with the following general observation.

\begin{lemma}
If the smallest simple root $\alpha_\varepsilon$ occurs only once in $\delta$, then for any decreasing chain
$\chain(\alpha) \in C_i$, we have $u_i(\alpha) = \alpha'$, the shortest element of $\chain(\alpha)$.
\end{lemma}

\begin{proof}
By Lemma~\ref{lemma:monotonicity.smallest.once}, we have $\chain(\alpha) \in \mathcal{D}$ if and only if $\SL(\alpha')$ does not
contain the smallest letter $\varepsilon$. Hence for any $\chain(\alpha),\chain(\beta),\chain(\gamma) \in \mathcal{D}$ such that
$\chain(\alpha) = \chain(\beta) + \chain(\gamma)$, we have $\alpha' = \beta' + \gamma'$. We thus proceed by induction on
$\hgt(\chain(\alpha))$. The base case $\hgt(\chain(\alpha)) = 1$, corresponding to irreducible $\chain(\alpha)$, follows
from Corollary~\ref{cor:dec.irr.chains}.

For the inductive step suppose the result holds for all chains of relative height $<\hgt(\chain(\alpha))$. If the first set
in the RHS of~\eqref{eqn:alt.u.def} is empty then the result follows from Lemma~\ref{lemma:alt.u.def}. Otherwise, by the above
argument and the inductive hypothesis the RHS of~\eqref{eqn:alt.u.def} still contains only $\alpha'$, and the result follows from
Lemma~\ref{lemma:alt.u.def}.
\end{proof}


\section{Periodicity of Increasing Chains}\label{sec:incr}


\subsection{Preliminaries}

\

This section is devoted to the description of periodicity pattern for increasing chains. While having
similar nature to the periodicity from Section~\ref{sec:dec}, the present patterns are more complicated,
as already illustrated by $A^{(1)}_n$-type in~\cite{AT}.

\begin{definition}\label{def:inc.chains.c}
For any $\chain(\alpha) \in \mathcal{I}$, we define $c_\alpha = c_i$, where
$\chain(\alpha) = c_1\chain(\beta_1) + c_2\chain(\beta_2) + \cdots + c_i\chain(\beta_i)$ with $c_i \neq 0$,
cf.\ Corollary~\ref{cor:lin.comb.chains}(a).
\end{definition}

The following result is crucial for many inductive arguments below
(and is a close analogue of Lemma~\ref{lemma:s.splitting} for decreasing chains).

\begin{lemma}\label{lemma:c.splitting}
For any $\chain(\alpha) \in \mathcal{I}$ with $c_\alpha > 1$, there exist $\chain(\beta),\chain(\gamma) \in \mathcal{I}$
such that $m_1(\beta) = m_1(\gamma) = m_1(\alpha)$ and $\chain(\beta) + \chain(\gamma) = \chain(\alpha)$.
\end{lemma}

\begin{proof}
This will be shown by induction on $\hgt(\chain(\alpha))$. Let $\chain(\alpha)$ be the chain with the minimal
relative height and $c_\alpha > 1$, and consider any splitting $\chain(\alpha) = \chain(\beta) + \chain(\gamma)$
with $\chain(\beta),\chain(\gamma) \in \mathcal{I}$. Assuming contradiction with the lemma, we then have
$m_1(\beta)=m_1(\alpha)<m_1(\gamma)$ or $m_1(\gamma)=m_1(\alpha)<m_1(\beta)$, which implies $c_\beta = c_\alpha > 1$
or $c_\gamma = c_\alpha > 1$, thus contradicting to the choice of $\alpha$. This completes the base case.

For the induction step, assume that the claim holds for all chains of relative height $< \hgt(\chain(\alpha))$.
Since $\chain(\alpha)$ is not an \irrchain, it can be expressed as $\chain(\alpha)=\chain(\beta)+\chain(\gamma)$
with $\chain(\beta),\chain(\gamma) \in \mathcal{I}$. Assuming that $\chain(\beta),\chain(\gamma)$ do not meet
the requirements of the claim, we then again have $m_1(\beta) = m_1(\alpha) < m_1(\gamma)$ or
$m_1(\gamma) = m_1(\alpha) < m_1(\beta)$. Without loss of generality, we shall assume the former case.
We then have $c_\beta \geq 2$ and by the induction assumption for $\chain(\beta)$ there exist
$\chain(\mu),\chain(\eta) \in \mathcal{I}$ such that $\chain(\mu) +\chain(\eta) = \chain(\beta)$ and
$m_1(\mu) = m_1(\eta) = m_1(\beta)$. Applying the same logic as in the proof of Lemma~\ref{lemma:s.splitting},
we must have either $\chain(\mu) + \chain(\gamma) \in \mathcal{I}$ or $\chain(\eta) + \chain(\gamma) \in \mathcal{I}$,
due to Lemma~\ref{lemma:chain.sum.monotone}. Hence, we can express $\chain(\alpha)$ as either
$(\chain(\gamma) + \chain(\mu)) + \chain(\eta)$ or $(\chain(\gamma) + \chain(\eta)) + \chain(\mu)$, with both terms
satisfying the conditions by Corollary~\ref{cor:both.increasing.m.k}. This completes the step of induction.
\end{proof}

The following result is a counterpart of Corollary~\ref{cor:M.prime.bounding} for increasing chains.

\begin{lemma}\label{lemma:no.bound.increasing}
For any $\chain(\alpha), \chain(\beta) \in \mathcal{I}$ with $m_1(\alpha) = m_1(\beta)$ and any
$\hat\alpha \in \chain(\alpha)$, there exists $\hat\beta \in \chain(\beta)$ such that $\hat\beta > \hat\alpha$.
\end{lemma}

\begin{proof}
Consider a relation on the set $\mathcal{I}$ defined through the set
\begin{gather*}
  R = \left\{(\chain(\alpha),\chain(\beta)) \in \mathcal{I} \times \mathcal{I} \,\bigg\vert\,
         \substack{ m_1(\chain(\alpha)) = m_1(\chain(\beta))
        \\\forall\, \hat\alpha \in \chain(\alpha) \ \exists\, \hat\beta\in\chain(\beta)\colon \hat\beta > \hat\alpha}\right\}.
\end{gather*}
This relation is clearly reflexive and transitive. Due to transitivity, to show $\chain(\alpha)R\chain(\beta)$ for any
$\chain(\alpha),\chain(\beta) \in \mathcal{I}$ with $m_1(\alpha) = (\delta,i) = m_1(\beta)$, it suffices
to show $\chain(\alpha)R\chain(\beta_i)$ and $\chain(\beta_i)R\chain(\beta)$. Henceforth, we fix index $i$.

We first prove $\chain(\alpha)R \chain(\beta_i)$ whenever $m_1(\alpha) = (\delta,i)$ by induction on the relative height
$\hgt(\chain(\alpha))$. The base case $\hgt(\chain(\alpha))=1$, corresponding to $\chain(\alpha) = \chain(\beta_i)$, is obvious.
For the inductive step, suppose the result holds for all $\chain(\mu)\in \mathcal{I}$ with $m_1(\mu) = (\delta,i)$ and
$\hgt(\chain(\mu)) < \hgt(\chain(\alpha))$. Consider any splitting $\chain(\alpha) = \chain(\gamma) + \chain(\eta)$ with
$\chain(\gamma),\chain(\eta) \in \mathcal{I}$. By Corollary~\ref{cor:both.increasing.m.k}, we can assume without loss of
generality that $m_1(\gamma) = (\delta,i)$. For any $\hat\alpha\in \chain(\alpha)$, we have
$\hat\alpha + (\delta - \eta') = \hat\gamma \in \chain(\gamma)$, where $\eta'$ is the shortest element of $\chain(\eta')$.
According to Corollary~\ref{cor:inc.dec.comparison}, we have $\hat\alpha < \hat\gamma < \delta-\eta'$. Applying the inductive
hypothesis to $\chain(\gamma)$, there is $\hat\beta_i \in \chain(\beta_i)$ such that $\hat\gamma < \hat\beta_i$, and so
$\hat{\alpha} < \hat\beta_i$. This completes the step of induction.

We will now prove $\chain(\beta_i)R\chain(\beta)$ whenever $m_1(\beta) = (\delta,i)$ by induction on the relative height
$\hgt(\chain(\alpha))$. The base case $\hgt(\chain(\beta)) = 1$ is obvious. For the inductive step, suppose the result holds
for all $\chain(\mu) \in \mathcal{I}$ with $m_1(\mu) = (\delta,i)$ and $\hgt(\chain(\mu)) < \hgt(\chain(\beta))$. Consider
any splitting $\chain(\beta) = \chain(\gamma) + \chain(\eta)$ with $\chain(\gamma), \chain(\eta) \in \mathcal{I}$.
Again by Corollary~\ref{cor:both.increasing.m.k}, suppose without loss of generality that $m_1(\gamma) = (\delta,i)$.
We now consider two cases:
\begin{itemize}[leftmargin=0.7cm]

\item
If $m_1(\eta) = m_1(\gamma)$, then we have $\chain(\beta_i)R\chain(\gamma)$ and $\chain(\beta_i)R\chain(\eta)$
by the inductive hypothesis, that is, for any $\hat\beta_i \in \chain(\beta_i)$ there exist $\hat\gamma\in \chain(\gamma)$
and $\hat\eta\in \chain(\eta)$ such that $\hat\beta_i < \hat\gamma,\hat\eta$. Then,
$\hat\beta_i < \hat\gamma + \hat\eta$ by Theorem~\ref{thm:convexity}, whereas $\hat\gamma + \hat\eta \in \chain(\beta)$;

\item
If $m_1(\eta) = (\delta,j) \neq (\delta,i)$, then $m_1(\eta) > m_1(\gamma)$ by Corollary~\ref{cor:both.increasing.m.k}.
Applying the previous part and Corollary~\ref{cor:mk.im.factor} yields
$\chain(\deg(\SL^{ls}_j(\delta)))R\chain(\beta_j)$. Moreover, $\chain(\beta_j)R\chain(\eta)$ by the inductive hypothesis,
and so $\chain(\deg(\SL^{ls}_j(\delta)))R\chain(\eta)$ by transitivity. Thus there is $\hat\eta \in \chain(\eta)$ satisfying
$\SL^{ls}_j(\delta) < \hat\eta$ and hence $(\delta,i) < \hat\eta$ by Lemma~\ref{lemma:left.standard.comp.imaginary}.
Applying the inductive hypothesis to $\chain(\gamma)$, for any $\hat\beta_i \in \chain(\beta_i)$ there exists
$\hat\gamma \in \chain(\gamma)$ such that $\hat\beta_i < \hat\gamma$. We also have $\hat\beta_i < (\delta,i) < \hat\eta$
by above and Proposition~\ref{prop:monotonicity}. Therefore, $\hat\beta_i < \hat\eta + \hat\gamma \in \chain(\beta)$
by Theorem~\ref{thm:convexity}.

\end{itemize}
This completes the inductive step.
\end{proof}

The following result is an analogue of Corollary~\ref{cor:tight.bound.M.prime} for increasing chains.

\begin{corollary}\label{cor:tight.bound.m}
For any $\chain(\alpha)\in \mathcal{I}$ with $m_1(\alpha) = (\delta,i)$, we have $(\delta,i+1) < \alpha + k\delta$
for $k \gg 0$.
\end{corollary}

\begin{proof}
According to Lemma~\ref{lemma:no.bound.increasing} and Corollary~\ref{cor:mk.im.factor}, there exists
$\hat\alpha \in \chain(\alpha)$ satisfying $\SL(\hat\alpha) > \SL_i^{ls}(\delta)$, and so $\hat\alpha > (\delta,i+1)$
by Lemma~\ref{lemma:left.standard.comp.imaginary}. As $\chain(\alpha)$ is increasing, we get
$\alpha+k\delta > (\delta,i+1)$ whenever $|\alpha+k\delta|\geq |\hat\alpha|$.
\end{proof}

\begin{remark}\label{rem:m1-role}
Similarly to Remark~\ref{rem:M'1-role}, we emphasize the importance played by the function $m_1$ for increasing
chains $\chain(\alpha)\in \mathcal{I}$: we have $\SL_{i+1}(\delta) < \SL(\hat{\alpha}) < \SL_{i}(\delta)$ for
all $\hat{\alpha}\in \chain(\alpha)$ with $|\hat\alpha|\gg 1$, where $m_1(\alpha)=(\delta,i)$.
\end{remark}

The following result is an analogue of Lemma~\ref{lemma:M.prime.prefix} for increasing chains.

\begin{lemma}\label{lemma:left.standard.prefix.increasing}
If $\chain(\alpha) \in \mathcal{I}$ and $\SL(\hat\alpha)$ has $\SL_i^{ls}(\delta)$ as a prefix for some
$\hat\alpha \in \chain(\alpha)$, then $m_1(\alpha) = (\delta,i)$.
\end{lemma}

\begin{proof}
We have $\SL(\hat\alpha) = \SL^{ls}_i(\delta)w$ for some $w$, and the proof proceeds by induction on $|\hat\alpha|$.
The base case $\SL(\hat\alpha) = \SL^{ls}_i(\delta)$ is clear. For the inductive step, suppose the result holds for
roots of height $ <|\hat\alpha|$.

Consider the canonical factorization $w = w_1w_2\ldots w_n$. Let $\gamma = \deg(w_n)$ and
$\beta = \deg(\SL^{ls}_i(\delta)w_1w_2\ldots w_{n-1})$. By Lemma~\ref{lemma:seq.right.word.Lyndon}, the word
$\SL^{ls}_i(\delta)w_1\ldots w_{n-1}$ is Lyndon, hence $\chain(\beta) \in \mathcal{I}$ by
Corollary~\ref{cor:lyndon.prefix.suffix.chains}, and so $m_1(\beta) = (\delta,i)$ by the inductive hypothesis.
We also claim that $m_1(\alpha) \geq (\delta,i)$, as otherwise we get get a contradiction with Proposition~\ref{prop:monotonicity}
since $\SL_{i+1}(\delta) < \SL_i^{ls}(\delta) < \SL(\hat\alpha)$ by Lemma~\ref{lemma:left.standard.comp.imaginary}.

If $m_1(\alpha) > (\delta,i)$, then $\chain(\gamma)\in \mathcal{D}$ and $m_1(\gamma) = (\delta,i)$
by Corollary~\ref{cor:both.increasing.m.k}. Pick an integer $k$ so that $0<|k\delta - \gamma| < |\delta|$, and consider
the decomposition $\beta = (\hat\alpha - k\delta) + (k\delta - \gamma)$. We have two cases to consider:
\begin{itemize}[leftmargin=0.7cm]

\item
If $|\hat\alpha| < |k\delta|$, then $\beta <  k\delta - \gamma  < k\delta -\hat\alpha$ by Corollary~\ref{cor:inc.dec.comparison}.
But then $m_1(k\delta - \gamma) = (\delta,i) < k\delta - \gamma$ by Lemma~\ref{lemma:equiv.to.standard.fac} as
$\SL_i^{ls}(\delta) \leq \SL(\beta) < \SL(k\delta - \gamma)$, a contradiction with Proposition~\ref{prop:monotonicity};

\item
If $|\hat\alpha| > |k\delta|$, then by Theorem~\ref{thm:convexity} we have two sub-cases to consider:
either $\hat\alpha - k\delta < \beta < k\delta - \gamma$ or $k\delta - \gamma < \beta < \hat\alpha - k\delta$.
In the first case, combining $|k\delta - \gamma| < |\delta|$, $m_1(\gamma) = (\delta,i)$,
$\chain(k\delta - \gamma) \in \mathcal{I}$, we get $\SL(k\delta - \gamma) < \SL_i^{ls}(\delta)< \SL(\beta)$
by Lemma~\ref{lemma:equiv.to.standard.fac}, a contradiction. In the second case, if $\SL(\hat\alpha - k\delta)$ has
$\SL_i^{ls}(\delta)$ as a prefix then $m_1(\alpha)=m_1(\hat\alpha-k\delta)=(\delta,i)$ by the induction hypothesis,
while otherwise $\hat\alpha - k\delta > \hat\alpha$, contradicting to $\chain(\alpha) \in \mathcal{I}$.

\end{itemize}
Therefore, $m_1(\alpha) = (\delta,i)$, which completes the inductive step.
\end{proof}


\subsection{$l$-function}
\

The following result is crucial for our analysis of increasing chains.

\begin{lemma}\label{lemma:recursive.u.j}
For any $1 \leq i \leq |I|$, consider a sequence of words $\{u_j\}_{j=0}^\infty$ defined by $u_0 = \SL^{ls}_i(\delta)$ and
the recursive relation $u_j = u_{j-1}v_{j}$, where $v_j$ is the lexicographically largest element of the following set:
\begin{equation}\label{eq:v.j}
    V_j = \left\{\ell \in \SL \, \bigg\vert \, \substack{ [\sb[u_{j-1}],\sb[\ell]] \neq 0, \\
    \chain(\deg(u_{j-1}) + \deg(\ell)) \in \mathcal{I}, \\
    m_1(\deg(u_{j-1}) + \deg(\ell)) = (\delta,i)}\right\}.
\end{equation}
This sequence of words satisfies the following properties:
\begin{enumerate}[leftmargin=0.8cm]

\item
$u_j \in \SL$ for all $j$;

\item
if $\deg(v_j)$ is real, then $\chain(\deg(v_j))\in \mathcal{D}$;

\item
$v_j \geq v_{j+1}$ for all j;

\item
there is $N$ such that $v_N$ is imaginary while all $v_{<N}$ are real, and furthermore
$v_{>N} = v_N = \SL(M_1(\deg(u_{N-1})))$.

\end{enumerate}
\end{lemma}

\begin{proof}
Let $\alpha = \deg(u_{j-1})$.
First, let us show the above $v_j$ is well-defined, that is, the set $V_j$ of~\eqref{eq:v.j} is non-empty and contains
the maximal element. The former property follows from $\SL(M_1(\alpha)) \in V_j$, while the latter
is established in the next paragraph (note that the set $V_j$ is infinite as $\{\SL(M_k(\alpha))\}_{k>0} \subseteq V_j$).

We will first show that we can ignore real $\ell \in V_j$ with $\chain(\deg(\ell)) \in \mathcal{I}$, when determining the
maximum of $V_j$. We note that $m_1(\deg(\ell)) < M'_1(\alpha)$ by Lemma~\ref{lemma:interlock.m.M}. We now consider two cases:
\begin{itemize}[leftmargin=0.8cm]

\item
If $M'_1(\alpha) = M_1(\alpha)$, then by Proposition~\ref{prop:monotonicity} we have
$\ell < \SL(m_1(\deg(\ell))) < \SL(M_1(\alpha))$, and so $\ell$ cannot be the maximum of the set $V_j$
since we already proved that $\SL(M_1(\alpha)) \in V_j$;

\item
If $M'_1(\alpha) \neq M_1(\alpha)$, then $s_{-\alpha} > 1$ by Corollary~\ref{cor:M.M.prime.diff}.
Applying Lemma~\ref{lemma:s.splitting}, consider the decomposition $\chain(-\alpha) = \chain(\beta) + \chain(\gamma)$
with $\chain(\beta),\chain(\gamma) \in \mathcal{D}$ and $M'_1(\beta) = M'_1(\gamma) = M'_1(\alpha)$. Without loss of
generality we can assume that $m_1(\gamma) = m_1(\alpha)$ using Corollary~\ref{cor:both.increasing.m.k}. Let $\beta'$
be the shortest element of $\chain(\beta)$. We claim that $\SL(\beta') \in V_j$, which follows from
$\chain(\alpha)+\chain(\beta) = \chain(-\gamma)\in \mathcal{I}$ and  $m_1(-\gamma) = m_1(\alpha)$.
By Lemma~\ref{lemma:dec.chains.lower.bound}, we have $M'_1(\alpha) = M'_1(\beta) < \beta'$, so that
$\ell < \SL(m_1(\deg \ell)) < \SL(M'_1(\alpha)) < \SL(\beta')$. Therefore, $\ell$ again cannot be
the maximum of the set $V_j$.

\end{itemize}
Thus for any $\ell\in V_j$ with $\chain(\deg(\ell)) \in \mathcal{I}$, we can find $\ell' \in V_j$ with $\ell'>\ell$ and
$\chain(\deg(\ell')) \in \mathcal{D}$ or $\ell'$ is imaginary. Once we restrict to $\ell \in V_j$ which correspond to
decreasing chains or are imaginary, this set has a clear maximum:
\begin{itemize}[leftmargin=0.8cm]

\item
for decreasing chains (there are finitely many of those), the shortest word that lies in $V_j$ is larger than
all other words in the same chain that belong to $V_j$;

\item
for imaginary words, this follows from Corollary~\ref{cor:equiv.mk.Mk} and Lemma~\ref{lemma:imaginary.words.decreasing}.

\end{itemize}
This proves that $v_j$ is indeed well-defined and establishes part~(2) of the lemma.

We next prove part~(1) by induction on $j$, the base case $j=0$ being vacuous. For the inductive step, assume that
$u_{j-1} \in \SL$. By the second line of~\eqref{eq:v.j} we have $\alpha + \deg(v_j) \in \wDelta^{+,\re}$. We also note
that $u_{j-1} < v_j$, due to Corollary~\ref{cor:inc.dec.comparison} (if $v_j$ is real) or Proposition~\ref{prop:monotonicity}
(if $v_j$ is imaginary). Thus, $u_{j-1}v_j \leq \SL(\alpha + \deg(v_j))$ by Leclerc's algorithm. Next, we claim that
$\SL(\alpha + \deg(v_j))$ contains $u_{k}$ as a prefix for all $0 \leq k \leq j$, which we prove by induction. The base case
$k=0$ follows from a sequence of inequalities $u_0 < u_1 < \ldots < u_j \leq \SL(\alpha + \deg(v_j)) < \SL_i(\delta)$ combined
with $u_0$ being a prefix of $\SL_i(\delta)$.
For the inductive step, assume that $u_{k-1}$ is a prefix of $\SL(\alpha + \deg(v_j))$.
Then $\SL(\alpha + \deg(v_j))) = u_{k-1}w$ for some word $w$, and consider the canonical factorization $w = w_1w_2\ldots w_n$.
We note that $u_{k-1}w_1$ is Lyndon by Lemma~\ref{lemma:seq.right.word.Lyndon}, and $\deg(u_{k-1}w_1) \in \mathcal{I}$
by Corollary~\ref{cor:lyndon.prefix.suffix.chains}, as $u_{k-1}w_1$ is a Lyndon prefix of $\SL(\alpha + \deg(v_j))$ with
$\chain(\alpha + \deg(v_j)) \in \mathcal{I}$. Furthermore, $m_1(\deg(u_{k-1}w_1)) = (\delta,i)$ by
Lemma~\ref{lemma:left.standard.prefix.increasing} and $[\sb[u_{k-1}],\sb[w_1]] \neq 0$ as
$\deg(w_1) + \deg(u_{k-1}) \in \wDelta^{+,\re}$. This implies $w_1 \in V_k$, so that $w_1 \leq v_k$.
If $w_1 < v_k$, then $u_{j-1}v_j = u_{k-1}v_k \ldots v_{j-1} v_{j} > u_{k-1}w_1 \ldots w_n = \SL(\alpha + \deg(v_j))$
by Lemma~\ref{lem:Melancon}, contradicting $u_{j-1}v_j \leq \SL(\alpha + \deg(v_j))$ established above.
Thus $w_1 = v_k$ which completes the induction step, while for $k=j$ we obtain that $\SL(\alpha + \deg(v_j))$ starts with $u_j$.
This establishes part (1) as $\SL(\alpha + \deg(v_j)) = u_j$ for degree reasons. By above, we also see that if
$\SL(\alpha + \deg(v_j)) = u_0w$, then $w = v_1\ldots v_j$ is the canonical factorization, implying part~(3).

As all real $v_j$ correspond to $\chain(\deg(v_j))\in \mathcal{D}$ and $\chain(\deg(u_j))\in \mathcal{I}$,
we note that there must be an imaginary word among $v_j$'s with $1\leq j\leq \hgt(\chain(\deg(u_0)))$.
Let $N$ be the smallest index such that $v_N$ is imaginary. Then $v_N = \SL(M_1(\deg(u_{N-1})))$ by
Corollary~\ref{cor:imaginary.suffix.prefix}.
We have $v_{N+1} \leq v_N$ by part (3), but also $v_{N+1} \geq \SL(M_1(\deg(u_{N-1}))) = v_N$ as
$\SL(M_1(\deg(u_{N-1}))) = \SL(M_1(\deg(u_N))) \in V_{N+1}$ by above. Thus $v_{N+1} = v_{N}$,
and repeating this argument implies part (4).
\end{proof}

Using the notations from the lemma above, we make the following key definition.

\begin{definition}\label{def:y-elt}
For any $1 \leq i \leq |I|$ and the corresponding sequence $\{u_j\}_{j=0}^{\infty}$ from Lemma~\ref{lemma:recursive.u.j},
set $y_i = u_{N-1}$.
\end{definition}

We start with the following important property of this construction.

\begin{lemma}\label{lemma:y.i.irr.chain}
Connectivity of Conjecture~\ref{conj:connectivity} implies that $\deg(y_i) \in  \chain(\beta_i)$ for all $1 \leq i \leq |I|$.
\end{lemma}

\begin{proof}
If $\deg(y_i) \not \in \chain(\beta_i)$, then $\chain(\deg(y_i))$ is not an \irrchain. Applying
Lemma~\ref{lemma:connectivity.splitting}, there exist $\chain(\beta),\chain(\gamma) \in \mathcal{I}$ with
$\chain(\beta) + \chain(\gamma) = \chain(\deg(y_i))$, $m_1(\beta) = (\delta,i)$, and $M'_1(\gamma) = M'_1(\deg(y_i))$.
We note that $\SL(\delta - \gamma') \in V_{N}$ and $\delta - \gamma' > M'_1(\deg(y_i)) \geq M_1(\deg(y_i))$ with the
inequalities due to Lemmas~\ref{lemma:M.prime.bound} and~\ref{lemma:dec.chains.lower.bound}. Thus we have
$\delta - \gamma' > M_1(\deg(y_i)) = v_N$, a contradiction.
\end{proof}

The following definition plays the role of Definition~\ref{def:u} for increasing chains, while
we note that it is well-defined due to Lemma~\ref{lemma:no.bound.increasing}.

\begin{definition}
For any $\chain(\alpha) \in \mathcal{I}$ with $m_1(\alpha) = (\delta,i)$, let $l(\alpha) \in \chain(\alpha)$
be the shortest element such that $\SL(l(\alpha)) \geq y_i$, with $y_i$ defined in Definition~\ref{def:y-elt}.
\end{definition}

The first property of this construction is as follows.

\begin{lemma}\label{lemma:y.i.prefix}
For any $\chain(\alpha) \in \mathcal{I}$ with $m_1(\alpha) = (\delta,i)$ and for any $\hat\alpha \in \chain(\alpha)$
with $|\hat\alpha| \geq |l(\alpha)|$, the word $\SL(\hat\alpha)$ contains $y_i$ as a prefix.
\end{lemma}

\begin{proof}
Consider $\{u_j\}_{j=0}^{\infty}$ as in Lemma~\ref{lemma:recursive.u.j}, and let $u_{N-1} = y_i$. We will now show
by induction that $u_j$ is a prefix of $\SL(\hat\alpha)$ for all $ 0 \leq j < N$. The base case $j=0$ is clear
as $u_0=\SL^{ls}(m_1(\alpha)) \leq y_i  \leq \SL(\hat\alpha) < m_1(\alpha)$ by Proposition~\ref{prop:monotonicity}.

For the inductive step, we shall assume that $u_{j-1}$ is a prefix of $\SL(\hat\alpha)$, that is,
$\SL(\hat\alpha) = u_{j-1}w$ for some $w$. Consider the canonical factorization $w=w_1w_2\ldots w_m$. We note that
$u_{j-1}w_1$ is Lyndon by Lemma~\ref{lemma:seq.right.word.Lyndon}. Moreover, $\chain(\deg(u_{j-1}w_1))$ must be increasing
by Corollary~\ref{cor:lyndon.prefix.suffix.chains}, as $\chain(\alpha) \in \mathcal{I}$ and $u_{j-1}w_1$ is a Lyndon prefix.
Thus $m_1(\deg(u_{j-1}w_1)) = (\delta,i)$ by Lemma~\ref{lemma:left.standard.prefix.increasing}, which together with
$\deg(u_{j-1}) + \deg(w_1) \in \wDelta^{+,\re}$ implies $w_1\in V_j$, and so $w_1 \leq v_j$. However, $w_1 < v_j$
would contradict to $\SL(\hat\alpha) > y_i$ by Lemma~\ref{lem:Melancon}. Hence $w_1=v_j$, completing the inductive step.
\end{proof}

We note the following simple property about imaginary words.

\begin{lemma}\label{lemma:im.left.standard.form}
For any $1 \leq i \leq |I|$ and $k \in \BZ_{> 0}$:
\begin{gather}\label{eq:im.left.standard.form}
   \deg(\SL_i^{ls}(k\delta)) = \max \left\{ \alpha \in \wDelta^{+,\re} \,\bigg\vert\,
   m_1(\alpha) = (\delta,i), |\alpha| < |k\delta|, \chain(\alpha) \in \mathcal{I}\right\} .
\end{gather}
\end{lemma}

\begin{proof}
First, we clearly have LHS $\leq$ RHS as the LHS meets the three requirements of the set in the RHS by
Proposition~\ref{prop:monotonicity}, Corollary~\ref{cor:mk.im.factor}, and Proposition~\ref{prop:spanset.equiv}.
Assuming the contradiction, let LHS $<\beta=$ RHS. Then $\SL(m_k(\alpha))=\SL_i(k\delta) < \beta$ by
Lemma~\ref{lemma:equiv.to.standard.fac} and Proposition~\ref{prop:spanset.equiv}, contradicting
Proposition~\ref{prop:monotonicity}.
\end{proof}

Using this result, we can now derive two simple properties of $y_i$ and $l(\alpha)$.

\begin{corollary}\label{cor:y.i.left.standard.im}
For any $1 \leq i \leq |I|$, the following three properties are equivalent:
\begin{enumerate}[leftmargin=0.8cm]

\item
$|y_i| < |\delta|$;

\item
$y_i = \SL^{ls}_i(\delta)$;

\item
$\chain(\deg(\SL_i^{ls}(\delta)))$ is irreducible.

\end{enumerate}
\end{corollary}

\begin{proof}
We clearly have $(2) \Rightarrow (1)$, and Lemma~\ref{lemma:recursive.u.j} implies $(3) \Rightarrow (2)$. It thus remains
to prove $(1) \Rightarrow (3)$. Assume the contrary: $|y_i| < |\delta|$ but $\chain(\deg(\SL_i^{ls}(\delta)))$ is not
irreducible. Using the notation of Lemma~\ref{lemma:recursive.u.j}, we note that $|y_i| \geq |u_1| = |u_0v_1|$.
As $m_1(u_1)=(\delta,i)$, $\chain(\deg(u_1))\in \mathcal{I}$, and $u_0=\SL^{ls}_i(\delta)$, we cannot have
$|u_1| < |\delta|$, according to Lemma~\ref{lemma:im.left.standard.form} (with $k=1$). Also $|u_1| \ne |\delta|$ as
$u_1$ is real by Lemma~\ref{lemma:recursive.u.j}. Thus, $|u_1|>|\delta|$ contradicting to $|\delta| > |y_i| \geq |u_1|$.
This proves $(1) \Rightarrow (3)$.
\end{proof}

For non-irreducible increasing chains, we have the following useful lower bound.

\begin{corollary}\label{cor:l.length}
If $\chain(\alpha) \in \mathcal{I}$ is not irreducible, then $|l(\alpha)| > |\delta|$.
\end{corollary}

\begin{proof}
Let $m_1(\alpha) = (\delta,i)$. By Lemma~\ref{lemma:y.i.prefix}, we have $l(\alpha) = y_iw$ for some $w$. The word $w$
is non-empty by Lemma~\ref{lemma:y.i.irr.chain} and assumption that $\chain(\alpha)$ is not irreducible. Thus,
$\SL(l(\alpha)) > \SL_i^{ls}(\delta) $ as $\SL_i^{ls}(\delta)$ is a prefix of $y_i$.
According to Lemma~\ref{lemma:im.left.standard.form} ($k=1$ case), we obtain $|l(\alpha)| \geq |\delta|$.
As $l(\alpha)$ is not imaginary, we get $|l(\alpha)| > |\delta|$.
\end{proof}

Our next result investigates the prefix and suffix of the costandard factorization for words in increasing chains,
akin to Lemma~\ref{lemma:left.standard.decreasing} for decreasing chains.

\begin{proposition}\label{prop:cost.fac.inc}
For any $\chain(\alpha) \in \mathcal{I}$ with $m_1(\alpha)=(\delta,i)$ and $\hat\alpha \in \chain(\alpha)$ with
$|\hat\alpha| \geq |l(\alpha)|$ (if $\chain(\alpha)$ is irreducible, we rather require $|\hat\alpha| > |l(\alpha)|$),
we have $\SL(\hat\alpha) = y_iv$ by Lemma~\ref{lemma:y.i.prefix}. Consider the canonical factorization
$v = v_1\ldots v_n$ and define $\beta = \deg(\SL^l(\hat\alpha))$, $\gamma = \deg(\SL^r(\hat\alpha))$. Then:
\begin{enumerate}[leftmargin=0.8cm]

\item
$\chain(\beta) \in \mathcal{I}$;

\item
$m_1(\beta) = (\delta,i)$;

\item
$\chain(\gamma) \in \mathcal{I}$ if $\hat\alpha \not \in \chain(\beta_i)$,
and $\gamma = M_1(\beta_i)$ if $\hat\alpha \in \chain(\beta_i)$;

\item
$\SL(\gamma)=\SL^r(\hat\alpha) = v_n$;

\item
if $v_j$ is imaginary for some $j$, then $v_{j'} = v_j = \SL(M_1(\beta_i))$ for all $1 \leq j' \leq j$;

\item
if $c_\alpha > 1$, then $y_i$ is a prefix of $v_n$, $m_1(\deg(v_n)) = (\delta,i)$, $|v_n| \geq |l(\deg(v_n))|$.

\end{enumerate}
\end{proposition}

\begin{proof}
While (1) follows from Corollary~\ref{cor:lyndon.prefix.suffix.chains}, we show (2)--(6) by induction on $n$.

Let us first treat the base case $n=1$. As $\chain(\deg(y_i))$ is irreducible by Lemma~\ref{lemma:y.i.irr.chain}
and $\chain(\hat\alpha) \in \mathcal{I}$, we cannot have $\chain(\deg(v_1)) \in \mathcal{D}$. Thus,
$\chain(\deg(v_1)) \in \mathcal{I}$ if $\hat\alpha \not \in \chain(\beta_i)$, and $v_1$ is imaginary otherwise.
In the latter case, $v_1 = \SL(M_1(\beta_i))$ by Corollary~\ref{cor:imaginary.suffix.prefix}, establishing part (5).
We also note that part (4) implies (2) and (3), so that it remains to verify (4) and (6). By Corollary~\ref{cor:left.cost.right},
part (4) follows from $y_i^r \geq v_1$. Suppose instead $y_i^r < v_1$, so that $y_i^r v_1$ is standard Lyndon, as
it is a Lyndon subword of a standard Lyndon word. As $\chain(\deg(y_i^l)) \in \mathcal{I}$ by
Corollary~\ref{cor:lyndon.prefix.suffix.chains}, we note that $\chain(\deg(y_i^r)) \not \in \mathcal{I}$.
If $y_i^r$ is imaginary, we would have a contradiction with Corollary~\ref{cor:lyndon.prefix.suffix.chains}(1).
If $\chain(\deg(y_i^r)) \in \mathcal{D}$, then we argue similarly:
\begin{itemize}[leftmargin=0.7cm]

\item
if $y_i^rv_1$ is imaginary, then we have a contradiction with Corollary~\ref{cor:lyndon.prefix.suffix.chains}(3,4),

\item
if $\chain(\deg(y_i^rv_1)) \in \mathcal{I}$, then we have a contradiction with Corollary~\ref{cor:lyndon.prefix.suffix.chains}(1),

\item
if $\chain(\deg(y_i^rv_1)) \in \mathcal{D}$, then we have a contradiction with Corollary~\ref{cor:lyndon.prefix.suffix.chains}(2).

\end{itemize}
All the contradictions above finally establish part (4).
For part (6), we first note that $m_1(v_1) = (\delta,i)$ as $c_\alpha > 1$ and $c_{\beta_i} = 1$, and thus $y_i$ must be
a prefix of $v_1$ by Lemma~\ref{lemma:y.i.prefix}, which can be applied as $y_i < \SL(\hat\alpha) < v_1$. As $v_1 > y_i$,
$\chain(\deg(v_1)) \in \mathcal{I}$, and $m_1(\deg(v_1)) = (\delta,i)$, we also have $|v_1| \geq |l(\deg(v_1))|$.

Let us now proceed to the inductive step, assuming (2)--(6) hold for $n-1$. For part (4), it suffices to show that
$(y_iv_1\ldots v_{n-1})^r \geq v_{n}$ by Corollary~\ref{cor:left.cost.right}, which follows from the inductive hypothesis
$(y_iv_1\ldots v_{n-1})^r = v_{n-1}$ combined with $v_{n-1} \leq v_n$. Part (2) follows from
Lemma~\ref{lemma:left.standard.prefix.increasing} coupled with (1).

Let us now prove part (5). By the inductive hypothesis, it suffices to show that if $v_n$ is imaginary then $v_{n-1} = v_n$.
First, we note that $v_n=\SL(M_1(\hat\alpha))$ by Corollary~\ref{cor:imaginary.suffix.prefix}.
If $\hat\alpha \in \chain(\beta_i)$, then
$\deg(y_iv_1\ldots v_{n-1}) \in \chain(\beta_i)$ and by the inductive hypothesis $v_{n-1} = \SL(M_1(\beta_i)) = v_n$.
If $\hat\alpha \not \in \chain(\beta_i)$, then $\chain(\deg(v_{n-1})) \in \mathcal{I}$ by the inductive assumption, hence
$v_{n-1} < \SL(\deg(v_{n-1}) + \delta)$. Since $y_iv_1\ldots v_{n-2}v_{n-1}$ is Lyndon by Lemma~\ref{lemma:seq.right.word.Lyndon},
we get $y_iv_1\ldots v_{n-2} < v_{n-1} < \SL(\deg(v_{n-1}) + \delta)$. According to Leclerc's algorithm, we thus obtain
  $\SL(\hat\alpha)=\SL(\deg(y_iv_1\ldots v_{n-2}) + \deg(v_{n-1}) + \delta) \geq
    y_iv_1\ldots v_{n-2}\SL(\deg(v_{n-1}) + \delta) > y_iv = \SL(\hat\alpha)$
with $\SL(\deg(v_{n-1}) + \delta) > v_{n-1}v_n$ due to Lemma~\ref{lem:Melancon}, a contradiction.

We will now prove part~(3). By part~(5), if $v_n$ is imaginary then $v_1,\ldots, v_{n}$ are all imaginary and
$\deg(y_i) \in \chain(\beta_i)$ by Lemma~\ref{lemma:y.i.irr.chain}, so that $\hat\alpha \in \chain(\beta_i)$ and
$v_n=\SL(M_1(\beta_i))$. For the case when $v_n$ is real, we use the following result.

\begin{claim}\label{claim:seq_gm}
Assume we have sequences of affine roots $(\gamma_i)_{i=0}^k$, $(\mu_i)_{i=0}^k$ such that
$\mu_0 = \gamma_0$, $\mu_i = \mu_{i-1} + \gamma_i$, $\chain(\mu_i) \in \mathcal{I}$ for all $i$, $\gamma_i$ is either
imaginary or $\chain(\gamma_i) \in \mathcal{I}$ for all $i$, and $\chain(\gamma_0)$ is irreducible. Then for any real
affine root $\rho$ such that $\chain(\mu_k + \rho) \in \mathcal{I}$ and $\chain(\rho) \in \mathcal{D}$, we have
$\rho > \gamma_i$ for some $1 \leq i \leq k$.
\end{claim}

\begin{proof}[Proof of Claim~\ref{claim:seq_gm}]
We will show this by induction on $k$, the base case $k=0$ being vacuous as there are no $\rho$ satisfying the conditions.
For the inductive step, assume that the result holds for $k-1$. We consider the following two cases:
\begin{itemize}[leftmargin=0.7cm]

\item
if $\gamma_k$ is imaginary, then the result follows from the inductive hypothesis;

\item
if $\gamma_k$ is real, we first handle the case when $\gamma_k + \rho$ or $\mu_{k-1} + \rho$ is imaginary.
In the former case, we get $\gamma_k < \rho$ by Theorem~\ref{thm:convexity} and Proposition~\ref{prop:monotonicity},
while in the latter case $\chain(\mu_{k-1}) + \chain(\gamma_k) \in \mathcal{I}$ implies
$\chain(\delta - \mu_{k-1}) + \chain(\delta - \gamma_k) = \chain(\rho) + \chain(\delta - \gamma_k) \in \mathcal{D}$
so that $m_1(\gamma_k) \leq M'_1(\rho)$ by Lemma~\ref{lemma:interlock.m.M}, resulting in
$\gamma_k < m_1(\gamma_k) \leq M'_1(\rho) < \rho$ by Proposition~\ref{prop:monotonicity} and
Lemma~\ref{lemma:dec.chains.lower.bound}. We shall henceforth assume that the sum of any two elements of
$\{\gamma_k,\rho,\mu_{k-1}\}$ is real, and so by Lemma~\ref{lemma:triple.sum.pairs} either $\gamma_{k} + \rho$ or
$\mu_{k-1} + \rho$ is a real root. If $\gamma_{k} + \rho$ is a root, then we have $\gamma_{k} < \gamma_{k} + \rho < \rho$
by Corollary~\ref{cor:inc.dec.comparison}, establishing the claim. Assume now that $\mu_{k-1} + \rho$ is a real root.
If $\chain(\mu_{k-1} + \rho) \in \mathcal{I}$, then the result follows from the inductive hypothesis. Finally, we
consider the case $\chain(\mu_{k-1} + \rho) \in \mathcal{D}$.
Let $\chain(\mu_{k-1}) = c_1\chain(\beta_1) + \cdots + c_{|I|}\chain(\beta_{|I|})$ and
$\chain(\rho) = c'_1\chain(\beta_1) + \cdots + c'_{|I|}\chain(\beta_{|I|})$, so that
$c_i\geq 0, c'_i\leq 0, c_i + c'_i \leq 0$ for all $i$, implying $M'_1(\mu_{k-1}) \leq M'_1(\rho)$.
Thus we have $\gamma_k < m_1(\gamma_k) \leq M'_1(\mu_{k-1}) \leq M'_1(\rho) < \rho$ by Proposition~\ref{prop:monotonicity},
Lemma~\ref{lemma:dec.chains.lower.bound}, and Lemma~\ref{lemma:interlock.m.M}.

\end{itemize}
This completes the proof of the inductive step.
\end{proof}

Define the sequences $(\gamma_i)_{i=0}^{n-1}$ as $(\deg(y_i),\deg(v_1),\deg(v_2),\ldots,\deg(v_{n-1}))$ and
$(\mu_i)_{i=0}^{n-1}$ as $(\deg(y_i),\deg(y_i) + \deg(v_1), \ldots,\deg(y_i) + \deg(v_1) + \ldots + \deg(v_{n-1}))$.
They satisfy all the conditions in the above Claim: $\chain(\mu_i)$ are all increasing by
Corollary~\ref{cor:lyndon.prefix.suffix.chains} coupled with Lemma~\ref{lemma:seq.right.word.Lyndon},
none of $\chain(\deg(y_i)),\chain(\deg(v_1)),\ldots$ are decreasing by part~(3) of the inductive hypothesis,
and $\chain(\gamma_0)=\chain(\deg(y_i))$ is irreducible by Lemma~\ref{lemma:y.i.irr.chain}.
Thus, we cannot have $\chain(\deg(v_n)) \in \mathcal{D}$ by Claim~\ref{claim:seq_gm}, as that would imply $v_k < v_n$
for some $k < n$ contradicting the canonical factorization.

Let us now establish part~(6). If $m_1(\deg(v_n)) > (\delta,i)$, then $c_{\hat\alpha - \deg(v_n)} > 1$, and so $y_i$ is a prefix
of $v_{n-1}$ by the inductive hypothesis. Since $y_i < \SL(\hat\alpha) < v_n \leq v_{n-1}$, we see that $v_n$ has $y_i$ as
a prefix, and so $m_1(v_n) = (\delta,i)$ by Lemma~\ref{lemma:left.standard.prefix.increasing}, a contradiction.
If $m_1(\deg(v_n)) = (\delta,i)$, then combining $y_i < \SL(\hat\alpha) < v_n$ with Lemma~\ref{lemma:y.i.prefix}, we see
that $y_i$ is a prefix of $v_n$. Finally, the case $m_1(\deg(v_n)) < (\delta,i)$ cannot occur as that would imply
$m_1(\hat\alpha) < (\delta,i)$ by Corollary~\ref{cor:both.increasing.m.k}, a contradiction. Finally, as $v_n > y_i$,
$\chain(\deg(v_n)) \in \mathcal{I}$, and $m_1(\deg(v_n)) = (\delta,i)$, we also have $|v_n| \geq |l(\deg(v_n))|$.

This completes the inductive step.
\end{proof}

\begin{remark}
We note that Proposition~\ref{prop:cost.fac.inc} immediately implies the analogue of
Corollary~\ref{cor:decreasing.left.factor.imaginary} for increasing chains: for any $\alpha \in \wDelta^{+,\re}$ such that
$\chain(\alpha) \in \mathcal{I}$ with $|\alpha| > |l(\alpha)|$, if $\SL^r(\alpha)$ is imaginary,
then $\chain(\alpha)$ is irreducible (and thus $c_\alpha=1$).
\end{remark}

Combining Proposition~\ref{prop:cost.fac.inc} with $\deg(y_i) \in \chain(\beta_i)$, due to Lemma~\ref{lemma:y.i.irr.chain},
we immediately obtain the following counterpart of Corollary~\ref{cor:dec.irr.chains} for increasing chains.

\begin{corollary}\label{cor:inc.irr.chains}
$\SL(\deg(y_i) + k\delta) = y_i\underbrace{\SL_j(\delta)}_{k \text{ times}}$ for any $1 \leq i \leq |I|$, $k \in \BZ_{\geq 0}$,
where $M_1(\beta_i) = (\delta,j)$.
\end{corollary}

Combining the above corollary with Lemma~\ref{lemma:no.bound.increasing} yields the following result:

\begin{lemma}\label{lemma:inc.upper.bound}
For any $\alpha$ with $m_1(\alpha) = (\delta,i)$ and $\chain(\alpha) \in \mathcal{I}$, we have:
\begin{gather*}
  \SL(\alpha) < y_i\underbrace{\SL_j(\delta)}_{ k \text{ times}}
\end{gather*}
where $M_1(\beta_i) = (\delta,j)$ and $k \gg 1$.
\end{lemma}

We note the following analogue of Lemma~\ref{lemma:no.splitting.decreasing} for increasing chains.

\begin{lemma}
For any $\alpha \in \wDelta^{+,\re}$ if $\chain(\alpha) \in \mathcal{I}$, then the costandard factorization of $\SL(\alpha)$
does not split imaginary words. Moreover, if $\chain(\alpha)$ is not an \irrchain\ and  $|\alpha| \geq |l(\alpha)|$, then
$\SL_i(\delta)$ is not a prefix of $\SL^{r}(\alpha)$.
\end{lemma}

\begin{proof}
The first claim follows from Corollary~\ref{cor:lyndon.prefix.suffix.chains}(1) and Lemma~\ref{lemma:stronger.no.splitting}(2).
For the second claim, we note that $\chain(\deg(\SL^r(\alpha)))$ is increasing by Proposition~\ref{prop:cost.fac.inc}, and so
$\SL_i(\delta)$ cannot be its prefix by Corollary~\ref{cor:lyndon.prefix.suffix.chains}(1).
\end{proof}

Similarly to decreasing chains, it is crucial that we use the costandard factorization for increasing chains,
since the standard factorization can actually split imaginary words, as illustrated in the following example,
cf.~Example~\ref{ex:costand-decr-fails}.

\begin{example}\label{ex:stand-incr-fails}
Consider the affine type $A_2^{(1)}$ with the order $1 < 2 <0$. According to \cite[Theorem 4.2]{AT},
applied to $\chain(\alpha_1 + \alpha_2 + \delta)\in \mathcal{I}$ by Lemma~\ref{lemma:monotonicity.smallest.once}, we have:
\begin{equation*}
  \SL(\alpha_1 + \alpha_2 + \delta) = 1210|2 = 12\SL_1(\delta),
\end{equation*}
where the vertical line denotes the standard factorization.
Moreover, this example also shows that there is no analogue of Proposition~\ref{prop:cost.fac.inc}(3)
for standard factorizations: $\chain(\deg(\SL^{rs}(\alpha_1+\alpha_2+\delta)))=\chain(\alpha_2)$ is decreasing
by Lemma~\ref{lemma:monotonicity.smallest.once}.
\end{example}

We are now ready to prove the first result regarding repetitive occurrence of imaginary subwords in increasing chains,
which is a counterpart of Proposition~\ref{prop:weak.form.desc.irr.chains}.

\begin{corollary}
For any increasing chain $\chain(\alpha)$ with $m_1(\chain(\alpha)) = (\delta,i)$ and any $\hat\alpha \in \chain(\alpha)$
with $|\hat\alpha| \geq |l(\alpha)|$, we have:
\begin{equation}\label{eq:period-incr}
  \SL(\hat\alpha) = y_i\underbrace{\SL_j(\delta)}_{p \text{ times }} w
\end{equation}
for some $p \geq 0$, where $M_1(\beta_i) = (\delta,j)$, $\SL_j(\delta)$ is not a prefix of $w$, and $w < \SL_j(\delta)$.
Moreover, for any $p' > 0$, there exists a $q$ such that if $|\hat\alpha | > q|\delta|$, then $p \geq p'$.
Finally, the above function $p$ is monotone non-decreasing.
\end{corollary}

\begin{proof}
The result clearly holds for irreducible increasing $\chain(\alpha)$, due to Lemma~\ref{cor:inc.irr.chains}.
If $\chain(\alpha)$ is not irreducible, then $\SL(\hat\alpha) = y_iv$ for a nonempty word $v$ by Lemma~\ref{lemma:y.i.prefix}.
Let $v = v_1\ldots v_n$ be the canonical factorization, as discussed in Proposition~\ref{prop:cost.fac.inc}, and let $m$ be
the smallest index such that $v_m \neq \SL_j(\delta)$ (such $m$ exists since $\chain(\alpha)$ is not irreducible).
Therefore,~\eqref{eq:period-incr} holds with $w = v_m\ldots v_n$ and $p=m-1$. It remains to verify $w < \SL_j(\delta)$.
As $y_i\underbrace{\SL_j(\delta)}_{p \text{ times}}w < y_i\underbrace{\SL_j(\delta)}_{p + t \text{ times}}$ for $t\gg 1$
by Lemma~\ref{lemma:inc.upper.bound}, and $v_m\ne \SL_j(\delta)$ we get $v_m<\SL_j(\delta)$ by Lemma~\ref{lem:Melancon},
and so $w < \SL_j(\delta)$ by Lemma~\ref{lem:Melancon}.
The latter inequality also implies that the function $p$ is monotone non-decreasing, as otherwise we would find $\hat\alpha$
such that $\hat\alpha > \hat\alpha + \delta$ contradicting to $\chain(\alpha) \in \mathcal{I}$.

For any $p' > 0$, we have $\SL(l(\beta_i) + p'\delta) = y_i\underbrace{\SL_j(\delta)}_{p' \text{ times}}$
by Corollary~\ref{cor:inc.irr.chains}. Hence, $\SL(l(\beta_i) + p') < \SL(\alpha' + q\delta)$ for some $q$, due to
Lemma~\ref{lemma:no.bound.increasing}. Combining this with Lemma~\ref{lemma:inc.upper.bound}, we get
$y_i\underbrace{\SL_j(\delta)}_{p' \text{ times}} < \SL(\alpha' + q\delta) < y_i\underbrace{\SL_j(\delta)}_{t \text{ times}}$
for some $t\gg 1$, so that $\SL(\alpha' + q\delta)$ contains $y_i\underbrace{\SL_j(\delta)}_{p' \text{ times}}$ as a prefix.
Combining this with the fact that $p$ is monotone non-decreasing completes the result.
\end{proof}

The following result is analogous to Lemma~\ref{lemma:alt.u.def}.

\begin{lemma}\label{lemma:recursive.l}
For any non irreducible $\chain(\alpha) \in \mathcal{I}$ with $m_1(\alpha) = (\delta,i)$, we have:
\begin{gather}\label{eq:l}
    l(\alpha) =
    \min_{|\cdot|} \left\{l_i(\beta) + l_i(\gamma) \,\bigg\vert\,
      \substack{\chain(\beta),\chain(\gamma) \in \mathcal{I}\\
      \chain(\beta) + \chain(\gamma) = \chain(\alpha)}\right\},
\end{gather}
where $l_i(\beta) = \min_{|\cdot|} \{\hat\beta \in \chain(\beta) \,|\, \SL(\hat\beta) > y_i\}$\footnote{Existence of
such $\hat\beta$ is established in the proof.}. Moreover if $c_\alpha > 1$, then actually
\begin{gather}\label{eq:l-2}
    l(\alpha) =
    \min_{|\cdot|} \left\{l_i(\beta) + l_i(\gamma) \,\bigg\vert\,
      \substack{\chain(\beta),\chain(\gamma) \in \mathcal{I}\\
      \chain(\beta) + \chain(\gamma) = \chain(\alpha)\\
      m_1(\beta) = m_1(\gamma) = m_1(\alpha)}\right\}.
\end{gather}
\end{lemma}

\begin{proof}
First, we note that $m_1(\beta),m_1(\gamma) \geq m_1(\alpha)$ for any $\beta,\gamma$ featured in the RHS
of~\eqref{eq:l} by Corollary~\ref{cor:both.increasing.m.k}. Furthermore, $l_i(\beta) = l(\beta)$ if $m_1(\beta) = (\delta,i)$,
and $l_i(\beta)$ is well-defined by Corollary~\ref{cor:tight.bound.m} if $m_1(\beta) > (\delta,i)$. Let
$\hat\alpha = \hat\beta + \hat\gamma$ be the RHS of~\eqref{eq:l} with $\hat\beta=l_i(\beta), \hat\gamma = l_i(\gamma)$.
Let us first show that $l(\alpha) \leq |\hat\alpha|$. Without loss of generality, let $\hat\beta < \hat\gamma$, so that
$\SL(\hat\alpha) \geq \SL(\hat\beta)\SL(\hat\gamma)$ by Leclerc's algorithm. As $\SL(\hat\beta) \geq y_i$, we get
$\SL(\hat\alpha) > y_i$, implying $l(\alpha) \leq |\hat\alpha|$. Let us now prove $l(\alpha) \geq |\hat\alpha|$.
Consider the costandard factorization $\SL(l(\alpha)) = \SL(\mu)\SL(\eta)$. By Proposition~\ref{prop:cost.fac.inc}
we have $\chain(\mu),\chain(\eta)\in \mathcal{I}$ and $y_i \leq \SL(\mu) < \SL(\eta)$, so that
$|\mu| \geq |l_i(\mu)|, |\gamma| \geq |l_i(\gamma)|$. Hence $l(\alpha) \geq |\hat\alpha|$, completing the proof of~\eqref{eq:l}.
Furthermore, if $c_\alpha  > 1$, then $m_1(\mu) = m_1(\eta) = m_1(\alpha)$ by Proposition~\ref{prop:cost.fac.inc},
establishing~\eqref{eq:l-2}.
\end{proof}


\subsection{Periodicity}
\

Our main goal is to describe the structure of increasing chains $\chain(\alpha)$ once the prefix $y_i$ starts appearing
(whereas $m_1(\alpha)=(\delta,i)$), that is, for $|\hat\alpha| \geq |l(\alpha)|$. Akin to Theorem~\ref{thm:dec.modulo.s},
we shall see that increasing chains have \emph{periodicity} $c_\alpha$, and thus it essentially suffices to calculate
$\SL(\hat\alpha)$ for $\hat\alpha\in \chain(\alpha)$ with $|\hat\alpha| < |l(\alpha) + c_\alpha \delta|$. Similarly to
decreasing chains, given a standard Lyndon word of the form $uvw$ where $v$ has a form $y_i\underbrace{\SL_j(\delta)}_{t \text{ times}}$
and $w$ does not start with $\SL_j(\delta)$, we will refer to $v$ as a \textbf{chunk}.

The following result is an analogue of Lemma~\ref{lemma:form.dec.words} for increasing chains and establishes $c_\alpha$
as the upper bound in the number of chunks for increasing chains.

\begin{lemma}\label{lem:chunks-increasing}
For any $\chain(\alpha) \in \mathcal{I}$ and $\hat\alpha \in \chain(\alpha)$ satisfying $|\hat\alpha| \geq |l(\alpha)|$,
we have:
\begin{gather}\label{eq:form.of.increasing.word}
  \SL(\hat\alpha) = y_i\underbrace{\SL_j(\delta)}_{p_1 \text{ times}} w_1
  y_i\underbrace{\SL_j(\delta)}_{p_2 \text{ times}}w_2  \ldots
  y_i\underbrace{\SL_j(\delta)}_{p_{c_\alpha} \text{ times}}w_{c_\alpha},
\end{gather}
whereas $m_1(\alpha) = (\delta,i)$ and $M_1(\beta_i) = (\delta,j)$. Furthermore, (possibly empty) words $w_k$ do not
contain $\SL_j(\delta)$ as a prefix nor $y_i$ as a subword, and satisfy $w_k < \SL_j(\delta)$.
\end{lemma}

\begin{proof}
Let $m_1(\alpha) = (\delta,i)$. We will prove the result by induction on $\hgt(\chain(\alpha))$. The base case
$\chain(\alpha)=\chain(\beta_i)$ follows from Corollary~\ref{cor:inc.irr.chains} and Lemma~\ref{lemma:y.i.irr.chain}.
Suppose now the result holds for increasing chains of the relative height $<\hgt(\chain(\alpha))$.
By Lemma~\ref{lemma:y.i.prefix}, we have $\SL(\hat\alpha) = y_iv$ and let $v = v_1\ldots v_n$ be the canonical factorization.
By Proposition~\ref{prop:cost.fac.inc}, if $v_a$ is imaginary then so are $v_{a-1},v_{a-2},\ldots,v_1$ and we further have
$v_1=\ldots=v_a=\SL(M_1(\beta_i))=\SL_j(\delta)$. Thus, if $p$ denotes the smallest index such that $v_p$ is not imaginary,
then we get $\SL(\hat\alpha) = y_i\underbrace{\SL_j(\delta)}_{p-1 \text{ times}}v_p\ldots v_n$.

First, we consider the case $c_\alpha = 1$. According to Proposition~\ref{prop:cost.fac.inc}(1,3,4), both chains
$\chain(\hat\alpha - \deg(v_n))$ and $\chain(\deg(v_n))$ are increasing. Thus
$y_i\underbrace{\SL_j(\delta)}_{p-1 \text{ times}}v_p\ldots v_{n-1}$ can be expressed as in~\eqref{eq:form.of.increasing.word}
by the inductive assumption, while $v_n$ cannot contain $\SL_j(\delta)$ as a prefix by
Corollary~\ref{cor:lyndon.prefix.suffix.chains}(1). By Lemma~\ref{lemma:lyndon.subword} and the inductive hypothesis,
it remains to show that $y_i$ is a not a subword of $v_n$. If $y_i$ was a subword of $v_n$ but not a prefix, then by
properties of Lyndon words we would have $y_i < \SL(\hat\alpha) < v_n < y_i\ldots$, so that $v_n$ contains $y_i$ as a prefix.
However, if $y_i$ was a prefix of $v_n$, then $m_1(\deg(v_n)) = (\delta,i)$ by Lemma~\ref{lemma:left.standard.prefix.increasing},
contradicting $c_\alpha = 1$.

Now, if $c_\alpha > 1$, then by Proposition~\ref{prop:cost.fac.inc}(6) and the inductive hypothesis, $v_n$ is of the form
as in~\eqref{eq:form.of.increasing.word} (possibly with more than one chunk) as well as
$\SL^l(\hat\alpha)=y_i\underbrace{\SL_j(\delta)}_{p_1 \text{ times}}v_1\ldots v_{n-1}$, hence so is their concatenation
$\SL(\hat\alpha)$.
\end{proof}

By the above proof, we note the following property of~\eqref{eq:form.of.increasing.word}:
\begin{equation}\label{eq:chunk-incr}
\begin{split}
  u_t=y_i\underbrace{\SL_j(\delta)}_{p_t \text{ times}}w_t \ \
  & \mathrm{are\ real\ standard\ Lyndon\ words\ satisfying} \\
  & m_1(\deg(u_t)) = (\delta,i),\ \chain(\deg(u_t)) \in \mathcal{I}.
\end{split}
\end{equation}

As mentioned in the last sentence of the above proof, if $c_\alpha > 1$ then $c_{\deg(v_n)}$ may be $>1$
(so that $v_n$ has $>1$ chunk), as illustrated in the following example.

\begin{example}
Consider the affine type $G_2^{(1)}$ with the order $0 < 1 < 2$. Then using the code we find for $k \geq 3$:
\begin{multline*}
  \SL(\alpha_0 + \alpha_1 + k\delta) = \\
  \begin{cases}
    01221\ub{1}{k/3-1} \big|\ 01221 \ub{1}{k/3-1} 01221 \ub{1}{k/3-1} 01222 & \mathrm{if}\ k \equiv 0 \ \mathrm{mod} \ 3 \\
    01221 \ub{1}{\lfloor k/3 \rfloor -1} \big|\ 01221 \ub{1}{\lfloor k/3 \rfloor -1} 0122201221 \ub{1}{\lceil k/3 \rceil -1}
      & \mathrm{if}\ k \equiv 1 \ \mathrm{mod} \ 3 \\
    01221 \ub{1}{\lfloor k/3 \rfloor-1} 0122201221 \ub{1}{\lceil k/3 \rceil -1} \big|\ 01221 \ub{1}{\lceil k/3 \rceil -1}
    & \mathrm{if}\ k \equiv 2 \ \mathrm{mod} \ 3 \\
  \end{cases}
\end{multline*}
where the vertical line represents the costandard factorization. Here, we have $\SL_1(\delta) = 012221$, while additionally
$m_1(\alpha_0 + \alpha_1) = (\delta,2)$, so that $y_2 = 01221$.
\end{example}

The next result gives an upper bound on the number of imaginary words in~\eqref{eq:form.of.increasing.word}.

\begin{lemma}\label{lemma:p.i.inc}
For any $\chain(\alpha) \in \mathcal{I}$, applying~\eqref{eq:form.of.increasing.word} to $\hat\alpha = l(\alpha) + k\delta$
we have $\sum_{j=1}^{c_\alpha} p_j \leq k$.
\end{lemma}

\begin{proof}
We show this by induction on the relative height of $\chain(\alpha)$. The base case corresponds to irreducible
$\chain(\alpha)$ and thus follows from Corollary~\ref{cor:inc.irr.chains} and Lemma~\ref{lemma:y.i.irr.chain}.
For the inductive step, suppose the result holds for all increasing chains of relative height $<\hgt(\chain(\alpha))$.
Let $\SL(l(\alpha)+k\delta) = \SL(\beta)\SL(\gamma)$ be the costandard factorization.

If $c_\alpha = 1$, then as $y_i < \SL(l(\alpha) + k\delta) < \SL(\gamma)$, we have $|\gamma| \geq |l_i(\gamma)|$.
Suppose that $p_1 > k$, then we must have that $p_1 > k$ for $\SL(\beta)$ by Lemma~\ref{lemma:stronger.no.splitting}
and additionally noting that $\SL_j(\delta)$ cannot be a prefix of $\SL(\gamma)$, as $\chain(\gamma) \in \mathcal{I}$ by
Proposition~\ref{prop:cost.fac.inc} and imaginary words cannot be a prefix of standard Lyndon words which are associated
with increasing chains by Corollary~\ref{cor:lyndon.prefix.suffix.chains}. Hence by the inductive hypothesis,
$|\beta| > |l(\beta) + k\delta| = |l_i(\beta) + k\delta|$. Thus, we obtain $|l(\alpha)| > |l_i(\beta) + l_i(\gamma)|$,
contradicting Lemma~\ref{lemma:recursive.l}.

If $c_\alpha > 1$, then $m_1(\beta) = m_1(\gamma) = m_1(\alpha)$ by Proposition~\ref{prop:cost.fac.inc}.
Writing $\beta = l(\beta) + k_1\delta, \gamma = l(\gamma) + k_2\delta$, we note that $k\geq k_1 + k_2$ by
Lemma~\ref{lemma:recursive.l}. By the inductive hypothesis, we note that $\SL(\beta),\SL(\gamma)$ have at most
$k_1,k_2$ occurrences of $\SL_j(\delta)$, respectively. By Lemma~\ref{lemma:stronger.no.splitting}, we then see
that $\SL(l(\alpha)+k\delta) = \SL(\beta)\SL(\gamma)$ has $\leq k_1+k_2\leq k$ occurrences of $\SL_j(\delta)$.

Thus in either case we must have $\sum_j p_j \leq k$, completing the inductive step.
\end{proof}

The following analogue of Lemma~\ref{lemma:lifting.chains.dec} is key to the proof of Theorem~\ref{thm:incr.modulo.c}.

\begin{lemma}\label{lemma:lifting.chains}
Consider four words $w,w',v,v'$ of the following form:
\begin{align*}
  w &= y\underbrace{\ell}_{p_1 \text{ times}}w_1y \underbrace{\ell}_{p_2 \text{ times}}w_2y \ldots
      \underbrace{\ell}_{p_{t} \text{ times}}w_t , \\
  w'&= y\underbrace{\ell}_{p_1+1 \text{ times}}w_1y \underbrace{\ell}_{p_2+1 \text{ times}}w_2y \ldots
      \underbrace{\ell}_{p_{t}+1 \text{ times}}w_t , \\
  v &= y\underbrace{\ell}_{q_1 \text{ times}}v_1y \underbrace{\ell}_{q_2 \text{ times}}v_2y \ldots
      \underbrace{\ell}_{q_{s} \text{ times}}v_s , \\
  v'&= y\underbrace{\ell}_{q_1+1 \text{ times}}v_1y \underbrace{\ell}_{q_2+1 \text{ times}}v_2y \ldots
      \underbrace{\ell}_{q_{s}+1 \text{ times}}v_{s} ,
\end{align*}
with $\ell \in \LL$ and $y$ being nonempty, which satisfy the following conditions for all $i$:
\begin{enumerate}[leftmargin=0.8cm]

\item
any suffix of $w_i,v_i$ is greater than $y$;

\item
$y$ is not a subword of $w_i$ and $v_i$;

\item
$\ell$ is not a prefix of $w_i,v_i$;

\item
$\ell > y,w_i,v_i$.

\end{enumerate}
Then, we have $w > v \Leftrightarrow w' > v'$.
\end{lemma}

\begin{proof}
We will first show the ``$\Rightarrow$'' direction, by induction on $\max\{s,t\}$. The result in the base case $s = t = 1$
clearly holds as $\ell$ is inserted at the same spot for both $w,v$ to produce words $w',v'$. For the inductive hypothesis,
suppose the result holds for all $s',t'$ with $\max\{s',t'\} < \max \{s,t\}$. Note that when comparing two words we can
repeatedly remove common prefixes, hence we shall assume that $\min\{p_1,q_1\} = 0$. We consider three cases:
\begin{itemize}[leftmargin=0.7cm]

\item
If $p_1 > q_1 = 0$, then we first claim that $\ell > v_1y$. This follows from (4) unless $v_1$ is a proper prefix of $\ell$,
but in the latter case $\ell=v_1u$ we also have $u>\ell>y$ by (4) and $\ell\in \LL$. If $q_2 = 0$, then $\ell > v_1y$ as $w > v$.
Then we clearly get $w' > v'$ unless $v_1y$
is a prefix of $\ell$. In the latter case, we have $\ell=v_1yu$ for some non-empty $u$ and $u>\ell$ since $\ell \in \LL$, so that
$\ell > v_1y\ell$ and hence $w' > v'$. If $q_2 > 0$, then $w>v$ already implies $\ell > v_1y\ell$ and thus $w' > v'$;

\item
Next, assume that $0 = p_1 < q_1$. We note that it is impossible for all $p_i = 0$, as that would imply $\ell < w_1yw_2\ldots yw_t$,
contradicting to (4) and $\ell\in \LL$ is Lyndon. Let $k>1$ be the smallest index such that $p_k > 0$, so that
$\ell < w_1y\ldots w_{k-1}y\ell$. If $w_1y\ldots w_{k-1}y$ is a prefix of $\ell$, i.e.\ $\ell=w_1y\ldots w_{k-1}y u$, then $u<\ell$
contradicting the fact that $\ell$ is Lyndon. However, if $w_1y\ldots w_{k-1}y$ is not a prefix of $\ell$, then
$\ell < w_1y\ldots w_{k-1}y\ell$ implies $\ell < w_1y\ldots w_{k-1}y$, contradicting to (4) and $\ell \in \LL$;

\item
Finally, suppose that $p_1 = q_1 = 0$. In this case, we consider three sub-cases:
\begin{itemize}[leftmargin=0.7cm]

\item
If $w_1 = v_1$, then the result follows from the inductive assumption applied to the pair of words
$y\underbrace{\ell}_{p_2 \text{ times}}w_2\ldots$ and $y\underbrace{\ell}_{q_2 \text{ times}}v_2\ldots$
(obtained from $w,v$ by erasing the common prefix $yw_1=yv_1$);

\item
Now suppose $w_1 < v_1$. As $w > v$, we must then have that $w_1$ is a prefix of $v_1$, i.e.\ $v_1=w_1u$ for some $u$.
As $y$ is nonempty, then $u > y$ by~(1) and $y$ cannot be a prefix of $u$ by (2), contradicting to $w>v$;

\item
If $w_1 > v_1$, then we clearly have $w' > v'$ unless $v_1$ is a prefix of $w_1$. In the latter case, $w_1=v_1u$
for some non-empty $u$. As $y$ is non-empty, then $w' > v'$ as $u > y$ and $y$ is not a prefix of $u$ by (1)--(2).

\end{itemize}

\end{itemize}
This exhausts all possible cases, thus completing the inductive step.

We note that the ``$\Leftarrow$'' direction follows by applying the above to the contrapositive as all conditions
(1)--(7) are symmetric.
\end{proof}

The following result will be useful in the proof of the theorem below.

\begin{lemma}\label{lemma:delta.length.suffix.l}
For any increasing chain $\chain(\alpha)$, if $|\alpha| > |l(\alpha)|$ and $w$ is a suffix of $\SL(\alpha)$ of length
$|w| = |\delta|$, then $w > y_i\SL_j(\delta)$ with $i,j$ as in Lemma~\ref{lem:chunks-increasing}.
\end{lemma}

\begin{proof}
Consider the decomposition~\eqref{eq:form.of.increasing.word} of $\SL(\alpha)$. Let $w' = w_{c_\alpha}$.
We shall now treat two cases determined by whether or not $p_{c_\alpha}$ vanishes:
\begin{itemize}[leftmargin=0.7cm]

\item
If $p_{c_\alpha} = 0$, then $y_iw'$ is a suffix of $\SL(\alpha)$. We note that $w'$ cannot be empty, as that would
contradict to $\SL(\alpha)$ being Lyndon. By~\eqref{eq:chunk-incr}, $y_iw'=\SL(\beta)$ for some
$\beta\in \wDelta^{+,\re}$ with $m_1(\beta)=(\delta,i)$, $\chain(\beta) \in \mathcal{I}$, and thus $|\beta| \geq |l(\beta)|$.
Note that $\chain(\beta)$ is not an \irrchain, due to Corollary~\ref{cor:inc.irr.chains} and $w'$ being non-empty, and so
$|y_iw'| > |\delta|$ by Corollary~\ref{cor:l.length}. Thus $w$ is a proper suffix of $y_iw'$, so that $w>y_iw'$.
We claim that $y_i$ is not a prefix of $w$. Indeed, since $y_i$ is not a subword of $w'$ by Lemma~\ref{lem:chunks-increasing},
if $y_i$ was a prefix of $w$, then some nonempty prefix of $y_i$ would be a suffix of $y_i$, contradicting $y_i$ being Lyndon.
Thus as $w > y_iw' > y_i$ and $y_i$ is not a prefix of $w$, we get $w > y_i\SL_j(\delta)$.

\item
If $p_{c_\alpha}\ne 0$, then $\SL_j(\delta)w'$ is a suffix of $\SL(\alpha)$, hence $w$ is a suffix of $\SL_j(\delta)w'$.
As $y_i\underbrace{\SL_j(\delta)}_{p_{c_\alpha} \text{ times}}w'$ is Lyndon by~\eqref{eq:chunk-incr}, we have
$w > y_i \underbrace{\SL_j(\delta)}_{p_{c_\alpha} \text{ times}}w' > y_i\SL_j(\delta)$.

\end{itemize}
Thus we obtain the claimed inequality $w > y_i\SL_j(\delta)$ in both cases.
\end{proof}

The following is our key structural result featuring the periodicity pattern for increasing chains,
which should be viewed as a natural analogue of Theorem~\ref{thm:dec.modulo.s}.

\begin{theorem}\label{thm:incr.modulo.c}
For any $\chain(\alpha) \in \mathcal{I}$, let $m_1(\alpha) = (\delta,i)$ and $M_1(\deg(y_i)) = (\delta,j)$.
If $\SL(l(\alpha) + k\delta)$ is of the form~\eqref{eq:form.of.increasing.word}
for $k \in \{0,1,\ldots, c_\alpha-1\}$, and $k' - k = qc_\alpha$
with $q \in \BZ_{\geq 0}$, then:
\begin{itemize}[leftmargin=0.7cm]

\item[(1)]
We have
\begin{gather}\label{eq:inc.chains.form}
  \SL(l(\alpha) + k'\delta) =
  y_i\underbrace{\SL_j(\delta)}_{q \text{ times}}w_1
  y_i\underbrace{\SL_j(\delta)}_{p_2 + q \text{ times}}w_2 \cdots
  y_i\underbrace{\SL_j(\delta)}_{p_{c_\alpha} +q\text{ times}}w_{c_\alpha}.
\end{gather}
Furthermore, we have $p_n\in \{0,1\}$ for any $2 \leq n \leq c_\alpha$.

\item[(2)]
If $\chain(\alpha) = \chain(\beta_i)$, then $\SL(l(\alpha) + k\delta) = \SL(l(\alpha) + (k-1)\delta)\SL_j(\delta)$ is
the costandard factorization for all $k \geq 1$. Otherwise, if $\beta = \deg(\SL^l(l(\alpha) + k\delta))$ then
$\deg(\SL^l(l(\alpha) + k'\delta)) = \beta + qc_{\beta}\delta$.

\end{itemize}
\end{theorem}

\begin{proof}
This will be proven by induction on $\hgt(\chain(\alpha))$. The base case corresponds to irreducible $\chain(\alpha)$,
in which case the results follow from Corollary~\ref{cor:inc.irr.chains}, Lemma~\ref{lemma:y.i.irr.chain}, and
Proposition~\ref{prop:cost.fac.inc}. Before proceeding to the inductive step, let us first establish the following result
(analogous to Claim~\ref{claim:aux} for decreasing chains):

\begin{claim}\label{claim:aux.incr}
If Theorem~\ref{thm:incr.modulo.c} holds for $\chain(\beta),\chain(\gamma) \in \mathcal{I}$ with $m_1(\beta) = m_1(\gamma)$,
then $\SL(l(\beta) + k_1\delta) < \SL(l(\gamma) + k_2\delta)$ for some $k_1,k_2 \in \BZ_{\geq 0}$ if and only if
$\SL(l(\beta) + (k_1 + c_\beta)\delta) < \SL(l(\gamma) + (k_2 + c_\gamma)\delta)$.
\end{claim}

\begin{proof}
We will first prove the ``$\Rightarrow$'' direction.
Consider the words $\SL(l(\beta) + k_1\delta)$, $\SL(l(\gamma) + k_2\delta)$
parsed as in~\eqref{eq:inc.chains.form}. Let $p$ be the index of the first different letter between these words,
and set $p'=p+|\delta|$. By our assumption, the length $p'-1$ prefixes of words $\SL(u(\beta) + (k_1+s_\beta)\delta)$
and $\SL(u(\gamma) + (k_2+s_\gamma)\delta)$ coincide, while their $p'$-th letters coincide with $p$-th letters of
$\SL(u(\beta) + k_1\delta)$ and $\SL(u(\gamma) + k_2\delta)$, respectively, implying the result.

The other direction ``$\Leftarrow$'' follows immediately from applying the previous result to the contrapositive.
\end{proof}

For the step of induction, we shall assume that the result holds for all chains of relative height $<\hgt(\chain(\alpha))$.
We will first prove part (2). It suffices to show that if $\beta = \deg(\SL^l(l(\alpha) + k\delta))$ for $k \in \BZ_{\geq 0}$,
then $\deg(\SL^l(l(\alpha) + (k + c_\alpha)\delta)) = \beta + c_\beta \delta$. Let $\gamma = \deg(\SL^r(l(\alpha) + k\delta))$
and consider $\beta' = \deg(\SL^l(l(\alpha) + (k+c_\alpha)\delta))$, $\gamma' = \deg(\SL^r(l(\alpha) + (k + c_\alpha)\delta))$.
As $\chain(\alpha)$ is not irreducible, roots $\gamma,\gamma'$ are real by Proposition~\ref{prop:cost.fac.inc}.
Let us first establish $\beta + c_\beta\delta < \gamma + (c_\alpha - c_\beta)\delta$ by analyzing several cases:
\begin{itemize}[leftmargin=0.7cm]

\item
if $m_1(\gamma) > m_1(\beta) = m_1(\alpha)$, then $c_\alpha - c_\beta = 0$ and $y_i$ is not a prefix of $\SL(\gamma)$
by Lemma~\ref{lemma:left.standard.prefix.increasing}. Hence $\beta + c_\beta \delta < \gamma$, as $y_i$ is a prefix of
both $\SL(\beta)$ and $\SL(\beta + c_\beta\delta)$;

\item
if $m_1(\gamma) = m_1(\beta)$, then $c_\gamma=c_\alpha-c_\beta$, and $\beta + c_\beta\delta < \gamma + c_\gamma\delta$
by Claim~\ref{claim:aux.incr};

\item
if $m_1(\gamma) < m_1(\beta)$, then we get a contradiction with Corollary~\ref{cor:both.increasing.m.k} as $\alpha=\beta+\gamma$,
$m_1(\beta) = m_1(\alpha)$, and chains $\chain(\beta),\chain(\gamma)$ are increasing by Proposition~\ref{prop:cost.fac.inc}.

\end{itemize}
Hence by Leclerc's algorithm, we get
\begin{equation}\label{eq:aux-thm99}
  \SL(\beta + c_\beta\delta)\SL(\gamma + (c_\alpha - c_\beta)\delta) \leq \SL(\beta')\SL(\gamma').
\end{equation}
We also note that $y_i\SL_j(\delta)$ is a prefix of $\SL(\beta + c_\beta\delta)$ by the induction assumption. Therefore
$y_i\SL_j(\delta) < \SL(\beta + c_\beta\delta) < \SL(\beta')\SL(\gamma') < y_i\underbrace{\SL_j(\delta)}_{t \text{ times}}$
for $t \gg 0$, with the last inequality due to Lemma~\ref{lemma:inc.upper.bound}. Thus $y_i\SL_j(\delta)$ is a prefix of
$\SL(\beta')\SL(\gamma')$, and hence it is also a prefix of $\SL(\beta')$ by Proposition~\ref{prop:cost.fac.inc}
(since $\SL(\gamma')$ is the smallest element of the canonical factorization of $v$ featured in
$\SL(\beta')\SL(\gamma') = y_i\SL_j(\delta)v$). This implies $|\beta'| \geq |l(\beta') + c_{\beta'}\delta|$ by
the induction assumption.

Let us now prove
\begin{equation*}
  \beta' -c_{\beta'}\delta < \gamma' - (c_\alpha - c_{\beta'})\delta
\end{equation*}
by analyzing two cases:
\begin{itemize}[leftmargin=0.7cm]

\item
if $c_\alpha = 1$, then this follows from
$\beta' - c_{\beta'}\delta < \beta' < \gamma' =  \gamma' - (c_\alpha - c_{\beta'})\delta$;

\item
if $c_\alpha > 1$, then $m_1(\beta') = m_1(\gamma') = m_1(\alpha)$ by Proposition~\ref{prop:cost.fac.inc} and
$|\gamma'| \geq |l(\gamma') + c_{\gamma'}\delta|$ by the inductive hypothesis as $y_i\SL_j(\delta)$ is a prefix of
$\SL(\gamma')$, since $y_i\SL_j(\delta) \leq \SL(\beta') < \SL(\gamma') < y_i\underbrace{\SL_j(\delta)}_{ \gg 0 \text{ times}}$.
Therefore, the claimed inequality
$\beta' - c_{\beta'}\delta < \gamma' - (c_\alpha - c_{\beta'})\delta = \gamma' - c_{\gamma'}\delta$ follows from
$\beta' < \gamma'$ due to Claim~\ref{claim:aux.incr}.

\end{itemize}
Thus by Leclerc's algorithm, we have
\begin{equation}\label{eq:aux1-thm99}
  \SL(\beta' - c_{\beta'}\delta)\SL(\gamma' - (c_{\alpha} - c_{\beta'})\delta) \leq \SL(\beta)\SL(\gamma).
\end{equation}
We shall now combine the inductive hypothesis with Lemma~\ref{lemma:lifting.chains} to show that in fact we must have
an equality in~\eqref{eq:aux1-thm99}. By the inductive hypothesis, we have:
\begin{align*}
  \SL(\beta) &= y_i\underbrace{\SL_j(\delta)}_{r_1 \text{ times}}z_1 \ldots
     y_i \underbrace{\SL_j(\delta)}_{r_{c_\beta} \text{ times}}z_{c_\beta} , \\
  \SL(\beta + c_\beta\delta) &= y_i\underbrace{\SL_j(\delta)}_{r_1+1 \text{ times}}z_1 \ldots
     y_i \underbrace{\SL_j(\delta)}_{r_{c_\beta} +1 \text{ times}}z_{c_\beta} , \\
  \SL(\beta') &= y_i\underbrace{\SL_j(\delta)}_{r'_1+1 \text{ times}}z'_1 \ldots
     y_i \underbrace{\SL_j(\delta)}_{r'_{c_{\beta'}+1} \text{ times}}z'_{c_{\beta'}} , \\
  \SL(\beta' - c_{\beta'}\delta) &= y_i\underbrace{\SL_j(\delta)}_{r'_1 \text{ times}}z'_1 \ldots
     y_i \underbrace{\SL_j(\delta)}_{r'_{c_{\beta'}} \text{ times}}z'_{c_{\beta'}} .
\end{align*}
If $c_\alpha = 1$, then $c_{\beta}=1=c_{\beta'}$ and we have:
\begin{align*}
    & \SL(\beta)\SL(\gamma) = y_i\underbrace{\SL_j(\delta)}_{r_1 \text{ times}}z_1 \SL(\gamma) ,\qquad
    \SL(\beta + c_\beta\delta)\SL(\gamma) = y_i\underbrace{\SL_j(\delta)}_{r_1+1 \text{ times}}z_1 \SL(\gamma) ,\\
    & \SL(\beta')\SL(\gamma') = y_i\underbrace{\SL_j(\delta)}_{r'_1+1 \text{ times}}z'_1 \SL(\gamma') ,\qquad
    \SL(\beta' - c_{\beta'}\delta)\SL(\gamma') = y_i\underbrace{\SL_j(\delta)}_{r'_1 \text{ times}}z'_1 \SL(\gamma').
\end{align*}
If $c_\alpha > 1$, then $m_1(\gamma) = m_1(\gamma') = m_1(\alpha)$, and by the inductive hypothesis we have:
\begin{align*}
    \SL(\gamma) &= y_i\underbrace{\SL_j(\delta)}_{t_1 \text{ times}}x_1 \ldots
       y_i \underbrace{\SL_j(\delta)}_{t_{c_\gamma} \text{ times}}x_{c_\gamma} , \\
    \SL(\gamma + c_\gamma\delta) &= y_i\underbrace{\SL_j(\delta)}_{t_1+1 \text{ times}}x_1 \ldots
       y_i \underbrace{\SL_j(\delta)}_{t_{c_\gamma} +1 \text{ times}}x_{c_\gamma} , \\
    \SL(\gamma') &= y_i\underbrace{\SL_j(\delta)}_{t'_1+1 \text{ times}}x'_1 \ldots
       y_i \underbrace{\SL_j(\delta)}_{t'_{c_{\gamma'}+1} \text{ times}}x'_{c_{\gamma'}} , \\
    \SL(\gamma' - c_{\gamma'}\delta) &= y_i\underbrace{\SL_j(\delta)}_{t'_1 \text{ times}}x'_1 \ldots
       y_i \underbrace{\SL_j(\delta)}_{t'_{c_{\gamma'}} \text{ times}}x'_{c_{\gamma'}} ,
\end{align*}
so that
\begin{align*}
  & \SL(\beta)\SL(\gamma) = y_i\underbrace{\SL_j(\delta)}_{r_1 \text{ times}}z_1 \ldots
    y_i \underbrace{\SL_j(\delta)}_{r_{c_\beta} \text{ times}}z_{c_\beta}y_i\underbrace{\SL_j(\delta)}_{t_1 \text{ times}}x_1 \ldots
    y_i \underbrace{\SL_j(\delta)}_{t_{c_\gamma} \text{ times}}x_{c_\gamma} , \\
  & \SL(\beta + c_\beta\delta)\SL(\gamma + c_\gamma\delta) = \\
  & \qquad \qquad \qquad \quad
    y_i\underbrace{\SL_j(\delta)}_{r_1+1 \text{ times}}z_1 \ldots
    y_i \underbrace{\SL_j(\delta)}_{r_{c_\beta} +1 \text{ times}}z_{c_\beta}y_i\underbrace{\SL_j(\delta)}_{t_1+1 \text{ times}}x_1 \ldots
    y_i \underbrace{\SL_j(\delta)}_{t_{c_\gamma} +1 \text{ times}}x_{c_\gamma} , \\
  & \SL(\beta')\SL(\gamma') = y_i\underbrace{\SL_j(\delta)}_{r'_1+1 \text{ times}}z'_1 \ldots
    y_i \underbrace{\SL_j(\delta)}_{r'_{c_{\beta'}+1} \text{ times}}z'_{c_{\beta'}}y_i\underbrace{\SL_j(\delta)}_{t'_1+1 \text{ times}}x'_1
    \ldots y_i \underbrace{\SL_j(\delta)}_{t'_{c_{\gamma'}+1} \text{ times}}x'_{c_{\gamma'}} , \\
  & \SL(\beta' - c_{\beta'}\delta)\SL(\gamma'- c_{\gamma'}\delta) = \\
  & \qquad \qquad \qquad \quad
    y_i\underbrace{\SL_j(\delta)}_{r'_1 \text{ times}}z'_1 \ldots
    y_i \underbrace{\SL_j(\delta)}_{r'_{c_{\beta'}} \text{ times}}z'_{c_{\beta'}}y_i\underbrace{\SL_j(\delta)}_{t'_1 \text{ times}}x'_1
    \ldots y_i \underbrace{\SL_j(\delta)}_{t'_{c_{\gamma'}} \text{ times}}x'_{c_{\gamma'}} .
\end{align*}
However, in either case, we can apply Lemma~\ref{lemma:lifting.chains}, with $y = y_i$, $w = \SL(\beta)\SL(\gamma)$,
and $v= \SL(\beta' - c_{\beta'}\delta)\SL(\gamma' - c_{\gamma'}\delta)$ to see that
$\SL(\beta' - c_{\beta'}\delta)\SL(\gamma'-(c_\alpha - c_{\beta'})\delta) < \SL(\beta)\SL(\gamma)$ implies
$\SL(\beta')\SL(\gamma') < \SL(\beta + c_\beta\delta) \SL(\gamma + (c_\alpha - c_\beta)\delta)$, which contradicts
to~\eqref{eq:aux-thm99}. To this end, we note that all the conditions of Lemma~\ref{lemma:lifting.chains} hold:
(1) is due to~\eqref{eq:chunk-incr}, (2)--(3) follow from Lemma~\ref{lem:chunks-increasing}, (4) follows from
Proposition~\ref{prop:monotonicity} combined with Proposition~\ref{prop:cost.fac.inc} and Lemma~\ref{lem:chunks-increasing}.

Thus, we must have equalities in~\eqref{eq:aux1-thm99} and~\eqref{eq:aux-thm99}:
\begin{equation}\label{eq:aux2-thm99}
\begin{split}
  \SL(\beta' - c_{\beta'}\delta)\SL(\gamma' - (c_\alpha - c_{\beta'})\delta) &= \SL(\beta)\SL(\gamma), \\
   \SL(\beta + c_\beta\delta)\SL(\gamma + (c_\alpha - c_\beta)\delta) &= \SL(\beta')\SL(\gamma').
\end{split}
\end{equation}
As both right-hand sides above are costandard factorizations, we must have
$|\gamma'| \geq |\gamma + (c_\alpha - c_\beta)\delta|$ and $|\gamma| \geq |\gamma' - (c_\alpha - c_{\beta'})\delta|$.
If $c_\alpha = 1$, then $c_\beta=1$ and $|\gamma'| = |\gamma|$, so that $\gamma' = \gamma$ and
$\beta' = \beta + c_\beta\delta$ in~\eqref{eq:aux2-thm99}.
If $c_\alpha > 1$, then $m_1(\gamma) = m_1(\gamma') = m_1(\alpha)$ by Proposition~\ref{prop:cost.fac.inc} and so
$c_\alpha-c_\beta=c_\gamma$, $c_{\alpha}-c_{\beta'}=c_{\gamma'}$. Applying Claim~\ref{claim:aux.incr} again, we see that
$\gamma < \gamma' - (c_\alpha - c_{\beta'})\delta \Leftrightarrow \gamma + (c_{\alpha} - c_\beta)\delta < \gamma'$.
If $|\gamma'| > |\gamma + (c_\alpha - c_{\beta})\delta|$, then $\gamma' < \gamma + (c_\alpha - c_{\beta})\delta$
by~\eqref{eq:aux2-thm99}, and hence $\gamma' - (c_\alpha - c_{\beta'})\delta < \gamma$, contradicting~\eqref{eq:aux2-thm99}.
Thus $|\gamma'| = |\gamma + (c_{\alpha} - c_\beta)\delta|$, and so $\gamma' = \gamma + (c_{\alpha} - c_\beta)\delta$,
$\beta' = \beta + c_\beta\delta$ by~\eqref{eq:aux2-thm99}.
This completes the proof of (2).

We shall now prove part (1). Consider the decomposition~\eqref{eq:form.of.increasing.word} for $\SL(l(\alpha) + k\delta)$
with $0\leq k< c_\alpha$. First, let us show that $p_1 = 0$ by assuming the contrary, i.e.\ $p_1 \geq 1$. To this end, let
$\SL(\alpha + k\delta) = \SL(\beta)\SL(\gamma)$ be the costandard factorization, and consider the following two cases:
\begin{itemize}[leftmargin=0.7cm]

\item
if $c_\alpha = 1$ (which forces $k=0$), then $|\beta| \geq |l(\beta) + \delta|$ by the induction assumption, and so
$|l(\alpha)| \geq |l(\beta)| + |\delta| + |l_i(\gamma)| > |l_i(\beta)| + |l_i(\gamma)|$ contradicts Lemma~\ref{lemma:recursive.l};

\item
if $c_\alpha > 1$, then $m_1(\gamma) = m_1(\beta) = m_1(\alpha)$ and $\SL(\beta)$ has $y_i\SL_j(\delta)$ as a prefix
by Proposition~\ref{prop:cost.fac.inc}, and moreover $\SL(\gamma)$ also has $y_i\SL_j(\delta)$ as a prefix, due to
Proposition~\ref{prop:cost.fac.inc} and Lemma~\ref{lemma:inc.upper.bound} as already argued several times before.
Thus $|\beta| \geq |l(\beta) + c_\beta\delta|, |\gamma| \geq |l(\gamma) + c_\gamma\delta|$
by the inductive hypothesis, and so $|l(\alpha)| > |\beta - c_\beta\delta| + |\gamma - c_\gamma\delta| \geq |l(\beta)| + |l(\gamma)|$
(the first inequality is due to $|l(\alpha) + k\delta| = |\beta + \gamma|$ and $k < c_\alpha = c_\beta + c_\gamma$),
which contradicts Lemma~\ref{lemma:recursive.l}.

\end{itemize}
This completes the proof of $p_1=0$. Next, let us show that $p_t \leq 1$ for $2\leq t \leq c_\alpha$.
This is vacuous for $c_\alpha = 1$, hence let $c_\alpha > 1$. Assuming the contrary, let $t$ be
the smallest number such that $p_t \geq 2$. Choosing $\beta,\gamma$ as before, we consider two cases:
\begin{itemize}[leftmargin=0.7cm]

\item
if $y_i\underbrace{\SL_j(\delta)}_{p_t \text{ times}}w_t$ is in $\SL(\beta)$, then $|\beta| \geq |l(\beta) + c_\beta\delta|$
by the induction assumption. Also $y_i\SL_j(\delta)$ is a prefix of $\SL(\gamma)$
(as $y_i\SL_j(\delta) < \SL(\beta) < \SL(\gamma) < y_i\underbrace{\SL_j(\delta)}_{ \gg 0 \text{ times}}$
by Lemma~\ref{lemma:inc.upper.bound}), and so $|\gamma| \geq |l(\gamma) + c_\gamma\delta|$ by the inductive hypothesis.
Therefore,  $|l(\alpha)| > |\beta - c_{\beta}\delta| + |\gamma - c_\gamma\delta| \geq |l(\beta)| + |l(\gamma)|$,
contradicting Lemma~\ref{lemma:recursive.l};

\item
if $y_i \underbrace{\SL_j(\delta)}_{p_t \text{ times}}w_t$ is in $\SL(\gamma)$, then by the inductive assumption
combined with Lemma~\ref{lemma:p.i.inc}, we have $|\gamma| > |l(\gamma) + c_\gamma\delta|$.
We now consider the appropriate concatenation of $\SL(\beta + \delta)$ and $\SL(\gamma - \delta)$:
\begin{itemize}[leftmargin=0.7cm]

\item
if $\SL(\beta + \delta) < \SL(\gamma - \delta)$, then clearly $\SL(\beta + \delta)\SL(\gamma - \delta) > \SL(\beta)\SL(\gamma)$
contradicting Leclerc's algorithm, unless $\SL(\beta)$ is a prefix of $\SL(\beta + \delta)$. In the latter case, we have
$\SL(\beta + \delta) = \SL(\beta)w$ for some word $w$ with $|w| = |\delta|$, and we claim that still $\SL(\gamma) < w$
(so that $\SL(\beta + \delta)\SL(\gamma - \delta) > \SL(\beta + \delta) = \SL(\beta)w > \SL(\beta)\SL(\gamma)$ again).
According to Lemma~\ref{lemma:delta.length.suffix.l}, we have $w > y_i\SL_j(\delta)$ and $w$ cannot have $y_i\SL_j(\delta)$
as a prefix since $|w|<|y_i\SL_j(\delta)|$. By the inductive hypothesis $\SL(\gamma)$ has $y_i\SL_j(\delta)$ as a prefix,
and therefore $w > \SL(\gamma)$, as claimed.

\item
if $\SL(\beta + \delta) > \SL(\gamma - \delta)$, then $\SL(\gamma - \delta)\SL(\beta + \delta) > \SL(\beta)\SL(\gamma)$ since
by the inductive assumption $\SL(\gamma-\delta)$ has $y_i\SL_j(\delta)$ as a prefix as $|\gamma| > |l(\gamma) + c_\gamma\delta|$
while the former does not (as $p_1 = 0$), contradicting to Leclerc's algorithm.

\end{itemize}

\end{itemize}
We note that we must have one of the preceding two cases by Lemma~\ref{lemma:stronger.no.splitting}(2), since
$y_i\underbrace{\SL_j(\delta)}_{p_t \text{ times}}w_t$ is a Lyndon subword of $\SL(l(\alpha) + k\delta)$ by~\eqref{eq:chunk-incr},
which cannot be a prefix (as $p_1=0$). This completes the proof of part (1) for $q = 0$.

As with part~(2), it suffices to verify (1) for $\hat\alpha, \hat\alpha - c_\alpha\delta$ with $\hat\alpha \in \chain(\alpha)$
satisfying $|\hat\alpha| \geq |l(\alpha) + c_\alpha\delta|$. Let $\SL(\hat\alpha - c_\alpha\delta) = \SL(\beta)\SL(\gamma)$ be
the costandard factorization, so that $\SL(\hat\alpha) = \SL(\beta + c_\beta\delta)\SL(\gamma + (c_\alpha - c_\beta)\delta)$
by part~(2). If $c_\alpha = 1$, then $c_\alpha-c_\beta=0$ and the result follows from the inductive assumption for $\SL(\beta)$.
If $c_\alpha > 1$, then $m_1(\beta)=m_1(\gamma)=m_1(\alpha)$, and the result follows from the inductive assumption for both
$\SL(\beta),\SL(\gamma)$.
\end{proof}

As in the case of decreasing chains (see Definition~\ref{def:periodicity-D}), the structure of increasing chains from
Theorem~\ref{thm:incr.modulo.c} motivates the following:

\begin{definition}\label{def:periodicity-I}
We shall call $c_\alpha$ the \textbf{periodicity} of the chain $\chain(\alpha)\in \mathcal{I}$.
\end{definition}


\subsection{Corollaries and periodicity bounds}
\

In this subsection, we discuss some natural corollaries of the above theorem, and derive precise bounds on $c_\alpha$.
We start with the counterpart of Remark~\ref{rem:cost.fact.structure}.

\begin{remark}\label{rem:cost.fact.structure-incr}
We note that Proposition~\ref{prop:cost.fac.inc} implies that we can iteratively take the costandard factorization of
$u=\SL(l(\alpha) + k\delta)$ for $k \in \BZ_{\geq 0}$ to split it into $c_\alpha$ pieces $u=u_1u_2\cdots u_{c_\alpha}$ of the form
$u_t=y_i\underbrace{\SL_j(\delta)}_{q_t \text{ times}}w_t$ satisfying $m_1(\deg(u_t))=m_1(\alpha)$ and $c_{\deg(u_t)} = 1$.
Furthermore, all $q_t \in \{\lfloor k/c_\alpha \rfloor, \lceil k/c_\alpha\rceil \}$ by part (1) of Theorem~\ref{thm:incr.modulo.c}.
\end{remark}

The next simple result, an analogue of Lemma~\ref{lem:aux_unused} for increasing chains,
describes the outcome of above procedure when $c_\alpha=1$.

\begin{corollary}\label{cor:costandard.l.i}
For any non-irreducible $\chain(\alpha) \in \mathcal{I}$ with $c_\alpha = 1$ and $|\alpha| \geq |l(\alpha)|$, consider the
costandard factorization $\SL(\alpha) = \SL(\beta)\SL(\gamma)$. Then $\gamma = l_i(\gamma)$, whereas $m_1(\alpha) = (\delta,i)$.
\end{corollary}

\begin{proof}
By Theorem~\ref{thm:incr.modulo.c}(2), it suffices to show this for $\alpha = l(\alpha)$, as
$\SL^r(l(\alpha))=\SL^r(l(\alpha)+t\delta)$ for all $t>0$. By Lemma~\ref{lemma:y.i.prefix}, $y_i$ is a prefix of $\SL(l(\alpha))$,
and so $y_i < \SL(l(\alpha)) < \SL(\gamma)$. Since $\chain(\gamma) \in \mathcal{I}$ by Proposition~\ref{prop:cost.fac.inc}(3),
we get $|\gamma| \geq |l_i(\gamma)|$. We also have $m_1(\gamma) > m_1(\alpha)$, due to Lemma~\ref{lemma:min.im.rule} as
$c_\alpha = 1$. If $|\gamma| > |l_i(\gamma)|$, then we would have $\SL(\gamma - \delta) > \SL(\beta)$ as $y_i < \SL(\gamma - \delta)$
and $y_i$ is a prefix of $\SL(\beta)$ but is not a prefix of $\SL(\gamma - \delta)$ by
Lemma~\ref{lemma:left.standard.prefix.increasing}. Then $\SL(l(\alpha) - \delta) \geq \SL(\beta)\SL(\gamma - \delta) > y_i$
by Leclerc's algorithm, contradicting to the definition of $l(\alpha)$. Thus $\gamma = l_i(\gamma)$.
\end{proof}

We then have the following analogous statement to Lemma~\ref{lemma:dec.chains.canon.fac}.

\begin{lemma}\label{lemma:inc.chains.canon.fac}
For any $\chain(\alpha) \in \mathcal{I}$, set $m_1(\alpha) = (\delta,i)$ and let $w$ be one of $w_\bullet$'s
in~\eqref{eq:inc.chains.form}. For the canonical factorization $w = v_1\ldots v_n$ and all $1\leq j\leq n$, we have:
\begin{enumerate}[leftmargin=0.8cm]

\item
$\chain(\deg(v_j)) \in \mathcal{I}$;

\item
$l_i(\deg(v_j)) = \deg(v_j)$;

\item
$m_1(\deg(v_j)) > (\delta,i)$.

\end{enumerate}
\end{lemma}

\begin{proof}
By~\eqref{eq:chunk-incr} it suffices to prove this when $c_\alpha = 1$.
Moreover, by Theorem~\ref{thm:incr.modulo.c}, we can further assume that $k=0$.
Then $v_n = \SL^r(l(\alpha))$ by Proposition~\ref{prop:cost.fac.inc}(4), and let $\beta = \deg(\SL^l(l(\alpha)))$.
We note that it suffices to verify all three properties for $v_n$, as the general case then follows by
applying the same argument towards $\SL(\beta)$.

The word $v_n$ satisfies (1) due to Proposition~\ref{prop:cost.fac.inc}, (2) due to Corollary~\ref{cor:costandard.l.i},
and (3) due to the assumption $c_\alpha=1$.
\end{proof}

We note however, that in contrast to Corollary~\ref{cor:dec.w.form}, $\deg(w_i)$ from~\eqref{eq:inc.chains.form} are not
always roots, as illustrated by the following example.

\begin{example}\label{ex:incr_nonroot_counterex}
Consider the affine type $C_2^{(1)}$ with the order $0 < 1 < 2$.
Then using the code in Listing~\ref{lst:C2.example}, we find:
\begin{align*}
  \SL(\alpha_0 + \delta) &= 01012, \\
  \SL(\alpha_0 + k\delta) &= 011\underbrace{\SL_1(\delta)}_{k -2 \text{ times}}012012 \qquad \forall\, k\geq 2.
\end{align*}
Here, we note that
  $m_1(\alpha_0) = (\delta,2)$, $y_2 = 011$, $M_1(\beta_2)=(\delta,1)$, $\SL_1(\delta) = 0121$,
$w_1 = 012012$. In particular, $\deg(w_1)=\delta+\alpha_0+\alpha_2$ is not a root.
\end{example}

In analogy with Corollary~\ref{cor:max.periodicities}, we shall now obtain a tight upper bound on the periodicity
$c_\alpha$ for $\chain(\alpha) \in \mathcal{I}$. This upper bound will be provided by the following function on
irreducible root systems, defined by induction on their rank.

\begin{definition}
For a reduced irreducible root system $\Delta$, we define $v(\Delta)=1$ if $\rank(\Delta)=1$, and otherwise set
\begin{equation*}
  v(\Delta) =
  \max \Big( \{l(\Delta)\} \cup \{v(\Delta') \,|\, \Delta' \text{ is a connected root subsystem of }\Delta\} \Big)
\end{equation*}
where $l(\Delta)$ denotes the maximum multiplicity of a simple root in $\theta$, which is a leaf of the Dynkin diagram of $\fg$.
\end{definition}

This function provides an upper bound on the periodicity as shown below.

\begin{lemma}
For any $\chain(\alpha) \in \mathcal{I}$, we have $c_\alpha \leq v(\Delta)$.
\end{lemma}

\begin{proof}
We will show this by induction on rank $|I|=\rank(\Delta)$. The base case $\Delta = A_1$ is clear.
For the step of induction, let us now assume that the result holds for all root systems of rank $< \rank(\Delta)$.
If $m_1(\alpha) > (\delta,|I|)$, then $\alpha \in \Delta_{|I|-1}$, see Definition~\ref{def:root_subsystems}.
As $\Delta_{|I|-1}$ is connected by Conjecture~\ref{conj:connectivity},
we have $c_\alpha \leq v(\Delta_{|I|-1}) \leq v(\Delta)$ by the inductive hypothesis. If $m_1(\alpha) = (\delta,|I|)$,
then $c_\alpha=c_{|I|}$, whereas $\chain(\alpha)=c_1\chain(\beta_1)+\cdots+c_{|I|}\chain(\beta_{|I|})$.
But $\beta_{|I|}$ must be a leaf of the Dynkin diagram of $\mathfrak{g}$ by the connectivity of
Conjecture~\ref{conj:connectivity}, and so $c_\alpha \leq v(\Delta)$.
\end{proof}

A detailed case-by-case treatment of all finite root systems yields the following.

\begin{corollary}\label{cor:max.period-increasing}
We have the following upper bounds for the periodicity of all increasing chains for any untwisted affine root system:
\begin{center}
\begin{tabular}{|c|c|}
    \hline
         type & max periodicity \\
    \hline
         $A_n^{(1)},D_n^{(1)}$ & 1 \\
    \hline
         $B_n^{(1)}$,$C_n^{(1)}$ & 2 \\
    \hline
         $E_6^{(1)}$ & 2 \\
    \hline
         $E_7^{(1)}$ & 2 \\
    \hline
         $E_8^{(1)}$ & 3 \\
    \hline
         $F_4^{(1)}$ & 2 \\
    \hline
         $G_2^{(1)}$ & 3 \\
    \hline
\end{tabular}
\captionof{table}{Upper bound on periodicity of increasing chains}\label{table:inc.periodicity}
\end{center}
\end{corollary}

In fact, these bounds are tight as shown in the next result.

\begin{proposition}\label{prop:c.bounds}
The bounds in Table~\ref{table:inc.periodicity} are tight. Moreover, any value from $1$ up to the specified bound can be achieved.
\end{proposition}

\begin{proof}
First, for each type we shall present a specific order on $\wI$ such that the periodicity $c_{\alpha_0}$ precisely coincides
with the corresponding value in Table~\ref{table:inc.periodicity}. We note that the coefficients $c_i$ entering into the
decomposition $\chain(\alpha_0)=\sum_{i=1}^{|I|} c_i\chain(\beta_i)$ are precisely the labels of the corresponding extended
Dynkin diagram, as indicated in~\cite[\S4, Table Aff 1]{K}, whereas $\beta_j$ matches with $\delta-\alpha_i$ by
Corollary~\ref{cor:simple-0}. For any order $0<k<\cdots$ on $\wI$, we note that $\beta_1 = \delta - \alpha_k$ by
Lemma~\ref{lemma:smallest.once.SL.1} and $\beta_{|I|}$ corresponds to one of the leafs of the corresponding
Dynkin diagram by Conjecture~\ref{conj:connectivity}.

The analysis for classical types is now straightforward:
\begin{itemize}[leftmargin=0.7cm]

\item
for $A_n^{(1)}$-type with the orders $0 < 1 < \cdots$ on $\wI$, we have $\beta_1=\delta-\alpha_1$ and so
$\beta_n$ corresponds to the only other leaf with label $1$, so that $c_{\alpha_0}=1$;

\item
for $D_n^{(1)}$-type with the orders $0 < 1 < \cdots$ on $\wI$, we have $\beta_1=\delta-\alpha_1$ and so
$\beta_n$ corresponds to one of the other two leaves with labels $1$, so that $c_{\alpha_0}=1$;

\item
for $B_n^{(1)}$-type with the orders $0 < 1 < \cdots$ on $\wI$, we have $\beta_1 = \delta - \alpha_1$ and so
$\beta_n$ corresponds to the only other leaf with the label $2$, so that $c_{\alpha_0}=2$;

\item
for $C_n^{(1)}$-type with the orders $0 < n < \cdots$ on $\wI$, we have $\beta_1 = \delta - \alpha_n$ and so
$\beta_n$ corresponds to the only other leaf with the label $2$, so that $c_{\alpha_0}=2$.

\end{itemize}
For exceptional types, we used the code in Listing~\ref{lst:c.bound} to find the desired orders:
\begin{itemize}[leftmargin=0.7cm]

\item
for $E_6^{(1)}$-type and the order $0<1<4<2<5<3<6$ on $\wI$, we have $\beta_6=\delta-\alpha_{6}$
so that $c_{\alpha_0}=2$;

\item
for $E_7^{(1)}$-type and the order $0<1<5<2<6<4<3<7$, we have $\beta_7=\delta-\alpha_{7}$
so that $c_{\alpha_0}=2$;

\item
for $E_8^{(1)}$-type and the order $0<1<6<2<3<4<7<5<8$, we have $\beta_8=\delta-\alpha_{8}$
so that $c_{\alpha_0}=3$;

\item
for $F_4^{(1)}$-type and the order $0<1<4<2<3$, we have $\beta_4=\delta-\alpha_{4}$
so that $c_{\alpha_0}=2$;

\item
for $G_2^{(1)}$-type and the order $0<1<2$, we have $\beta_2=\delta-\alpha_2$ so that $c_{\alpha_0}=3$.

\end{itemize}

Pick the order on $\wI$ as above, so that $\chain(\alpha_0)$ has periodicity $c_{\alpha_0}$ as in Table~\ref{table:inc.periodicity}.
For any $1\leq c < c_{\alpha_0}$, we shall apply the same logic as in the proof of Proposition~\ref{lem:precise-bound} to provide
$\chain(\beta)\in \mathcal{I}$ with $c_\beta=c$. Explicitly, consider
a sequence of chains $(\chain(\eta_i))_{i=1}^{N}$ such that $\chain(\eta_{i+1})-\chain(\eta_i)$ is an irreducible increasing
chain for all $1\leq i\leq N$, whereas $\eta_0=0$, and $\chain(\eta_{N}) = \chain(\alpha_0)$. Choose the minimal $k \in [1,N]$
such that $m_1(\eta_k) = m_1(\alpha_0)$, hence $c_{\eta_k} = 1$. Then $m_1(\eta_{i}) = m_1(\alpha_0)$ for all $k\leq i\leq N$
by Corollary~\ref{cor:both.increasing.m.k}, and furthermore $c_{\eta_{i+1}} - c_{\eta_i} \in \{0,1\}$.
As $c_{\eta_k} = 1$ and $c_{\eta_N}=c_{\alpha_0}$, we thus have $c_{\eta_a}=c$ for some $k\leq a < N$.
Therefore, $\chain(\eta_a)\in \mathcal{I}$ is as needed.
\end{proof}

The next two results establish a relation of the present work to the study of imaginary $\SL$-words in~\cite[\S6]{ET}.
We start with the following important observation.

\begin{corollary}\label{cor:left.imaginary.chain.c}
For any $k\geq 1$ and $1\leq i\leq |I|$, we have $c_{\deg(\SL^{ls}_i(k\delta))} = 1$.
\end{corollary}

\begin{proof}
Let $\alpha = \deg(\SL^{ls}_i(k\delta))$, so that $m_k(\alpha)=(k\delta,i)$ by Corollary~\ref{cor:mk.im.factor}.
Then $\chain(\alpha)=c_1\chain(\beta_1) + c_2\chain(\beta_2) + \cdots + c_i\chain(\beta_i)$ with $c_i > 0$ by
Corollary~\ref{cor:mk.roots}. Assume the contradiction, that is $c_i > 1$. By Lemma~\ref{lemma:c.splitting}, then
there exist $\chain(\beta),\chain(\gamma) \in \mathcal{I}$ such that $\chain(\beta)+\chain(\gamma) = \chain(\alpha)$
and $m_1(\beta) = m_1(\gamma) = m_1(\alpha)$. Assume first that there exist $\hat\beta \in \chain(\beta)$ and
$\hat\gamma \in \chain(\gamma)$ such that $\alpha = \hat\beta + \hat\gamma$. Without loss of generality, we may
assume $\hat\beta < \hat\gamma$, so that $\hat\beta < \alpha < \hat\gamma$ by Theorem~\ref{thm:convexity}. But,
$\hat\gamma < (k\delta,i)$ by Proposition~\ref{prop:monotonicity}, hence $\hat\gamma <\alpha$
by Lemma~\ref{lemma:equiv.to.standard.fac}, a contradiction.

If there are no $\hat\beta,\hat\gamma$ as above, then $\alpha + \delta = \beta' + \gamma'$. Without loss of generality,
we may assume $\beta' < \gamma'$, so that $\beta' < \alpha + \delta < \gamma'$ by Theorem~\ref{thm:convexity}. As
$\gamma' < (\delta,i)$ by Proposition~\ref{prop:monotonicity}, we obtain $\gamma' < \alpha < \alpha + \delta$ by
Lemma~\ref{lemma:equiv.to.standard.fac}, a contradiction.
\end{proof}

We can now establish a weaker version of~\cite[Conjecture 6.26]{ET}.

\begin{lemma}\label{lemma:y.i.im.periodicity}
For any $1 \leq i \leq |I|$, let $k$ be the smallest value such that $|y_i| < k|\delta|$ and $M_1(\beta_i) = (\delta,j)$.
Then $\SL^{ls}_i(k\delta) = y_iw$ for some non-empty word $w$, which is actually a prefix of $\SL_j(\delta)$, and
\begin{equation}\label{eq:ls-imaginary-specialorders}
  \SL_i^{ls}(k'\delta) = y_i\underbrace{\SL_j(\delta)}_{k' -k \text{ times}}w\qquad \forall\, k'\geq k.
\end{equation}
\end{lemma}

\begin{proof}
Consider $\alpha = \deg(\SL_i^{ls}(k\delta))$ and $\alpha' = \deg(\SL_i^{ls}(k'\delta))$.
First, we note that $|\alpha| \geq |l(\alpha)|$ and $|\alpha'| \geq |l(\alpha')|$ as $\SL(\alpha),\SL(\alpha') \geq y_i$ by
Lemma~\ref{lemma:im.left.standard.form}, hence $y_i$ is a prefix of $\SL(\alpha),\SL(\alpha')$ by Lemma~\ref{lemma:y.i.prefix}.
In particular, $\SL_i^{ls}(k\delta)=y_iw$ for some $w$, and to prove~\eqref{eq:ls-imaginary-specialorders}
it suffices to verify $\alpha' = \alpha + (k'-k)\delta$, due to Theorem~\ref{thm:incr.modulo.c}
and Corollary~\ref{cor:left.imaginary.chain.c}.
We note that $\alpha + (k'-k)\delta \leq \alpha'$ by Lemma~\ref{lemma:im.left.standard.form}.
As $y_i\underbrace{\SL_j(\delta)}_{k' - k \text{ times}}$ is a prefix of $\SL(\alpha + (k'-k)\delta) \leq \alpha'$, it must
also be a prefix of $\SL(\alpha')$, due to Lemma~\ref{lemma:inc.upper.bound}. Thus $|\alpha'| \geq |l(\alpha') + (k'-k)\delta|$
by Theorem~\ref{thm:incr.modulo.c}, since $c_{\alpha'} = 1$ by Corollary~\ref{cor:left.imaginary.chain.c}.
This estimate allows to apply Claim~\ref{claim:aux.incr} to conclude that if $\alpha + (k'-k)\delta < \alpha'$ then
$\alpha  < \alpha' - (k'-k)\delta$, which contradicts to Lemma~\ref{lemma:im.left.standard.form},
whereas $|\alpha'|>(k-1)|\delta|$ by Lemma~\ref{lemma:length.standfac.imaginary}.
This completes the proof of $\alpha + (k'-k)\delta = \alpha'$.

Finally, to see that $w$ is a prefix of $\SL_j(\delta)$, we note that
$\SL_i^{ls}(k\delta) = y_i w \leq \SL_i^{ls}((k+1)\delta) = y_i\SL_j(\delta)w < \SL_i(k\delta)$ with the first inequality
due to Lemma~\ref{lemma:im.left.standard.form} and the last inequality due to Proposition~\ref{prop:monotonicity}.
Therefore, $w$ is a prefix of $\SL_j(\delta)w$ and hence also of $\SL_j(\delta)$, since $|w| < |\delta|$ by our choice of $k$.
\end{proof}

We conclude this section with the following two general observations.

\begin{lemma}
If the smallest simple root $\alpha_\varepsilon$ occurs only once in $\delta$, then
\begin{equation*}
  \chain(\beta_i) = \chain(\deg(\SL_i^{ls}(\delta)) \qquad \forall\, 1 \leq i \leq |I|.
\end{equation*}
\end{lemma}

\begin{proof}
First, we note that $\SL^{rs}_i(\delta)$ is always a single letter and $\SL^{rs}_i(\delta)\ne \SL^{rs}_j(\delta)$ for $i\ne j$,
due to Lemma~\ref{lemma:delta.stand.one.letter}. Thus,
$\{\chain(\SL^{ls}_i(\delta))\}_{i=1}^{|I|}=\{\chain(\delta-\alpha_a)\}_{a\in \wI}^{a\ne \varepsilon}=\{\chain(\beta_{j})\}_{j=1}^{|I|}$,
due to an analogue of Corollary~\ref{cor:simple-0}. The result follows as
$\SL^{ls}_i(\delta)>\SL^{ls}_j(\delta) \Leftrightarrow \SL_i(\delta)>\SL_j(\delta)$, where we use the above structure
of $\SL^{ls}_{\imath}(\delta)$.
\end{proof}

Combining the above lemma with Corollary~\ref{cor:y.i.left.standard.im}, we obtain the following result.

\begin{corollary}\label{cor:y-for-specialorder}
If the smallest simple root $\alpha_\varepsilon$ occurs only once in $\delta$, then
\begin{equation*}
 y_i = \SL_i^{ls}(\delta) \qquad \forall\, 1 \leq i \leq |I|.
\end{equation*}
\end{corollary}

\begin{remark}
We note that the above Corollary~\ref{cor:y-for-specialorder} and Lemma~\ref{lemma:y.i.im.periodicity} provide
a new proof of~\cite[Theorem~6.14]{ET}.
\end{remark}


\appendix

\section{Code}\label{sec:app_code}

In this Appendix, we present the source code (written using Python):

$\qquad$ \href{https://github.com/corbyte/AffineStandardLyndonWords}{https://github.com/corbyte/AffineStandardLyndonWords}.

\noindent
Almost everything in the code is done through a rootSystem object and example of the initialization can be seen below.
For simplicity, we use $G_2$-type for Listings~\ref{lst:root.system}--\ref{lst:additional.functions}.

\begin{lstlisting}[language=python,caption=RootSystem initialization, label={lst:root.system}, captionpos=t,upquote]
"""
rootSystem(ordering,type:str):
Initialization of root system

ordering -- list of ordering for the rootsystem with ordering[0] < ordering[1] < and so on
type -- type of the rootsystem
"""

G2 = rootSystem([2,1,0],'G')

G2.delta
#[1,2,3]
G2.baseRoots
# Will return all roots in the root system with height \leq \delta
\end{lstlisting}

\medskip
\noindent
In addition to the rootSystem class another important class is the word class: with this class you can do comparison
and concatenation between words. The word class acts as a wrapper around a list of elements from the letter class:
\begin{lstlisting}[language=python,caption=Word class,captionpos=t,upquote]
#'b'<'a'<'c'
u #abc
v #bc
u < v
# False
u + v
# abcbc
print(u)
#a,b,c
u.no_commas()
#abc
\end{lstlisting}

\medskip
\noindent
Getting a standard Lyndon word for a given rootSystem and ordering is very quick, additionally one can quickly get
chains of standard Lyndon words:

\begin{lstlisting}[language=python,caption= Standard Lyndon words,captionpos=t,upquote]
root_system # any rootSystem object
l = root_system.SL(degree) #Where degree is an element of the root system


#l will be an array of word objects, if degree is real there will only be one, but if the degree is imaginary there will be several

chain = root_system.chain(degree)

#chain is a list of all currenly generated standard Lyndon words with degree, degree + k\delta

root_system.periodicity(degree)

#Returns the periodicity of ch(degree)
\end{lstlisting}

\medskip
\noindent
The notation to use for degree is that the $(i+1)$-th element of the degree list you want corresponds to the multiplicity
of $\alpha_i$ in that degree. Additionally, words can be quickly parsed into the ``chunk format'':

\begin{lstlisting}[language=python,caption=Block format,captionpos=t,upquote]
G2 = rootSystem([1,2,0],'G')

G2.SL(G2.delta*8 + [1,0,0])[0].no_commas()
#'1222101222102122210101222101222102122210101222102'

print(G2.to_chunk_format(G2.SL(G2.delta*8 + [1,0,0])[0]))
#[[1, 2], '2', [1, 1], '10', [1, 2], '2', [1, 1], '10', [1, 1], '2']
#[i,j] means that there is an \SL_i(\delta) j times in that spot
\end{lstlisting}

\medskip
\noindent
There are some additional functions which will give useful information about standard Lyndon words and degrees:
\begin{lstlisting}[language=python,caption=Additional functions, label={lst:additional.functions},  captionpos=t,upquote]
G2 = rootSystem([1,2,0],'G')
#get_monotonicity returns 1 if the chain is increasing and -1 if it is decreasing
G2.get_monotonicity([0,1,0])
#1

#rootSystem.M_k(degree) returns i where M_k(degree) = (k\delta,i)
G2.M_k([0,1,0])
#1

#rootSystem.M_prime_k(degree) returns i where M'_k(degree) = (k\delta,i)
G2.M_prime_k([0,1,0])
#1

#rootSystem.m_k(degree) returns i where m_k(degree) = (k\delta,i)
G2.m_k([0,1,0])
#2

#rootSystem.y_i(i) return y_i for the given i
G2.y_i(2)
#array([1, 2, 2], dtype=int64)

#rootSystem.u_i(a, i) returns u_i(a), if i is not given it return u(a)
G2.u_i([1,0,0])
# array([2, 2, 3])


#Additionally we can easily calculate s and c for given roots
G2.s([1,1,3])
#5
G2.c([0,1,0])
#2

#mod_delta(\alpha+k\delta) will return (\alpha,k\delta)

G2.mod_delta(G2.delta*5 + [0,1,0])
#(array([0, 1, 0]), 5)

#generate_up_to_delta(k) will generate all standard Lyndon words upto height n\delta, results will be cached

G2.generate_up_to_delta(5)

#get_decompositions(\alpha) will return all possible \beta,\gamma \in \wDelta^{+} such that \beta + \gamma = \alpha
G2.get_decompositions(G2.delta)

#You can also get the standard and costandard factorization of words

l = G2.SL(G2.delta)[1]

print(*[i.no_commas()for i in G2.costfac(l)],sep=',')
#2,21210
print(*[i.no_commas()for i in G2.standfac(l)],sep=',')
#22121,0
\end{lstlisting}

\medskip
\noindent
The following listing is used in Example~\ref{ex:D5}.

\begin{lstlisting}[language=python,caption=$D^{(1)}_5$ example,label={lst:D5},captionpos=t,upquote]
D5 = rootSystem([0,1,2,3,4,5],"D")
print(*[D5.to_chunk_format(x) for x in D5.get_chain(D5.delta -[1,0,0,0,0,0],5)],sep='\n')
#['1235432']
#[[1, 1], '2354312']
#[[1, 1], '2341', [1, 1], '235']
#[[1, 1], '1', [1, 1], '23543', [1, 1], '2']
#[[1, 2], '23543', [1, 1], '1', [1, 1], '2']
#[[1, 2], '234', [1, 1], '1', [1, 2], '235']

#This will get all decreasing irreducible chains, in order, i.e. the ith row will be -\beta_i
D5.delta - D5.irr_chains()
#array([[0, 1, 0, 0, 0, 0],
#       [0, 0, 1, 0, 0, 0],
#       [0, 0, 0, 1, 0, 0],
#       [0, 0, 0, 0, 1, 0],
#       [0, 0, 0, 0, 0, 1]])


#Prints the index of irreducible chains and it's respective M'_1(*)
print(*[(i[0]+1,D5.M_prime_k(i[1])) for i in zip(range(5), D5.irr_chains())],sep='\n')
#(1, 1)
#(2, 1)
#(3, 2)
#(4, 3)
#(5, 3)
\end{lstlisting}

\medskip
\noindent
The following Listing is used in Remark~\ref{rem:u_i} to find the upper bound of $|u_i(\alpha)|$ for a collection of root systems.

\begin{lstlisting}[language=python,caption=$u_i$ upper bound,label={lst:u_i},captionpos=t,upquote]
# rootSystems is a collection of rootSystems
m = 0
for sys in rootSystems:
    if(sys.max_u_i() > m):
        m = sys.max_u_i()
m

#For example, if we take rootSystems to be all for F4, this will return 4
\end{lstlisting}

\noindent
The following listing is used in Example~\ref{ex:strong_period_counterex}.

\begin{lstlisting}[language=python,caption=Decreasing periodicity counterexample,label={lst:dec.per.counter},captionpos=t,upquote]
F4 = rootSystem([0,2,4,1,3],'F')
print(*[F4.to_chunk_format(i) for i in F4.get_chain([0,1,1,2,2],10)], sep='\n')
#['233144']
#[[1, 1], '342314']
#[[1, 1], '34', [1, 1], '3421']
#[1, 1], '12', [1, 1], '34', [1, 1], '34']
#[[1, 1], '2', [1, 1], '34', [1, 1], '34', [1, 1], '1']
#[[1, 2], '34', [1, 1], '2', [1, 1], '34', [1, 1], '1']
#[[1, 2], '34', [1, 2], '34', [1, 1], '2', [1, 1], '1']
#[[1, 2], '1', [1, 1], '2', [1, 2], '34', [1, 2], '34']
#[[1, 2], '2', [1, 2], '34', [1, 2], '34', [1, 2], '1']
\end{lstlisting}

\medskip
\noindent
The following listing is used in Example~\ref{ex:incr_nonroot_counterex}.

\begin{lstlisting}[language=python,caption=$C_2^{(1)}$ example,label={lst:C2.example},captionpos=t,upquote]
C2 = rootSystem([0,1,2],'C')
print(*[C2.to_chunk_format(i) for i in C2.get_chain([1,0,0],5)], sep='\n')
#['0']
#['01012']
#['011012012']
#['011', [1, 1], '012012']
#['011', [1, 2], '012012']

\end{lstlisting}

\medskip
\noindent
The following listing is used in Proposition~\ref{prop:c.bounds}.

\begin{lstlisting}[language=python,caption=$c_\alpha$ bounds,label={lst:c.bound},captionpos=t,upquote]
E6 = rootSystem([0,1,4,2,5,3,6],'E')
print((E6.delta - E6.irr_chains())[5])
# \beta_6 = \delta - \alpha_6
print(E6.c([1,0,0,0,0,0,0]))
#\alpha_0 periodicity is 2

E7 = rootSystem([0,1,5,2,6,4,3,7],'E')
print((E7.delta - E7.irr_chains())[6])
# \beta_7 = \delta - \alpha_6
print(E7.c([1,0,0,0,0,0,0,0]))
#\alpha_0 periodicity is 2

E8 = rootSystem([0,1,6,2,3,4,7,5,8],'E')
print((E8.delta - E8.irr_chains())[7])
#\beta_8 = \delta - \alpha_8
print(E8.c([1,0,0,0,0,0,0,0,0]))
#\alpha_0 periodicity is 3

F4 = rootSystem([0,1,4,2,3],'F')
print((F4.delta - F4.irr_chains())[3])
#\beta_4 = \delta - \alpha_4
print(F4.c([1,0,0,0,0]))
#\alpha_0 periodicity is 2

G2 = rootSystem([0,1,2],'G')
print((G2.delta - G2.irr_chains())[1])
#\beta_2 = \delta - \alpha_2
print(G2.c([1,0,0]))
#\alpha_0 periodicity is 3

\end{lstlisting}



\begin{thebibliography}{XXX}

\bibitem[AT]{AT}
Y.~Avdieiev, A.~Tsymbaliuk,
  {\em Affine standard Lyndon words: A-type},
Int.\ Math.\ Res.\ Not.\ IMRN (2024), no.~21, 13488--13524.

\bibitem[C]{C}
R.~Carter,
  {\em Lie Algebras of finite and affine Type},
Cambridge University Press, Cambridge (2005), xviii+632pp.

\bibitem[ET]{ET}
C.~Elkins, A.~Tsymbaliuk,
  {\em Affine standard Lyndon words},
preprint, ar$\chi$iv:2505.15432.

\bibitem[K]{K}
V.~Kac,
  {\em Infinite dimensional Lie algebras},
Cambridge University Press, Cambridge (1990), xxii+400pp.

\bibitem[LR]{LR}
P.~Lalonde, A.~Ram,
  {\em Standard Lyndon bases of Lie algebras and enveloping algebras},
Trans.\ Amer.\ Math.\ Soc.\ {\bf 347} (1995), no.~5, 1821--1830.	

\bibitem[Le]{L}
B.~Leclerc,
  {\em Dual canonical bases, quantum shuffles and $q$-characters},
Math.\ Z.\ {\bf 246} (2004), no.~4, 691--732.	

\bibitem[Lo]{Lo}
M.~Lothaire,
   {\em Combinatorics of words},
Cambridge University Press, Cambridge (1997), xviii+238pp.

\bibitem[M]{M}
G.~Melan\c{c}on,
 {\em Combinatorics of Hall trees and Hall words},
J.\ Combinat.\ Theory A {\bf 59} (1992), no.~2, 285--308.

\bibitem[NT]{NT}
A.~Negu\c{t}, A.~Tsymbaliuk,
  {\em Quantum loop groups and shuffle algebras via Lyndon words},
Adv.\ Math.\ {\bf 439} (2024), Paper No.~109482, 69pp.

\bibitem[R]{R}
M.~Rosso,
  {\em Lyndon bases and the multiplicative formula for $R$-matrices},
unpublished preprint (2002).

\end{thebibliography}
\end{document}